\documentclass{amsart}
\usepackage{amssymb,euscript,amsmath, mathrsfs}
\usepackage[dvips]{graphicx}
\usepackage[dvips]{color}

\newcounter{ENUM}

\input xy
\xyoption{all}
\CompileMatrices

\def\<{\langle}
\def\>{\rangle}
\def\0{{{\bf 0}}}

\def\OO{{\mathcal O}}

\def\CI{{\mathcal I}}
\def\CK{{\mathcal K}}
\def\CA{{\mathcal A}}
\def\CP{{\mathcal P}}

\def\tkappa{{\tilde \kappa}}
\def\tsigma{{\tilde \sigma}}

\def\CC{{\mathbb C}}

\def\FF{{\mathbb F}}
\def\GG{{\mathbb G}}

\def\QQ{{\mathbb Q}}

\def\tx{{\tilde x}}

\def\tU{{\tilde U}}

\def\tV{{\tilde V}}
\def\tM{{\tilde M}}
\def\tP{{\tilde P}}
\def\tpi{{\tilde \pi}}
\def\tPi{{\tilde \Pi}}
\def\tchi{{\tilde \chi}}

\def\tTheta{{\tilde \Theta}}

\def\unif{{\varpi}}
\def\abs{\mid~\;~\mid}

\newcommand{\Ind}{\operatorname{Ind}}
\newcommand{\cInd}{\operatorname{c-Ind}}

\newcommand{\St}{\operatorname{St}}

\newcommand{\Rep}{\operatorname{Rep}}

\newcommand{\reg}{\mbox{\rm \tiny reg}}

\newcommand{\cusp}{\mbox{\rm \tiny cusp}}

\newcommand{\red}{\mbox{\rm \tiny red}}

\def\ZZ{{\mathbb Z}}

\def\I{{\mathcal I}}
\def\K{{\mathcal K}}

\def\i{{\mathfrak i}}

\def\Hom{\operatorname{Hom}}

\def\End{\operatorname{End}}
\def\Aut{\operatorname{Aut}}

\def\Gal{\operatorname{Gal}}
\def\GL{\operatorname{GL}}

\def\Spec{\operatorname{Spec}}

\def\red{\operatorname{red}}

\newcommand{\Id}{\operatorname{Id}}

\newcommand{\margh}[1]{}

\newtheorem{thm}{Theorem}[section]
\newtheorem{theorem}[thm]{Theorem}
\newtheorem{prop}[thm]{Proposition}
\newtheorem{proposition}[thm]{Proposition}
\newtheorem{lemma}[thm]{Lemma}

\newtheorem{corollary}[thm]{Corollary}

\theoremstyle{definition}

\newtheorem{definition}[thm]{Definition}

\newtheorem{example}[thm]{Example}

\newtheorem{remark}[thm]{Remark}

\numberwithin{equation}{section}

\begin{document}
\title[The Bernstein center]{The Bernstein center of the category of smooth $W(k)[\GL_n(F)]$-modules}
\author{David Helm}
\date{1-9-2012}

\begin{abstract}
We consider the category of smooth $W(k)[\GL_n(F)]$-modules, where $F$ is a $p$-adic field
and $k$ is an algebraically closed field of characteristic $\ell$ different from $p$.  We describe
a factorization of this category into blocks, and show that the center of each such block is
a reduced, $\ell$-torsion free, finite type $W(k)$-algebra.  Moreover, the $k$-points of the center
of a such a block are in bijection with the possible ``supercuspidal supports'' of the smooth
$k[\GL_n(F)]$-modules that lie in the block.  Finally, we describe a large explicit subalgebra of the center
of each block and give a description of the action of this algebra on the simple objects of the block, in
terms of the description of the classical ``characteristic zero'' Bernstein center of~\cite{BD}.
\end{abstract}

\maketitle

\section{Introduction}

The {\em center} of an abelian category is $\CA$ the endomorphism ring of the identity functor
of that category.  It is a commutative ring that acts naturally on every object of $\CA$,
a fact which often allows one to approach questions about $\CA$ from a module-theoretic point
of view.

One spectacular success of this approach originates with Bernstein-Deligne~\cite{BD},
who computed the centers of categories of smooth complex representations of $p$-adic algebraic
groups.  The center of such a category is called the Bernstein center.  Bernstein and
Deligne give a factorization of this category into blocks, known as Bernstein components,
as well as a simple and explicit description of the center of each block, which is
a finite type $\CC$-algebra.  In particular they showed that the $\CC$-points of 
the Bernstein center were in bijection with the supercuspidal supports of irreducible smooth complex representations.

The results of~\cite{BD} made it possible to give purely algebraic proofs of theorems
about smooth representations that previously could only be proven via deep results from
Fourier theory; Bushnell-Henniart's results about Whittaker models in~\cite{BH-whittaker} are an
example of this approach.

In recent years there has been considerable interest in studying smooth representations
over fields other than the complex numbers, or even over more general rings.  To apply
similar techniques in such a setting one needs to understand the centers of categories
of smooth representations over $\FF_{\ell}$ or $\ZZ_{\ell}$ (or even over $\ZZ$).
Some progress along these lines was made by Dat~\cite{dat-integral}; in particular he was able to
give an explicit description of the center of the category of smooth
representations of a $p$-adic algebraic group $G$ over $\ZZ_{\ell}$, for $\ell$ a
{\em banal} prime; that is, for $\ell$ prime to the order of $G(\FF_p)$.

More recently Paskunas~\cite{paskunas} has studied the center of a category of
representations of $\GL_2(\QQ_p)$ over $\FF_p$; his results allow him to characterize the
image of the Colmez functor.

We fix our attention on the category $\Rep_{W(k)}(\GL_n(F))$
of smooth representations of $\GL_n(F)$, where $F$ is a $p$-adic field,
over a ring of Witt vectors $W(k)$, for $k$ an algebraically closed field of characteristic
$\ell$ different from $p$.  We obtain a factorization of this category into blocks that parallels
the Bernstein decomposition over $\CC$.  This description is closely related to the decomposition
due to Vigneras~\cite{vig98} of the category of smooth representations of $\GL_n(F)$ over $k$;
in both decompositions the blocks are parameterized by inertial equivalence classes of pairs
$(L,\pi)$, where $L$ is a Levi subgroup of $\GL_n(F)$ and $\pi$ is an irreducible supercuspidal
representation of $L$ over $k$.  Indeed, our approach owes a substantial amount to the ideas
of Vigneras, and in particular her use of $G$-covers in an integral setting.  In spite of this
we do not rely on Vingeras's Bernstein decomposition over $k$~\cite{vig98}, and indeed our 
description of the Bernstein
decomposition over $W(k)$ implies the result of Vigneras over $k$.

More precisely, we introduce (Definition~\ref{def:mod ell support}) a notion
of {\em mod $\ell$ inertial supercuspidal support} for simple objects $\Pi$
in $\Rep_{W(k)}(\GL_n(F))$.  When $\Pi$ is killed by $\ell$ this is just the usual notion
of supercuspidal support (taken up to inertial equivalence), but this notion also makes
sense for those $\Pi$ on which multiplication by $\ell$ is an isomorphism.  For an inertial
equivalence class $(L,\pi)$ as above, we can then consider the full subcategory
$\Rep_{W(k)}(\GL_n(F))_{[L,\pi]}$ whose objects are those $\Pi$ such that every simple subquotient
of $\Pi$ has mod $\ell$ inertial supercuspidal support given by $(L,\pi)$.  Our first main
result, Theorem~\ref{thm:decomposition}, then
states that the blocks of $\Rep_{W(k)}(\GL_n(F))$ are precisely the subcategories
$\Rep_{W(k)}(\GL_n(F))_{[L,\pi]}$.

For each such $(L,\pi)$, let $A_{[L,\pi]}$ be the center of 
$\Rep_{W(k)}(\GL_n(F))_{[L,\pi]}$.  The main objective of this paper is to
understand the structure of $A_{[L,\pi]}$ and its action on the objects
of $\Rep_{W(k)}(\GL_n(F))_{[L,\pi]}$.  Our starting point is the Bernstein-Deligne
description of the analogous question over algebraically closed fields.  Indeed,
if $\overline{\CK}$ is the algebraic closure of the field of fractions of $W(k)$,
then tensoring with $\overline{\CK}$ defines a functor from $\Rep_{W(k)}(\GL_n(F))$
to $\Rep_{\overline{\CK}}(\GL_n(F))$; the image of this functor is a product of
blocks of $\Rep_{\overline{\CK}}(\GL_n(F))$.

More precisely, Bernstein and Deligne show that the blocks of $\Rep_{\overline{\CK}}(\GL_n(F))$
are indexed by inertial equivalence classes of pairs $(M,\tpi)$, where $\tpi$ is an
irreducible supercuspidal representation of $\GL_n(F)$ over $\overline{\CK}$.  If we
denote the corresponding block by $\Rep_{\overline{\CK}}(\GL_n(F))_{M,\tpi}$, then the
essential image of $\Rep_{\overline{\CK}}(\GL_n(F))_{[L,\pi]}$ under the functor that tensors
with $\overline{\CK}$ is the product of those $\Rep_{\overline{\CK}}(\GL_n(F))_{M,\tpi}$ 
whose objects, when considered
as $W(k)[\GL_n(F)]$-modules, lie in $\Rep_{W(k)}(\GL_n(F))_{[L,\pi]}$.  It is then
not hard to show (Proposition~\ref{prop:invert}) that in this case
one has an isomorphism of $\overline{\CK}$-algebras:
$$A_{[L,\pi]} \otimes \overline{\CK} \cong \prod_{M,\tpi} A_{M,\tpi},$$
where $A_{M,\tpi}$ is the center of $\Rep_{\overline{\CK}}(\GL_n(F))_{M,\tpi}$.
Since (for formal reasons) $A_{[L,\pi]}$ is $\ell$-torsion free, this shows that
$A_{[L,\pi]}$ embeds naturally as a $W(k)$-subalgebra of the product of the
$A_{M,\tpi}$.  The Bernstein-Deligne theory (which we summarize in section~\ref{sec:BD})
gives a completely explicit description of the latter.

Understanding the integral structure of $A_{[L,\pi]}$ is considerably more difficult;
our main tool for doing so is the construction of explicit projective objects
$\CP_{K,\tau}$ of $\Rep_{W(k)}(\GL_n(F))$, from maximal distinguished
cuspidal $k$-types $(K,\tau)$.  These projectives have a number of useful properties;
in particular their endomorphism rings are commutative, and can be explicitly computed
after tensoring with $\overline{\CK}$.  We carry out this computation in sections~\ref{sec:zero generic}
(in the depth zero case) and \ref{sec:generic} (for the general case).
Although the integral structure of the endomorphism ring is harder to describe,
we are able to construct an explicit subalgebra $C_{K,\tau}$ of the endomorphism ring
of $\CP_{K,\tau}$ in sections~\ref{sec:zero endomorphisms} and~\ref{sec:endomorphisms}, 
and describe the action of this subalgebra
explicitly.  We also show, in section~\ref{sec:fg}, that the action of $C_{K,\tau}$
makes $\CP_{K,\tau}$ into an {\em admissible} $C_{K,\tau}[\GL_n(F)]$-module.  

As a result of this admissibility, we are able to deduce Bernstein's second adjointness
for the group $\GL_n(F)$, a fact which allows us to build further projective modules out of
the $\CP_{K,\tau}$ via parabolic induction.  In this way, we construct, for each
block $\Rep_{W(k)}(\GL_n(F))_{[L,\pi]}$, a faithfully projective object $\CP_{[L,\pi]}$.

From the fact that $\CP_{[L,\pi]}$ is faithfully projective, we deduce
(see section~\ref{sec:hecke}) that an element of $A_{[L,\pi]} \otimes_{W(k)} \overline{\CK}$ lies in
$A_{[L,\pi]}$ if, and only if, it preserves $\CP_{[L,\pi]}$ in $\CP_{[L,\pi]} \otimes_{W(k)} \overline{\CK}$.
The module $\CP_{[L,\pi]}$ is a direct sum of modules built by parabolic induction from tensor products
of $\CP_{K,\tau}$.  We therefore have a wide array of endomorphisms of $\CP_{[L,\pi]}$ at our disposal, because
of our explicit understanding of the endomorphism algebras $C_{K,\tau}$.  In Section~\ref{sec:hecke}
we exploit this to construct a subalgebra $C_{[L,\pi]}$ of the endomorphism ring of $\CP_{[L,\pi]}$,
and we show that this subalgebra actually lies in $A_{[L,\pi]}$.  Moreover, we can show that
$\CP_{[L,\pi]}$ is admissible over $C_{[L,\pi]}$, and thus that $A_{[L,\pi]}$ is a finitely
generated $C_{[L,\pi]}$-module. 

A consequence of this finiteness is that an object $\Pi$ of
$\Rep_{W(k)}(\GL_n(F))_{[L,\pi]}$ is admissible as an $A_{[L,\pi]}(\GL_n(F))$-module if, and
only if, it is finitely generated as a $W(k)[\GL_n(F)]$-module.  This is an integral analogue
of a result of Bernstein over fields of characteristic zero.

Our construction of $\CP_{[L,\pi]}$ also allows us to prove results describing the interaction
between the ring $A_{[L,\pi]}$ and parabolic induction (we refer the reader to 
Theorem~\ref{thm:integral Bernstein parabolic induction} for a precise statement).  This also
allows us to show that $A_{[L,\pi]}$ admits a tensor factorization as a tensor product
of rings $A_{[L_i,\pi_i]}$, where each $(L_i,\pi_i)$ is a so-called {\em simple pair.} 
(Theorem~\ref{thm:tensor factorization}.)

The action of $C_{[L,\pi]}$ on irreducible representations of $\GL_n(F)$, both in characteristic
zero and in characteristic $\ell$, can be made completely explicit in terms of a choice
of certain ``compatible systems of cuspidals''.  (We refer the reader to Theorem~\ref{thm:compatibility}
and the discussion preceding it for a description of these systems.)
From this one can show that the action of $C_{[L,\pi]}$ on irreducible smooth $k$-representations
of $\GL_n(F)$ in $\Rep_{W(k)}(\GL_n(F))_{[L,\pi]}$ is sufficient to distinguish representations
with different supercuspidal supports (Proposition~\ref{prop:c action}).  On the other hand,
it is not hard to show that if two irreducible $k$-representations of $\GL_n(F)$ have the same
supercuspidal support, then $A_{[L,\pi]}$ acts on both of them via the same map $A_{[L,\pi]} \rightarrow k$.
In particular, the $k$-points of {\em both} $\Spec C_{[L,\pi]}$ and $\Spec A_{[L,\pi]}$ are in natural
bijection with the set of possible supercuspidal supports of irreducible $k$-representations of $\GL_n(F)$
in the block corresponding to $(L,\pi)$.  (Corollary~\ref{cor:points}.)  This gives a ``mod $\ell$'' analogue
of the corresponding result of Bernstein-Deligne for complex points of the classical
Bernstein center.  An immediate corollary is that the map $\Spec A_{[L,\pi]}/\ell \rightarrow \Spec C_{[L,\pi]}/\ell$
is a bijection on $k$-points; from this we can conclude that the inclusion of $C_{[L,\pi]}/\ell$ in
$A_{[L,\pi]}/\ell$ identifies $C_{[L,\pi]}/\ell$ with the reduced quotient $(A_{[L,\pi]}/\ell)^{\red}$.

In the final section, we prove some additional results about the structure of the
endomorphism rings of the projectives $\CP_{K,\tau}$, assuming a result (Theorem~\ref{thm:generic1})
whose proof we postpone to the sequel paper~\cite{bernstein3}.  Conditionally on this result,
we are able to realize $A_{[L,\pi]}$ as a subalgebra of a much simpler ring than the
endomorphism ring of $\CP_{[L,\pi]}$; this simpler ring can, in small cases, be given a completely
explicit description.  As a result we are able to give completely explicit descriptions
of $A_{[L,\pi]}$ in numerous situations where $n$ is small relative to $\ell$ and the order of $q$
modulo $\ell$.  (See the examples at the end of section~\ref{sec:explicit} for precise statements.)

This is the first paper in a four-part series.  The second paper, \cite{bernstein3},
applies the structure theory of $A_{[L,\pi]}$ developed here
to questions that arise from the
theory of Whittaker models, that were first studied over the complex numbers
in~\cite{BH-whittaker}.  We establish versions of several of these results that hold for smooth
representations of $\GL_n(F)$ over $W(k)$.  These results have implications for
the structure theory of certain representations associated to the ``local Langlands
correspondence in families'' of~\cite{emerton-helm}.  In particular,~\cite{emerton-helm}
conjectures the existence of certain algebraic families of admissible representations of $\GL_n(F)$
attached to Galois representations.  These families are characterized by certain properties
that can be understood in terms of the spaces of Whittaker functions we consider.  We make
use of this in~\cite{bernstein3} to reduce the question of the existence of such
families to a natural conjecture that relates $A_{[L,\pi]}$ to the deformation theory
of Galois representations, via the local Langlands correspondence.

The third and fourth papers in the series,~\cite{curtis} and~\cite{converse}, will be devoted
to exploring this relationship between $A_{[L,\pi]}$ and Galois theory.
In particular we construct a natural isomorphism between 
the completion of $A_{[L,\pi]}$ at a point $x$ as a
subalgebra of the universal framed deformation ring of the semisimple representation of
$G_F$ attached to $x$ via local Langlands.  
In particular, this will yield a proof of the existence
of the families conjectured in~\cite{emerton-helm}.

The author is grateful to Matthew Emerton, Richard Taylor, Sug-Woo Shin, Florian Herzig,
David Paige and J.-F. Dat for helpful conversations and encouragement on the subject of
this note, and also to an anonymous referee who suggested many improvements to the exposition,
and a strengthening of Theorem~\ref{thm:tensor factorization}.  
This paper was partially supported by NSF grant DMS-1161582, and partially by EPSRC grant EP/M029719/1.

Throughout this paper we work with the base ring $W(k)$.  We thus
adopt the convention that all tensor products are taken over
$W(k)$ unless otherwise specified.

\section{Faithfully projective modules and the Bernstein center} \label{sec:faithful}

\begin{definition} Let $\CA$ be an abelian category.  The {\em center} of
$\CA$ is the ring of endomorphisms of the identity functor $\Id: \CA \rightarrow \CA$.
More prosaically, an element of $\CA$ is a choice of element
$f_{M} \in \End(M)$ for every object $M$ in $\CA$, satisfying the condition
$f_M \circ \phi = \phi \circ f_N$ for every morphism $\phi: N \rightarrow M$ in $\CA$.
\end{definition}

If $\CA$ is the category of right $R$-modules for some (not necessarily commutative)
ring $R$, then it is easy to see that the center of $\CA$ is the center
$Z(R)$ of the ring $R$.  Indeed, $Z(R)$ acts on every object of $\CA$; this defines
a map of $Z(R)$ into the center.  Its inverse is constructed by considering
the action of $Z(R)$ on $R$, considered as a right $R$-module.

When $\CA$ is a more general abelian category, we can often describe its 
center by reducing to the case of a module category.  We more or less follow the
ideas of~\cite{roche}, section 1.1.  The key is to find an object in
$A$ that is faithfully projective, in the following sense:

\begin{definition}
Let $\CA$ be an abelian category with direct sums.  An object $P$ in $\CA$ is
{\em faithfully projective} if:
\begin{enumerate}
\item $P$ is a projective object of $\CA$.
\item The functor $M \mapsto \Hom(P,M)$ is faithful.
\item $P$ is {\em small}; that is, one has an isomorphism:
$$\oplus_{i \in I} \Hom(P, M_i) \cong \Hom(P, \oplus_{i \in I} M_i)$$
for any family $M_i$ of objects of $\CA$ indexed by a set $I$.
\end{enumerate}
\end{definition}

One checks easily that the condition that $M \mapsto \Hom(P,M)$ is faithful is
equivalent to the condition that $\Hom(P,M)$ is nonzero for every object $M$ of $\CA$.
If $\CA$ has the property that every object of $\CA$ has a simple subquotient,
then it suffices to check that $\Hom(P,M)$ is nonzero for every simple $M$.

If $P$ is a faithfully projective object of $\CA$, one has:
\begin{proposition}[\cite{roche}, Theorem 1.1]
Let $P$ be faithfully projective.  The functor $M \mapsto \Hom(P,M)$ is
an equivalence of categories from $\CA$ to the category of right $\End(P)$-modules.
In particular, the center of $\CA$ is isomorphic to the center of $\End(P)$.
\end{proposition}

Idempotents of the center correspond to factorizations of $\CA$ as
a product of categories.  In practice we can obtain these factorizations by
constructing suitable injective objects of $\CA$.

\begin{proposition} \label{prop:decompose}
Suppose that every object of $\CA$ has a simple subquotient, let $S$ be a subset
of the simple objects of $\CA$, and let $I_1,I_2$ be a
injective objects of $\CA$ such that, up to isomorphism:
\begin{enumerate}
\item every simple subquotient of $I_1$ is in $S$,
\item every object in $S$ is a subobject of $I_1$,
\item no simple subquotient of $I_2$ is in $S$, and
\item every simple object of $\CA$ that is not in $S$ is a subobject of $I_2$.
\end{enumerate}
Then every object $M$ of $\CA$ splits canonically
as a product $M_1 \times M_2$, where every simple subquotient of $M_1$ is in $S$, and
and no simple subquotient of $M_2$ is in $S$.  
This gives a decomposition of $\CA$ as a product of the full subcategories $\CA_1$
and $\CA_2$ of $\CA$, where the objects of $\CA_1$ are those objects $M_1$ of $\CA$
such that every simple subquotient of $M_1$ is in $S$,
and the objects of $\CA_2$ are those objects $M_2$ of $\CA$ such that no
simple subquotient of $M_2$ is in $S$.  Moreover, every object of $\CA_1$ has an injective
resolution by direct sums of copies of $I_1$, and every object of $\CA_2$ has an injective
resolution by direct sums of copies of $I_2$.
\end{proposition}
\begin{proof}
Let $M_1$ be the maximal quotient of $M$ such that every simple subquotient of $M_1$ is in $S$,
and let $M_2$ be the kernel of the map $M \rightarrow M_1$.  We first show $\Hom(M_2,I_1) = 0$.
Suppose we have a nonzero map of $M_2$ into $I_1$, with kernel $N$.  Then the injection
of $M_2/N$ into $I_1$ would extend to an injection of $M/N$ into $I_1$, and thus
$M/N$ would be a quotient of $M$, dominating $M_1$, all of whose simple subquotients were in $S$.
It follows that no simple subquotient of $M_2$ lies in $S$, as such a subquotient would
yield a nonzero map of $M_2$ to $I_1$.

If we let $M_3$ be the maximal quotient of $M$ such that no simple subquotient of $M_3$
is in $S$, then the same argument (with $I_1$ and $I_2$ reversed) shows that every
simple subquotient of the kernel of the map $M \rightarrow M_3$ lies in $S$.  In particular
the projection of $M_2$ onto $M_3$ is injective.  Suppose the image of $M_2$ were not all
of $M_3$.  Then (as $M$ surjects onto $M_3$), there is a simple subquotient of $M_3$
that is also a subquotient of $M/M_2$; such an object would have to be in both $S$ and
its complement.  Thus $M_2$ is isomorphic to $M_3$, and hence $M$ splits, canonically,
as a product $M_1 \times M_2$.
The decomposition of $\CA$ as the product $\CA' \times \CA''$
is now immediate, as is the claim about resolutions.
\end{proof}

\begin{remark} It is easy to make a dual argument with projective objects; we state
the proposition in terms of injectives because that is the form of the proposition we
will use.
\end{remark}

\section{The Bernstein center of $\Rep_{\overline{\CK}}(G)$} \label{sec:BD}

Let $G = \GL_n(F)$ be a general linear group over a $p$-adic field $F$, and let
$k$ be an alebraically closed field of characteristic $\ell$ not equal to $p$.  Our goal
is to study the Bernstein center of the category $\Rep_{W(k)}(G)$ of
smooth $W(k)[G]$-modules.  It will be convenient for us to assume throughout that $\ell$ 
is odd, so that $W(k)$ necessarily contains a square root of $q$, where $q$ is the order
of the residue field of $F$.  This is done largely for convenience; when $\ell$ is $2$ all
the arguments we present remain valid after adjoining a square root of $q$ to $W(k)$.
We begin by studying the category $\Rep_{\overline{\CK}}(G)$
of smooth $\overline{\CK}[G]$-modules, where $\CK$ is the field of fractions of $W(k)$.  Most of
the results of this section are standard.  We limit ourselves to the case
of $\GL_n(F)$, although the results of this section have analogues for a general
reductive group.

The description of the center of $\Rep_{\overline{\CK}}(G)$ depends heavily on
the theory of parabolic induction, and particularly
the notions of cuspidal and supercuspidal support, which we now recall.
Let $(M,\pi)$ be an ordered pair consisting of a Levi subgroup $M$ of $G$
and an absolutely irreducible cuspidal representation $\pi$ of $M$.

Let $P$ be a parabolic subgroup of $G$, with Levi subgroup
$M$ and unipotent radical $U$, and let $\pi = \pi_1 \otimes ... \otimes \pi_r$ be
a $W(k)[M]$-module.  We let $i_P^G$ be the normalized parabolic induction functor of~\cite{BZ};
that is, $i_P^G \pi$ is the $W(k)[G]$-module obtained by extending $\pi$ by a trivial
$U$-action to a representation of $P$, twisting by a square root of the modulus character of $P$,
and inducing to $G$.  (This depends on a choice of square root of $q$ in $W(k)$; we fix such a
choice once and for all.)  Similarly, we denote by $r_G^P$ the parabolic restriction
functor from $W(k)[G]$-modules to $W(k)[M]$-modules.

\begin{definition}
Let $M$ be a Levi subgroup of $G$, and let $\pi$ be an absolutely irreducible supercuspidal
representation of $M$ over a field $L$.  Let $\Pi$ be an aboslutely irreducible
representation of $G$ over $L$.  Then the pair $(M,\pi)$ belongs to the
{\em supercuspidal support} of $\Pi$ if 
there exists a parabolic subgroup $P$ of $G$, with
Levi subgroup $M$, such that 
$\Pi$ is isomorphic to a Jordan--H\"older constituent
of the normalized parabolic induction $i_P^G \pi.$
\end{definition}

\begin{definition}
Let $M$ be a Levi subgroup of $G$, and let $\pi$ be a cuspidal
representation of $M$ over a field $L$.  Let $\Pi$ be an
absolutely irreducible of $G$ over $L$.  Then the pair $(M,\pi)$ belongs to the
{\em cuspidal support} if $\Pi$
if there exists a parabolic subgroup $P$ of $G$, with
Levi subgroup $M$, such that
$\Pi$ is isomorphic to a {\em quotient}
of $i_P^G \pi.$
\end{definition}

Over a field $L$ of characteristic zero, the notions of cuspidal and supercuspidal
support are equivalent, but the notions differ over fields of finite characteristic.

Two pairs $(M,\pi)$ and $(M',\pi')$ are conjugate in $G$ if there is an element
$g$ of $G$ that conjugates $M$ to $M'$ and $\pi$ to $\pi'$.  This determines
an equivalence relation on the set of pairs $(M,\pi)$.  It is then well-known that
the cuspidal and supercuspidal supprt of a given representation $\Pi$ are given
by a single conjugacy class of pairs.  For cuspidal support, this is an easy consequence
of Frobenius reciprocity, but for supercuspidal support this is a deep theorem, first
proven by Vigneras (\cite{vig98}, V.4).

\begin{definition}
We say that two representations $\pi$,$\pi'$ of $G$ that differ by a twist
by $\chi \circ \det$, where $\chi$ is an unramified character of $F^{\times}$,
are {\em inertially equivalent}.  More generally, if $M$ is a Levi subgroup of
$G$, two representations $\pi$ and $\pi'$ are inertially equivalent if they
differ by a twist by an unramified character $\chi$ of $M$, that is, a character
$\chi$ trivial on all compact open subgroups of $M$.
\end{definition}

We are primarily interested in cuspidal and supercuspidal support up to inertial equivalence.
Two pairs $(M,\pi)$ and $(M',\pi')$ are inertially equivalent if there is a representation
$\pi''$ of $M$, inertially equivalent to $\pi$, such that $(M,\pi'')$ is conjugate
to $(M',\pi')$.  The
{\em inertial supercuspidal support} (resp. {\em inertial
cuspidal support}) of an absolutely irreducible representation $\Pi$ of $G$ is the inertial
equivalence class of its supercuspidal support (resp. cuspidal support).

\begin{theorem}[Bernstein-Deligne, \cite{BD}, 2.13] \label{thm:B-D}
Let $M$ be a Levi subgroup of $G$, and let $\pi$ be an irreducible cuspidal
representation of $M$.  Let $\Rep_{\overline{\CK}}(G)_{M,\pi}$ be the full subcategory 
of $\Rep_{\overline{\CK}}(G)$ consisting of representations $\Pi$ such that
every simple subquotient of $\Pi$ has inertial supercuspidal support $(M,\pi)$.
Then $\Rep_{\overline{\CK}}(G)_{M,\pi}$ is a direct factor
of $\Rep_{\overline{\CK}}(G)$.
\end{theorem}

There is thus an idempotent $e_{M,\pi,\overline{\CK}}$ of the Bernstein center
of $\Rep_{\overline{\CK}}(G)$ that acts by the identity on all objects
of $\Rep_{\overline{\CK}}(G)_{M,\pi}$ and annihilates all of the other
Bernstein components.

Moreover, it is possible to give a complete description of the center $A_{M,\pi}$ of
$\Rep_{\overline{\CK}}(G)_{M,\pi}$.  Let $\Psi(M)$ denote the group of
unramified characters of $M$.  Then $\Psi(M)$ can be identified with the
algebraic torus $\Spec \overline{\CK}[M/M_0]$, where $M_0$ is the subgroup of $M$
generated by all compact open subgroups of $M$.  The group $\Psi(M)$ acts transitively
(by twisting) on the space of representations of $M$ inertially equivalent to $\pi$,
and the stabilizer of $\pi$ is a finite subgroup $H$ of $\Psi(M)$.  Note that $H$
depends only on the inertial equivalence class of $\pi$, not $\pi$ itself.  
The group $\Psi(M)/H$ is a torus, isomorphic to
$\Spec \overline{\CK}[M/M_0]^H$; a choice of $\pi$ identifies $\Psi(M)/H$
with the space of representations of $M$ inertially equivalent to $\pi$.

Let $W_M$ be the subgroup of the Weyl group $W(G)$ of $G$ (taken with respect to a maximal
torus contained in $M$) consisting of elements $w$ of $W(G)$ such that $w M w^{-1} = M$.
Define a subgroup $W_M(\pi)$ of $W_M$ consisting of all $w$ in $W_M$ such that
$\pi^w$ is inertially equivalent to $\pi$ (this subgroup depends only
on the inertial equivalence class of $\pi$.)  Then $W_M(\pi)$ acts on the space
of representations of $M$ inertially equivalent to $\pi$, and hence (via a choice of $\pi$)
on the torus $\Spec \overline{\CK}[M/M_0]^H$.  This action is in general a twist of the
usual (permutation) action of $W_M$ on $M/M_0$, but if $\pi$ is {\em invariant} under the action
of $W_M(\pi)$, then the action of $W_M(\pi)$ on $M/M_0$ is untwisted.

We have:

\begin{theorem}[Bernstein-Deligne] \label{thm:B-D presentation}
A choice of $\pi$ identifies $A_{M,\pi}$
with the ring $(\overline{\CK}[M/M_0]^H)^{W_M(\pi)}$.  More canonically,
the space of representations of $M$ inertially equivalent to $\pi$
is naturally a $\Psi/H$-torsor with an action of $W_M(\pi)$,
and the center of $\Rep_{\overline{\CK}}(G)$ is the ring
of $W_M(\pi)$-invariant regular functions on this torsor.
Moreover, if $f$ is an element of $A_{M,\pi}$, and $\Pi$ is an
object of $\Rep_{\overline{\CK}}(G)_{M,\pi}$ with supercuspidal
support $(M,\pi')$, then $f$ acts on $\Pi$ by the scalar $f(\pi')$.
\end{theorem}

We conclude with a standard result describing the action of the Bernstein center on
modules arising by parabolic induction.  Let $(M_i,\pi_i)$ be pairs consisting of
a Levi subgroup $M_i$ of $\GL_{n_i}(F)$, and an irreducible cuspidal representation $\pi_i$ of
$M_i$ such that $\pi_i$ is invariant under the action of $W_{M_i}(\pi_i)$.  Let
$M$ be the product of the $M_i$, considered as a subgroup of $\GL_n(F)$, where $n$
is the sum of the $n_i$.  Let $\pi$ be the tensor product of the $\pi_i$; it
is an irreducible cuspidal representation of $M$.  We then have an action of $W_M(\pi)$
on the inertial equivalence class of $(M,\pi)$.

In this setting, the group $(M/M_0)^H$ is the product of the groups $(M_i/(M_i)_0)^{H_i}$, where
$H_i$ is the subgroup of characters fixing $\pi_i$ under twist.  The isomorphism:
$$\overline{\CK}[M/M_0]^H \cong \bigotimes_i \overline{\CK}[M_i/(M_i)_0]^{H_i}$$
then restricts to give an embedding:
$$\Phi: (\overline{\CK}[M/M_0]^H)^{W_M(\pi)} \hookrightarrow \bigotimes_i (\overline{\CK}[M_i/(M_i)_0]^{H_i})^{W_{M_i}(\pi_i)}$$

\begin{proposition} \label{prop:Bernstein induction}
Let $\Pi_i$ be a collection of representations of $\GL_{n_i}$ such that for each $i$,
$\Pi_i$ lies in $\Rep_{\overline{\CK}}(\GL_{n_i}(F))_{M_i,\pi_i}$.  Let
$P = LU$ be a parabolic subgroup of $\GL_n(F)$, with $L$ isomorphic to the product of the
$\GL_{n_i}(F)$, and let $\Pi$ be the tensor product of the $\Pi_i$, considered as a representation of $L$.
Then $i_P^{\GL_n(F)} \Pi$ lies in $\Rep_{\overline{\CK}}(\GL_n(F))_{M,\pi}$.

Moreover, under the identifications
$$A_{M,\pi} \cong
(\overline{\CK}[M/M_0]^H)^{W_M(\pi)},$$ 
$$A_{M_i,\pi_i} \cong (\overline{\CK}[M_i/(M_i)_0]^{H_i})^{W_{M_i}(\pi_i)}$$
induced by $(M,\pi)$ and $(M_i,\pi_i)$,
if $x$ lies in $A_{M,\pi}$, then the endomorphism of $i_P^{\GL_n(F)} \Pi$ induced by $x$
coincides with the endomorphism of $i_P^{\GL_n(F)} \Pi$ arising from the action of $\Phi(x)$ on $\Pi$.
\end{proposition}

\section{Construction of projectives} \label{sec:projectives}
Our goal is to apply the theory of section~\ref{sec:faithful}
to the category $\Rep_{W(k)}(G)$ of smooth $W(k)[G]$-modules.  In particular we will factor this
category as a product of blocks, and construct an explicit faithfully projective module in 
each block.  The first step is to obtain a supply of suitable projective $W(k)[G]$-modules,
and study their properties.

\begin{definition} Let $R$ be a $W(k)$-algebra.  By an $R$-type of $G$, we mean a pair $(K,\tau)$, where
$K$ is a compact open subgroup of $G$ and $\tau$ is an $R[K]$-module that is finitely generated
as an $R$-module.  The Hecke algebra $H(G,K,\tau)$ is the ring $\End_{R[G]}(\cInd_K^G \tau)$.
\end{definition}

For the most part we will be concerned with $R$-types for $R=k$, or $R = \CK$, where
$\CK$ is the fraction field of $W(k)$.
Note that for any $R[G]$-module $\pi$, Frobenius reciprocity gives an isomorphism
of $\Hom_{R[K]}(\tau,\pi)$ with $\Hom_{R[G]}(\cInd_K^G \tau, \pi)$, and hence an
action of $H(G,K,\tau)$ on $\Hom_{R[K]}(\tau,\pi)$.  Moreover, if $V$ is the underlying
$R$-module of $\tau$, $H(G,K,\tau)$
can be identified with the convolution algebra of compactly supported smooth functions
$f: G \rightarrow \End_R(V)$ that are left and right $K$-invariant, in the sense that
$f(kgk') = \tau(k)f(g)\tau(k')$.  If $g$ is in $G$, then we denote by $I_g(\tau)$
the space $\Hom_{K \cap gKg^{-1}}(\tau, \tau^g)$, where $\tau^g$ is the representation
of $gKg^{-1}$ defined by $\tau^g(k) = \tau(g^{-1}kg)$.  Then the map $f \mapsto f(g)$
is an isomorphism:
$$H(G,K,\tau)_{KgK} \cong I_g(\tau),$$
where $H(G,K,\tau)_{KgK}$ is the space of functions in $H(G,K,\tau)$ supported
on $KgK$.

In this section, we will primarily be concerned with a certain class of types
which are called maximal distinguished cuspidal types in~\cite{vig98}, IV.3.1B.
We omit the precise definition of these types here; for our purposes it suffices
to know certain specific properties of a maximal distinguished cuspidal $R$-type $(K,\tau)$,
where $R$ is a field.

Such a type arises from a simple stratum $[{\mathfrak A}, n , 0, \beta]$,
together with a character $\theta$ in the set ${\mathcal C}({\mathfrak A}, 0, \beta)$
defined in~\cite{BK}, 3.2.1.  Here ${\mathfrak A}$ is a maximal order in $M_n(F)$,
and $\beta$ is an element of ${\mathfrak A}$ such that $E = F[\beta]$ is
a field.  This allows us to identify $E^{\times}$ with a subgroup of $\GL_n(F)$.
Let $e$ and $f$ denote the ramification index and residue class degree of $E$ over $F$.
One then has:

\begin{itemize}
\item $K$ is the group $J(\beta,{\mathfrak A})$ of~\cite{BK}.  In particular,
$K$ contains a normal pro-$p$ subgroup $K^1$ (called $J^1(\beta,{\mathfrak A})$ in~\cite{BK},) 
such that the quotient $K/K^1$
is isomorphic to $\GL_{\frac{n}{ef}}(\FF_{q^f})$, where $q$ is the order of
the residue field of $F$.
\item $\tau$ has the form $\kappa \otimes \sigma$, where $\sigma$ is a the inflation
of a cuspidal representation of $K/K^1$ over $R$, and $\kappa$ is a representation of $K$
that is a $\beta$-extension of the unique irreducible representation of $K^1$ containing $\theta$.
\end{itemize}

Maximal distinguished cuspidal $R$-types have the following useful properties:

\begin{theorem}[\cite{vig98}, IV.1.1-IV.1.3]
Let $R$ be a field, and let $(K,\tau)$ be a maximal distinguished cuspidal $R$-type arising from
an extension $E/F$.
\begin{enumerate} 
\item There is a unique embedding of $\GL_{\frac{n}{ef}}(E)$ into $G$ such that the center
$E^{\times}$ of $\GL_{\frac{n}{ef}}(E)$ normalizes
$K$ and $K^1$, and acts trivially on $K/K^1$.  We identify $\GL_{\frac{n}{ef}}(E)$ and
$E^{\times}$ with their images under
this embedding.  The intersection of $\GL_{\frac{n}{ef}}(E)$ with $K$ is $\GL_{\frac{n}{ef}}(\OO_E)$.
\item The subgroup $E^{\times}$ of $G$ normalizes $\tau$.  In particular $\tau$ extends
to a representation of $E^{\times}K$, and any two extensions of $\tau$ differ by
a twist by a character of $E^{\times}K/K \cong \ZZ$.
\item The $G$-intertwining of $(K,\tau)$ is equal to $E^{\times}K$.
\item For any extension $\hat \tau$ of $\tau$ to a representation of $E^{\times}K$,
there is an isomorphism of $H(G,K,\tau)$ with the polynomial ring $R[T,T^{-1}]$,
that sends $T$ to the unique element $f_{\hat \tau}$ of $H(G,K,\tau)$ such that
$f_{\hat \tau}$ is supported on $K\unif_E K$ (where $\unif_E$ is a uniformizer of $E$), and
$$f_{\hat \tau}(k\unif_Ek') = \tau(k) {\hat \tau}(\unif_E) \tau(k').$$
\item For any extension $\hat \tau$ of $\tau$ to a representation of $E^{\times} K$,
the representation $\cInd_{E^{\times}K}^G \hat \tau$ is an irreducible cuspidal representation
of $G$ over $R$.
\item Every irreducible cuspidal representation of $G$ over $R$ arises in this fashion.
Those irreducible cuspidal $\pi$ that arise from a given $(K,\tau)$ are precisely those
$\pi$ whose restriction to $K$ contains $\tau$.
\end{enumerate}
\end{theorem}

If $R$ is a field, and $\pi,\pi'$ are two absolutely
irreducible cuspidal $R$-representations containing 
a maximal distinguished cuspidal type $(K,\tau)$, then
$\Hom_{R[K]}(\tau,\pi)$ and $\Hom_{R[K]}(\tau,\pi')$ are modules over $H(G,K,\tau) = R[T,T^{-1}]$ that
are one-dimensional as $R$-vector spaces.  In particular $T$ acts via scalars $c$ and $c'$
on $\Hom_{R[K]}(\tau,\pi)$ and $\Hom_{R[K]}(\tau,\pi')$, respectively.  Let $\chi$ be an unramified
$k$-valued character of $F^{\times}$ such that $\chi(\unif_F)^{\frac{n}{e}} = c'c^{-1}$.
As $T$ is supported on $K\unif_E K$, and $\det \unif_E = \unif_F^{\frac{n}{e}}$, the
$H(G,K,\tau)$-modules $\Hom_{R[K]}(\tau,\pi \otimes \chi \circ \det)$ and
$\Hom_{R[K]}(\tau, \pi')$ are isomorphic, and so $\pi'$ is a twist of $\pi$ by
an unramified character; that is, $\pi$ and $\pi'$ are inertially equivalent.

\begin{remark} \label{rem:canonical}
\rm The isomorphism of $R[T,T^{-1}]$ with $H(G,K,\tau)$ depends on a choice of
extension $\hat \tau$ of $\tau$, and also a uniformizer $\unif_E$.  When $R$ is a field,
this isomorphism may be reinterpreted in the language Bernstein and Deligne
use to describe the Bernstein center, and made independent of $\unif_E$ (but not of $\hat \tau$).
Let $\pi$ be the irreducible cuspidal representation of $G$
whose restriction to $E^{\times}K$ contains $\hat \tau$.  Then every representation of $G$
inertially equivalent to $\pi$ has the form $\pi \otimes \chi$ for some unramified character $\chi$
of $G/G_0$, and we have $\pi = \pi \otimes \chi$ if, and only if, $\pi \otimes \chi$ contains $\hat \tau$.
The latter holds precisely when $\hat \tau = \hat \tau \otimes \chi|_{E^{\times}K}$, which holds
if and only if $\chi(\unif_E) = 1$.  Let $Z$ be the subgroup of $G$ generated by $\unif_E$;
our choice of $\unif_E$ identifies $R[T,T^{-1}]$ with $R[Z]$.  The map
$Z \rightarrow G/G_0$ is injective (but not in general surjective), and induces a map
$\Hom({G/G_0},\GG_m) \rightarrow \Hom(Z,\GG_m)$ by restriction.  Let $H$ be the kernel of this map;
the induced map on rings of regular functions then identifies $R[Z]$ with the $H$-invariants $R[G/G_0]^H$.
The identifications:
$$H(G,K,\tau) \cong R[T,T^{-1}] \cong R[Z] \cong R[G/G_0]^H$$
give an identification of $H(G,K,\tau)$ with $R[G/G_0]^H$ that does not depend on $\unif_E$.
The map $\chi \mapsto \pi \otimes \chi$ describes a bijection between $(\Spec R[G/G_0])/H$
and the set of representations of $G$ inertially equivalent to $\pi$.  Under the above isomorphisms,
the character of $H(G,K,\tau)$ that corresponds to a representation $\pi \otimes \chi$ of $G$
is the character of $R[G/G_0]^H$ obtained by treating $R[G/G_0]$ as the ring of regular functions
on the space of unramified characters of $G$ and evaluating such functions at $\chi$.
\end{remark}

Let $\CK'$ be a finite extension of the field of fractions $\CK$ of $W(k)$, and let $\OO$ be
its ring of integers.  If $\pi$ is an absolutely irreducible cuspidal
integral representation of $G$ over $\CK'$, then $\pi$ contains a unique homothety
class of $G$-stable $\OO$-lattices, and the reduction $r_{\ell} \pi$ of any such
lattice modulo $\ell$ is an absolutely irreducible cuspidal representation of $G$ over $k$.
In this situation we have the following compatibilities between the types
attached to $\pi$ and $r_{\ell} \pi$, due to Vigneras:

\begin{theorem}[\cite{vig98}, IV.1.5]
Let $\pi$ be an irreducible cuspidal representation of $G$ over $\CK'$,
containing a maximal distinguished cuspidal $\CK'$-type $(K,\tilde \tau)$.
\begin{enumerate}
\item There is an unramified character $\chi: F^{\times} \rightarrow (\CK')^{\times}$,
such that $\pi \otimes (\chi \circ \det)$ is integral.  
\item The $K$-representation $\tilde \tau$ is defined over $\CK$, and the mod $\ell$
reduction of $\tilde \tau$ is an irreducible $k$-representation $\tau$ of $K$,
such that $(K,\tau)$ is a maximal distinguished cuspidal $k$-type contained
in $r_{\ell} [\pi \otimes (\chi \circ \det)]$.  
\item Every maximal distinguished cuspidal $k$-type arises from a maximal distinguished 
cuspidal $\CK'$-type via ``reduction mod $\ell$'', for some finite extension $\CK'$
of $\CK$.
\end{enumerate}
\end{theorem}

We can use this theory of reduction mod $\ell$ to turn inertial equivalence into
an equivalence relation on simple cuspidal smooth $W(k)[G]$-modules $\pi$.  Such
$\pi$ fall into two classes: either $\ell$ annihilates $\pi$, in which case
$\pi$ is an irreducible cuspidal $k$-representation of $G$, or $\ell$ is invertible on
$\pi$, in which case $\pi$ is an absolutely irreducible cuspidal representation of $G$
over some finite extension $\CK'$ of $\CK$.

\begin{lemma}
Let $\pi$ be a simple smooth $W(k)[G]$-module on which $\ell$ is invertible,
and let $\CK'$ be a finite extension of $\CK$ such that every $\CK'[G]$-simple
subquotient of $\pi \otimes_{\CK} \CK'$ is absolutely simple.  Then $\pi \otimes_{\CK} \CK'$
is a direct sum of absolutely simple $\CK'[G]$-modules, and $\Gal(\overline{\CK}/\CK)$
acts transitively on these summands.
\end{lemma}
\begin{proof}
Let $\pi_0$ be an absolutely simple $\CK'[G]$-submodule of $\pi \otimes_{\CK} \CK'$.  Then
the sum of the submodules $\pi_0^g$ for $g$ in $\Gal(\overline{\CK}/\CK)$
is a Galois-stable $\CK'[G]$-submodule of $\pi \otimes_{\CK} \CK'$, and hence descends
to a $\CK[G]$-submodule of $\pi$.  This submodule must be all of $\pi$, and the result follows.
\end{proof}

\begin{definition}
Let $(K,\tau)$
be a maximal distinguished cuspidal $k$-type, and let $\pi$ be a simple cuspidal smooth
$W(k)[G]$-module.  We say that $\pi$ {\em belongs to the mod $\ell$ inertial equivalence class
determined by $(K,\tau)$} if either $\ell$ annihilates $\pi$ and $\pi$ contains $(K,\tau)$, or
if $\ell$ is invertible on $\pi$ and there exists a finite extension $\CK'$ of
$\CK$ such that one (equivalently, every) absolutely simple summand of $\pi \otimes_{\CK} \CK'$ 
is inertially equivalent to an integral
representation of $G$ over $\CK'$ whose mod $\ell$ reduction contains $(K,\tau)$.
\end{definition}

Fix a maximal distinguished cuspidal $k$-type $(K,\tau)$, with $\tau = \kappa \otimes \sigma$,
and let $\CP_{\sigma} \rightarrow \sigma$ be the projective envelope of $\sigma$ in the
category of $W(k)[GL_{\frac{n}{ef}}(\FF_{q^f})]$-modules.  We then have:

\begin{lemma}
The representation $\kappa$ lifts to a representation $\tkappa$ of $K$ over
$W(k)$.
\end{lemma}
\begin{proof}
As $K_1$ is a pro-$p$-group, the restriction $\kappa_1$ of $\kappa$ to $K_1$
lifts uniquely to a representation $\tkappa_1$ of $K_1$ over $W(k)$, normalized by
$K$.  The obstruction to extending $\tkappa_1$ to $K$ is thus an element of
$H^2(K,W(k)^{\times})$.  The first two paragraphs of~\cite{BK}, Proposition 5.2.4,
show that this element can be represented by a cocycle taking values in the $p$-power
roots of unity, and is thus a $p$-power torsion element $\alpha$ of
$H^2(K,W(k)^{\times})$.  (In~\cite{BK} the authors work over $\CC$ rather than $W(k)$, but their
argument adapts without difficulty.  Note that they denote by $\eta_M$ the representation
we call $\tkappa_1$, by $J_M$ the group we call $K$, and $J^1_M$ the group we call $K_1$.)

Let $U$ be a $p$-sylow subgroup of $K$ containing $K_1$.  The restriction of $\kappa$
to $U$ lifts uniquely to a representation over $W(k)$, extending $\tkappa_1$.  It
follows that the image of $\alpha$ in $H^2(U,W(k)^{\times})$ under restriction vanishes.
But the corestriction of this image to $H^2(K,W(k)^{\times})$ is equal to $r\alpha$,
where $r$ is the index of $U$ in $K$.  Thus $\alpha$ is killed by a power of $p$ and an integer
prime to $p$, and must therefore vanish.
\end{proof}

\begin{lemma}
The tensor product
$\tkappa \otimes \CP_{\sigma}$ is a projective envelope of $\kappa \otimes \sigma$
in the category of $W(k)[K]$-modules.
\end{lemma}
\begin{proof}
The restriction of $\kappa$ to $K_1$ is irreducible, and the restriction
of $\tkappa \otimes \CP_{\sigma}$ to $K_1$ is a direct sum of copies of $\tkappa_1$.
We thus have isomorphisms of $W(k)[K/K_1]$-modules:
$$\Hom_{K_1}(\tkappa,\tkappa \otimes \CP_{\sigma}) \cong \CP_{\sigma}.$$  
(Here $g \in K/K_1$ acts on $\Hom_{K_1}(\tkappa, \tkappa \otimes \CP_{\sigma})$ by
$f \mapsto f^{g}$, where $f^g(x) = gf(g^{-1}x)$; note that this action
depends on $\tkappa$, not just its restriction to $K_1$.)

Now suppose we have a surjection:
$$\theta' \rightarrow \theta$$
of $W(k)[K]$-modules.  We need to show that any map
$\tilde \kappa \otimes \CP_{\sigma} \rightarrow \theta$ lifts to a map to $\theta'$.
As we have identified $\CP_{\sigma}$ with $\Hom_{K_1}(\tkappa,\tkappa \otimes \CP_{\sigma})$,
such a map induces a map $\CP_{\sigma} \rightarrow \Hom_{K_1}(\tkappa, \theta)$.  This
latter map is $K/K_1$-equivariant.

As $K_1$ is a pro-$p$ group, the surjection of $\theta' \rightarrow \theta$
induces a surjection 
$$\Hom_{K_1}(\tkappa,\theta') \rightarrow \Hom_{K_1}(\tkappa,\theta)$$
of $W(k)[K/K_1]$-modules.  As $\CP_{\sigma}$ is projective, the map
$$\CP_{\sigma} \rightarrow \Hom_{K_1}(\tkappa,\theta)$$
lifts to a map
$$\CP_{\sigma} \rightarrow \Hom_{K_1}(\tkappa,\theta').$$
Tensoring with $\tkappa$, we obtain the desired map $\tkappa \otimes \CP_{\sigma} \rightarrow \theta'$,
so $\tkappa \otimes \CP_{\sigma}$ is projective.

On the other hand, $\tilde \kappa \otimes \CP_{\sigma}$ is indecomposable over $W(k)[K]$,
and is therefore a projective envelope of $\kappa \otimes \sigma$ in the category
of $W(k)[K]$-modules.  
\end{proof}

As the functor $\cInd_K^G$ is a left adjoint of an exact functor, it takes projectives to projectives.
In particular the module $\CP_{K,\tau}$ defined by $\CP_{K,\tau} := \cInd_K^G 
\tkappa \otimes \CP_{\sigma}$ is a projective object in $\Rep_{W(k)}(G)$.

\begin{proposition} \label{prop:cuspidal quotient}
Let $\pi$ be a simple cuspidal smooth $W(k)[G]$-module in the mod $\ell$
inertial equivalence class determined by $(K,\tau)$.  Then there exists
a surjection $\CP_{K,\tau} \rightarrow \pi$.
\end{proposition}
\begin{proof}
First suppose that $\ell$ annihilates $\pi$.
The surjection of $\CP_{\sigma}$ onto $\sigma$ gives rise to a surjection
$\CP_{K,\tau} \rightarrow \cInd_K^G \tau$.  As the restriction of $\pi$ to $K$
contains $\tau$, we have a nonzero (thus surjective) map $\cInd_K^G \tau \rightarrow \pi$ as claimed.

On the other hand, if $\ell$ is invertible in $\pi$, then fix an absolutely simple
summand $\pi_0$ of $\pi \otimes_{\CK} \CK'$ for some finite extension $\CK'$ of $\CK$.
As $\pi_0$ is in the mod $\ell$ inertial equivalence class determined by $(K,\tau)$,
there is a maximal distinguished
cuspidal $\CK'$-type $(K, \tilde \tau)$ contained in $\pi_0$; its mod $\ell$ reduction is $(K,\tau)$.  If we
regard $\tilde \tau$ as a representation of $K$ over the ring of integers
$\OO'$ of $\CK'$, we have surjections
$$\tilde \tau \rightarrow \tau$$
$$\tkappa \otimes \CP_{\sigma} \rightarrow \tau,$$
and thus obtain a map $\tkappa \otimes \CP_{\sigma} \rightarrow \tilde \tau$ by projectivity
of $\tkappa \otimes \CP_{\sigma}$.  This map is necessarily surjective,
so by applying the functor $\cInd_K^G$, we obtain a surjection
$$\CP_{K,\tau} \otimes_{W(k)} \OO' \rightarrow \cInd_K^G \tilde \tau.$$
Composing this surjection with the nonzero maps
$$\cInd_K^G \tilde \tau \rightarrow \cInd_K^G (\tilde \tau) \otimes_{W(k)} \CK \rightarrow \pi_0$$
yields a nonzero map $\CP_{K,\tau} \otimes_{\CK} \CK' \rightarrow \pi_0$, and hence a nonzero map
$$\CP_{K,\tau} \otimes_{\CK} \CK' \rightarrow \pi \otimes_{\CK} \CK'.$$
As
$\Hom_{\CK'[G]}(\CP_{K,\tau} \otimes_{\CK} \CK', \pi \otimes_{\CK} \CK')$ is isomorphic to
$\Hom_{\CK[G]}(\CP_{K,\tau}, \pi) \otimes \CK'$, there exists a nonzero map
from $\CP_{K,\tau}$ to $\pi$, which must be surjective by simplicity
of $\pi$.
\end{proof}

It will follow from results in section~\ref{sec:generic} that not every simple quotient of $\CP_{K,\tau}$ 
has the above form.

Our next goal is to use the $\CP_{K,\tau}$ to construct projectives that admit surjections
onto representations with given cuspidal support.  We must first introduce some additional language.
As a maximal distinguised cuspidal $k$-type
determines an inertial equivalence class of cuspidal representations, we will
sometimes say that the supercuspidal or cuspidal support of a representation $\Pi$ is given by
a collection $\{(K_1,\tau_1), \dots, (K_r,\tau_r)\}$ of maximal distinguished
cuspidal $k$-types; this means that $\Pi$ has supercuspidal (or cuspidal) support
$(M,\pi)$, where $M$ is a ``block diagonal'' subgroup of the form
$\GL_{n_1} \times \dots \times \GL_{n_r}$, and $\pi$ is a tensor product
$\pi_1 \otimes \dots \otimes \pi_r$ where $\pi_i$ is in the inertial equivalence class
determined by $(K_i,\tau_i)$ for all $i$.

\begin{definition}
Let $M$ be a Levi subgroup of $G$ and let $\pi$ be an irreducible
cuspidal representation of $M$ over $k$.  Let $\Pi$ be a simple smooth $W(k)$-module.
We say that $(M,\pi)$ belongs to the {\em mod $\ell$ inertial cuspidal support} of $\Pi$
if either:
\begin{enumerate}
\item $\Pi$ is killed by $\ell$, and its cuspidal support
is the inertial equivalence class of $(M,\pi)$, or
\item $\ell$ is invertible on $\Pi$, and there exists a finite extension $\CK'$ of $\CK$,
with ring of integers $\OO'$,
such that $\Pi \otimes_{\CK} \CK'$ is a direct sum of absolutely simple $\CK'[G]$-modules,
and for some (equivalently every) absolutely simple summand $\Pi_0$ of $\Pi \otimes_{\CK} \CK'$,
there exists a smooth $\OO'$-integral
representation $\tpi$ of $M$ lifting $\pi$, such 
that the cuspidal support of $\Pi_0$
is the inertial equivalence class of $(M,\tpi)$.
\end{enumerate}
\end{definition}

In this language, Proposition~\ref{prop:cuspidal quotient} says that every simple
$W(k)[G]$-module with mod $\ell$ inertial cuspidal support given by $(K,\tau)$ is a quotient
of $\CP_{K,\tau}$.  

We will also need a notion of mod $\ell$ inertial supercuspidal support.  We first 
recall a standard result about the behavior of supercuspidal support under
reduction mod $\ell$:

\begin{proposition}[\cite{vigss}, 1.5] \label{prop:supercuspidal support reduction}
Let $\tPi$ be an absolutely
irreducible smooth integral representation of $G$ over a finite extension
$\CK'$ of $\CK$, with supercuspidal support $(M,\tpi)$.  Then $\tpi$ is an integral
representation of $M$.  Moreover, let $\Pi$ and $\pi$ denote the mod $\ell$
reductions of $\tPi$ and $\tpi$, respectively.  Then $\pi$ is irreducible and cuspidal
(but not necessarily supercuspidal).  Moreover, the supercuspidal support
of any simple subquotient of $\Pi$ is equal to the supercuspidal support
of $\pi$.
\end{proposition}

\begin{definition} \label{def:mod ell support}
Let $M$ be a Levi subgroup of $G$ and let $\pi$ be an irreducible
supercuspidal representation of $M$ over $k$.  Let $\Pi$ be a simple smooth
$W(k)[G]$-module.  We say that $(M,\pi)$ belongs to the {\em mod $\ell$ inertial supercuspidal support}
of $\Pi$ if there exists a Levi subgroup $M'$ of $G$ containing $M$,
and an irreducible cuspidal representation $\pi'$ of $M'$ over $k$,
such that $\Pi$ has mod $\ell$ inertial cuspidal support containing $(M',\pi')$,
and $\pi'$ has supercuspidal support containing $(M,\pi)$.
\end{definition}

\begin{proposition}
Let $\Pi$ be an irreducible smooth integral representation of $G$ over
a finite extension $\CK'$ of $\CK$.  The following are equivalent:
\begin{enumerate}
\item $\Pi$ has mod $\ell$ inertial supercuspidal support $(M,\pi)$.
\item Every simple subquotient of the mod $\ell$ reduction
of $\Pi$ has supercuspidal support inertially equivalent to $(M,\pi)$.
\end{enumerate}
\end{proposition}
\begin{proof}
This is immediate from Proposition~\ref{prop:supercuspidal support reduction}.
\end{proof}

If $M$ is a Levi subgroup of $G$, and $\pi$ is an irreducible cuspidal representation
of $M$ over $k$,
define $\CP_{(M,\pi)}$ to be the normalized parabolic induction
$$\CP_{(M,\pi)} := i_P^G [\CP_{K_1,\tau_1} \otimes \dots \otimes \CP_{K_r,\tau_r}],$$
where $P$ is a parabolic subgroup whose associate Levi subgroup is $M$, and
the $(K_i,\tau_i)$, are a sequence of maximal distinguished cuspidal $k$-types
whose associated mod $\ell$ inertial equivalence class is $(M,\pi)$.

\begin{remark} \rm Strictly speaking, $\CP_{(M,\pi)}$ may depend on the choice of $P$, as well
as the particular pairs $(K_i,\tau_i)$; we suppress
these dependences from the notation.  In fact, it seems likely that changing $P$, or replacing
some of the $(K_i,\tau_i)$ with equivalent types,
give rise to isomorphic modules $\CP_{(M,\pi)}$, but we will not need this and do not attempt to
prove it.
\end{remark}

\begin{lemma} Let $\Pi$ be an absolutely irreducible representation of $G$ over $k$ or a finite
extension $\CK'$ of $\CK$, and let
$\pi_1, \dots, \pi_r$ be a sequence of absolutely irreducible cuspidal representations such that
$\Pi$ is a quotient of $i_P^G [\pi_1 \otimes \dots \otimes \pi_r]$.  Let $s_i$ be the permutation
of $1, \dots, r$ that interchanges $i$ and $i+1$ and fixes all other integers.  Then
either $\Pi$ is a quotient of $i_P^G [\pi_{s_i(1)} \otimes \dots \otimes \pi_{s_i(r)}]$,
or $\pi_i$ and $\pi_{i+1}$ are inertially equivalent.
\end{lemma}
\begin{proof}
The parabolic induction $i_P^G [\pi_1 \otimes \dots \otimes \pi_r]$ is isomorphic to 
$i_P^G [\pi_{s_i(1)} \otimes \dots \otimes \pi_{s_i(r)}]$ unless 
$\pi_i = (\abs \circ \det)^{\pm 1} \pi_{i+1}$.  If this is the case then
$\pi_i$ is inertially equivalent to $\pi_{i+1}$.
\end{proof}

\begin{proposition} \label{prop:cuspidal support}
let $\Pi$ be a simple smooth $W(k)[G]$-module with mod $\ell$ cuspidal support
given by the inertial equivalence class $(M,\pi)$.  Then
$\Pi$ is a quotient of $\CP_{(M,\pi)}$.
\end{proposition}
\begin{proof}
Either $\Pi$ is defined and absolutely irreducible over $k$, or $\ell$ is invertible
on $\Pi$ and there exists
a finite extension $\CK'$ of $\CK$ such that $\Pi \otimes_{\CK} \CK'$
is a direct sum of absolutely simple $\CK'[G]$-modules.
Choose an
absolutely irreducible cuspidal representation $\tpi$ of $M$,
defined over the appropriate field ($k$ or $\CK'$) such that, over this field, $\Pi$ admits a
nonzero map from $i_P^G \pi.$  Write $\pi = \pi_1 \otimes \dots \otimes \pi_r$,
and choose $(K_i,\tau_i)$ maximal distinguished cuspidal types such that $\pi_i$
is in the inertial equivalence class determined by $(K_i,\tau_i)$ for all $i$.
Then, for a suitable parabolic subgroup $P$, we have 
$$\CP_{M,\pi} = i_P^G [\CP_{(K_1,\tau_1)} \otimes \dots \otimes \CP_{(K_r,\tau_r)}].$$
Reorder the $\tpi_i$ such
that for all $i$ $\tpi_i$ is in the mod $\ell$ inertial equivalence class determined by
$(K_i,\tau_i)$.  By the previous lemma, we may do this and still assume that there is a nonzero map
from $i_P^G \tpi$ to $\Pi$.  Now we have a nonzero
map of $\CP_{K_i,\tau_i}$ into $\tpi_i$ for each $i$, that is surjective if $\tpi$ is
defined over $k$ and that becomes surjective after inverting $\ell$ if $\tpi$ is
defined over some $\CK'$.  Hence, after induction, we obtain
a nonzero map of $\CP_{(M,\pi)}$ to $\i_P^G \tpi$
that is surjective if $\Pi$ is defined over $k$, and becomes surjective after tensoring
with $\CK'$ if $\ell$ is invertible on $\Pi$.  If $\Pi$ is defined over $k$,
then composing this surjection with
the surjection of $i_P^G \tpi$ onto $\Pi$ yields a nonzero map of $\CP_{(M,\pi)}$
to $\Pi$; this map must be surjective as $\Pi$ is simple.  On the other hand, if $\ell$
is invertible on $\Pi$, we have a nonzero map
$$\CP_{(M,\pi)} \otimes_{W(k)} \CK' \rightarrow \Pi \otimes_{\CK} \CK'$$
obtained by composing the surjection of $\CP_{(M,\pi)} \otimes \CK'$ onto $i_P^G \tpi$
with the map $i_P^G \tpi \rightarrow \Pi \otimes_{\CK} \CK'$.  We thus have a nonzero
map of $\CP_{(M,\pi)}$ onto $\Pi$, and the result follows.
\end{proof}

We will see later that in fact the $W(k)[G]$-modules $\CP_{(M,\pi)}$ are
projective, and once we have established this they will form the basic building blocks
of our theory.  

\section{Finite group theory} \label{sec:finite}

Our first step in understanding the projectives $\CP_{(M,\pi)}$ and
$\CP_{K,\tau}$ is to understand the projective envelope $\CP_{\sigma}$
of a representation $\sigma$ of $\GL_n(\FF_q)$.  This mostly uses
standard facts from the representation theory 
of $\GL_n(\FF_q)$ that we now recall.  For conciseness,
we let $\overline{G}$ denote the group $\GL_n(\FF_q)$,
and fix an $\ell$ prime to $q$.

If $\overline{P} = \overline{M}\overline{U}$ is a parabolic subgroup of $\overline{G}$, with 
Levi subgroup $\overline{M}$
and unipotent radical $\overline{U}$, we have parabolic induction and restriction functors:
$$i_{\overline{P}}^{\overline{G}}: \Rep_{W(k)}(\overline{M}) \rightarrow \Rep_{W(k)}(\overline{G})$$
$$r_{\overline{G}}^{\overline{P}}: \Rep_{W(k)}(\overline{G}) \rightarrow \Rep_{W(k)}(\overline{M}).$$
Here $i_{\overline{P}}^{\overline{G}}$ takes a $W(k)[\overline{M}]$-module, considers it as a $W(k)[\overline{P}]$-module
by letting $\overline{U}$ act trivially, and induces to $\overline{G}$.  Its adjoint
$r_{\overline{G}}^{\overline{P}}$ (which is both a left and right adjoint in the finite group case)
takes the $\overline{U}$-invariants of a $W(k)[\overline{G}]$-module and considers the resulting space
as a $W(k)[\overline{M}]$-module.

Just as in the representation theory of $\GL_n$ over a local field, we can then define:
\begin{definition}
A $W(k)[\overline{G}]$-module $\pi$ is {\em cuspidal} if $r_{\overline{G}}^{\overline{P}} \pi = 0$
for all proper parabolics $\overline{P}$ of $\overline{G}$.
An irreducible $k[\overline{G}]$ or $\overline{\CK}[\overline{G}]$-module $\pi$ is
{\em supercuspidal} if it does not arise as a subquotient of $i_{\overline{P}}^{\overline{G}} \pi'$
for any proper parabolic subgroup $\overline{P}=\overline{M}\overline{U}$ of $\overline{G}$ and
any representation $\pi'$ of $\overline{M}$ (over $k$ or $\overline{\CK}$, as appropriate).
\end{definition}

Over $\overline{\CK}$, an irreducible representation $\pi$ is cuspidal
if and only if it is supercuspidal; over $k$ a supercuspidal representation
is cuspidal but the converse need not hold.  We also define:

\begin{definition}
Let $\overline{P} = \overline{M}\overline{U}$ be a parabolic subgroup of $\overline{G}$, and let $\pi'$ be an
irreducible representation of $\overline{M}$.
\begin{enumerate}
\item An irreducible $k[\overline{G}]$ or $\overline{\CK}[\overline{G}]$-module $\pi$
has {\em cuspidal support} $(\overline{M},\pi')$ if $\pi'$ is cuspidal and $\pi$
is a quotient of $i_{\overline{P}}^{\overline{G}} \pi'$.
\item An irreducible $k[\overline{G}]$ or $\overline{\CK}[\overline{G}]$-module $\pi'$
has {\em supercuspidal support} $(\overline{M},\pi')$ if $\pi'$ is supercuspidal
and $\pi$ is a subquotient of $i_{\overline{P}}^{\overline{G}} \pi'$.
\end{enumerate}
\end{definition}

The cuspidal and supercuspidal support of an irreducible $\pi$ always exist,
and are unique up to $\overline{G}$-conjugacy.  Over $\overline{\CK}$ the two notions coincide,
but this is not true over $k$ because of the existence of cuspidal representations
that are not supercuspidal.  As with representations of $\GL_n(F)$, the uniqueness of
cuspidal support (up to $\overline{G}$-conjugacy) is an immediate consequence of
Frobenius reciprocity, but the uniqueness of supercuspidal support is deep and requires
work (in the setting of general linear groups over finite fields, it follows from
the classification of irreducible modular representations of $\overline{G}$ due to James~\cite{james}.

As the parabolic induction and restriction functors are defined on the level
of $W(k)[\overline{G}]$-modules, it is clear that
the reduction mod $\ell$ of a cuspidal representation is cuspidal.  The notion of
supercuspidal support is compatible with reduction mod $\ell$ in the following sense:
if $\tpi$ is an irreducible representation of $\overline{G}$ over $\overline{\CK}$,
and $\pi$ is any subquotient of its mod $\ell$ reduction, then the supercuspidal support
of $\pi$ is equal to the supercuspidal support of the mod $\ell$ reduction of
the supercuspidal support of $\tpi$.

Deligne-Lusztig theory provides a parameterization of the irreducible cuspidal representations
of $\overline{G}$ over $\overline{\CK}$ in terms of semisimple elements $s$ of
$\overline{G}$ whose characteristic polynomials are irreducible, up to conjugacy.  To an arbitrary
semisimple element $s$ (up to conjugacy), we associate a subset of the irreducible representations
of $\overline{G}$ over $\overline{\CK}$ as follows: let $\overline{M}_s$ be the 
minimal split Levi subgroup of $\overline{G}$ containing $s$.
Then $\overline{M}_s$ is a product of general linear groups $\overline{G}_{n_i}$, and the factors
$s_i$ of $s$ under this decomposition all have irreducible characteristic polynomials.
We call the $s_i$ the ``irreducible factors'' of $s$, and refer to an $s$ with only one irreducible factor as ``irreducible.''
Thus each $s_i$ yields a cuspidal representation of $\overline{G}_{n_i}$, and hence $s$ yields a
cuspidal representation $\pi'$ of $\overline{M}_s$.  Let $\CI(s)$ be the
set of irreducible representations of $\overline{G}$ over $\overline{\CK}$
with cuspidal support $(\overline{M}_s,\pi')$.  Note that if $s$ and $t$ are conjugate
then the pair $(\overline{M}_t,\pi'')$ attached to $t$ is conjugate to the pair $(\overline{M}_s,\pi')$, so that
$\CI(s) = \CI(t)$.  More generally, if $\overline{M}$ is a Levi subgroup of $\overline{G}$,
and $s$ is a semisimple conjugacy class in $\overline{M}$, we let $\CI_{\overline{M}}(s)$ be the set of irreducible
representations of $\overline{M}$ over $\overline{\CK}$ whose cuspidal support is $\overline{M}$-conjugate
to $(\overline{M}_s,\pi')$.  We let $I_s$ denote the parabolic induction
$i_{\overline{P}}^{\overline{G}} \pi'$; $I_s$ depends only on the conjugacy class of $s$, and the irreducible
representations that appear with nonzero multiplicity in $I_s$ 
are precisely the elements of $\CI(s)$.

We now recall the concept of a generic representation of $\overline{G}$.  Let $\overline{U}$ be the unipotent
radical of a Borel subgroup of $\overline{G}$, and let $\Psi: \overline{U} \rightarrow W(k)^{\times}$ be
a generic character.  (For instance, if $U$ is the subgroup of upper triangular matrices
with $1$'s on the diagonal, we can fix a nontrivial map $\psi$ of $\FF_{q}^{+}$ into
$W(k)^{\times}$ and set $\Psi(u) = \psi(u_{12} + ... + u_{n-1,n})$.)  We say an
irreducible representation of $\overline{G}$ over $k$ (resp. $\overline{\CK}$)
is {\em generic} if its restriction to $\overline{U}$ contains a copy of $\Psi \otimes_{W(k)} k$
(resp. $\Psi \otimes_{W(k)} \overline{\CK}$).  By Frobenius reciprocity a representation
is generic if and only if it admits a nontrivial map from $\cInd_{\overline{U}}^{\overline{G}} \Psi$.

We summarize the relevant facts about generic representations that we will need below:
\begin{enumerate}
\item If $\pi$ is an irreducible generic representation over $k$ (resp. $\overline{\CK}$)
then its restriction to $\overline{U}$ contains exactly one copy of $\Psi \otimes_{W(k)} k$
(resp. $\Psi \otimes_{W(k)} \overline{\CK}$.)  (Uniqueness of Whittaker models.)
\item Every cuspidal representation is generic.
\item If $\overline{P}$ is a parabolic subgroup of $\overline{G}$, with Levi subgroup $\overline{M}$, and
$\pi$ is an irreducible generic representation of $\overline{M}$, then $i_{\overline{P}}^{\overline{G}} \pi$ has a
unique irreducible generic subquotient.  (In particular there is, up to isomorphism, a unique
generic irreducible representation with given supercuspidal support.)  On the other hand,
if $\pi$ is irreducible but not generic, then $i_{\overline{P}}^{\overline{G}} \pi$ has no generic
subquotient.
\end{enumerate}

In light of these facts, for any
semisimple element $s$ of $\overline{G}$, we let $\St_s$ denote the unique 
irreducible generic representation of $\overline{G}$ over $\overline{\CK}$ that lies in $\CI(s)$.
(Of course, $\St_s$ only depends on $s$ up to
conjugacy.)  Note that $\St_s$ is a direct summand of $I_s$.  It will be necessary to understand the behavior
of $I_s$ and $\St_s$ under parabolic restriction.

\begin{proposition} \label{prop:steinberg restriction}
We have a decomposition:
$$r_{\overline{G}}^{\overline{P}} \St_s \cong \bigoplus_{t} \St_{\overline{M},t},$$
where $t$ runs over a set of representatives for $\overline{M}$-conjugacy classes of semisimple elements
of $\overline{M}$ that are $\overline{G}$-conjugate to $s$, and $\St_{\overline{M},t}$ is the unique irreducible generic
representation of $\overline{M}$ over $\overline{\CK}$ that lies in $\CI_{\overline{M}}(t)$.
\end{proposition}
\begin{proof}
By Frobenius Reciprocity, for any irreducible representation $\sigma$ of $\overline{M}$, we
have an isomorphism:
$$\Hom_{\overline{M}}(\sigma, r_{\overline{G}}^{\overline{P}} \St_s) = 
\Hom_{\overline{G}}(i_{\overline{P}}^{\overline{G}} \sigma, \St_s).$$
As $\St_s$ is generic and irreducible, the right hand side is zero unless $i_{\overline{P}}^{\overline{G}} \sigma$
has an irreducible generic summand; if this is the case then $\sigma$ is irreducible and generic,
$i_{\overline{P}}^{\overline{G}} \sigma$ has a unique irreducible generic summand, so the right hand side has 
dimension at most one.
In particlar, every irreducible summand of $r_{\overline{G}}^{\overline{P}} \St_s$ 
is generic and occurs with multiplicity one.
Moreover, if the right hand side is nonzero, then the cuspidal support of one (hence every) 
summand of $\sigma$ is given by $s$, so the $\overline{M}$-cuspidal support of $\sigma$ is given by a
conjugacy class $t$ of $\overline{M}$ that is $\overline{G}$-conjugate to $s$.
\end{proof}

We can rewrite this isomorphism as follows: Let $\overline{M}_s$ be the minimal split Levi of $\overline{G}$ containing
$s$, so that $\St_{\overline{M}_s,s}$ is cuspidal.  Fix a maximal torus of
$\overline{M}$; then conjugating $s$ appropriately we may assume that it is also a maximal torus of
$\overline{M}_s$.  Consider the set $W(\overline{M}_s,\overline{M})$ of elements $w$ of $W(\overline{G})$ such that
$w \overline{M}_s w^{-1}$ lies in $\overline{M}$.  Then $W(\overline{M}_s,\overline{M})$ has a left
action by $W(\overline{M})$ and a right action by the subgroup $W_{\overline{M}_s}(s)$ of $W(\overline{G})$ consisting
of those $w$ in $W(\overline{G})$ such that $w\overline{M}_s w^{-1} = \overline{M}_s$ and $w s w^{-1}$ is 
$\overline{M}_s$-conjugate to $s$.
Moreover, the map $s \mapsto w s w^{-1}$ then yields a bijection between 
$W(\overline{M}) \backslash W(\overline{M}_s,\overline{M}) / W_{\overline{M}}(s)$ and the set of $\overline{M}$-conjugacy
classes $t$ of elements that are $\overline{G}$-conjugate to $s$.  We thus obtain a decomposition:

\begin{proposition} \label{prop:steinberg restriction 2}
We have a decomposition:
$$r_{\overline{G}}^{\overline{P}} \St_s \cong \bigoplus_{w} \St_{\overline{M},wsw^{-1}},$$
where $w$ runs over a set of representatives for
$W(M) \backslash W(\overline{M}_s,\overline{M}) / W_{\overline{M}}(s)$. 
\end{proposition}

We will need to understand the compatibility of this decomposition with a decomposition of
$r_{\overline{G}}^{\overline{P}} I_s$.  We have:

\begin{proposition} \label{prop:induction restriction}
There is a direct sum decomposition:
$$r_{\overline{G}}^{\overline{P}} I_s \cong \bigoplus_w I_{\overline{M},w s w^{-1}},$$
where $w$ runs over a set of representatives for
$W(M) \backslash W(\overline{M}_s,\overline{M})$, and $I_{\overline{M},w s w^{-1}}$
is the parabolic induction:
$i_{w \overline{P}_s w^{-1} \cap \overline{M}}^{\overline{M}} \St_{\overline{M}_s,s}^w.$
Moreover, on a summand $\St_{\overline{M},w's(w')^{-1}}$ of $r_{\overline{G}}^{\overline{P}} \St_s$,
the map $r_{\overline{G}}^{\overline{P}} \St_s \rightarrow r_{\overline{G}}^{\overline{P}} I_s$ induces an injective map
$$\St_{\overline{M},w's(w')^{-1}} \rightarrow \bigoplus_w I_{\overline{M}, w s w^{-1}}.$$
\end{proposition}
\begin{proof}
This is a consequence of Mackey's induction-restriction formula, together with Proposition~\ref{prop:steinberg restriction}.
\end{proof}

We now turn to considerations related to reduction modulo $\ell$.
Given an irreducible cuspidal representation $\pi$ of $\overline{G}$ over $\overline{\CK}$,
its mod $\ell$ reduction is irreducible and cuspidal.  Every cuspidal
representation of $\overline{G}$ over $k$ arises by mod $\ell$ reduction from some such $\pi$.
If $\pi$ corresponds to a
semisimple conjugacy class $s$, then the reduction mod $\ell$ of $\pi$ is supercuspidal
if, and only if, the characteristic polynomial of the $\ell$-regular part $s^{\reg}$
of $s$ is irreducible.  Moreover, if $\pi$ and $\pi'$ are irreducible cuspidal representations
correspond to semisimple conjugacy classes $s$ and $s'$, then the mod $\ell$ reductions
of $\pi$ and $\pi'$ coincide if, and only if, $s^{\reg} = (s')^{\reg}$.  Thus the supercuspidal
representations of $\overline{G}$ over $k$ are parameterized by $\ell$-regular semisimple conjugacy classes
in $\overline{G}$ with irreducible characteristic polynomial, and the cuspidal representations of
$\overline{G}$ over $k$ are parameterized by $\ell$-regular semisimple conjugacy classes $s'$
such that there exists a semisimple conjugacy class $s$, with irreducible characteristic
polynomial, such that $s' = s^{\reg}$.  

Let $\pi$ be a cuspidal but not supercuspidal representation of $\overline{G}$ over $k$,
and let $s'$ be the corresponding semisimple element.  Such an $s'$, factors
into $m$ identical irreducible factors $s'_0$ for some $m$ dividing $n$.
Moreover, the supercuspidal
support of the $\pi$ that corresponds to $s'$ is the tensor product of $m$ copies
of the supercuspidal representation of $\overline{G}_{\frac{n}{m}}$ corresponding to $s'_0$.

Fix an $\ell$-regular semisimple element $s'$ of $\overline{G}$, and let
${\mathcal E}(s')$ be the union of the sets $\CI(s)$ for those semisimple $s$ with $s^{\reg} = s'$.
Let
$e_{s'}$ be the idempotent in $\overline{\CK}[\overline{G}]$ that is the sum of the primitive idempotents
$e_{\pi}$ for all $\pi$ in ${\mathcal E}(s')$.  Then one has:

\begin{theorem} \label{thm:integrality}
The element $e_{s'}$ lies in $W(k)[\overline{G}]$.
\end{theorem}
\begin{proof}
This is an immediate consequence of~\cite{CE}, Theorem 9.12.
\end{proof}

Fix an irreducible cuspidal representation $\pi$ of $\overline{G}$ over $k$ that 
is not supercuspidal, corresponding to an $\ell$-regular semisimple element
$s'$ (up to conjugacy).  The representation $\pi$ arises as the mod $\ell$ reduction of an irreducible
cuspidal representation $\tpi$ of $\pi$ over $\overline{\CK}$; there is thus a semisimple element
$s$ of $\overline{G}$, with irreducible characteristic polynomial, such that $s' = s^{\reg}$.  
It is then easy to see that there exists an $m > 1$ dividing $n$ such
that, up to conjugacy, $s'$ factors as a block matrix consisting of $m$ irreducible factors $s'_0$.
Moreover, $m$ lies in the set $\{1,e_q,\ell e_q, \ell^2 e_q, \dots \}$,
where $e_q$ is the order of $q$ modulo $\ell$.
The supercuspidal support of $\pi$ is the tensor product of
$m$ copies of the supercuspidal representation $\pi_0$ corresponding to $s'$.

\begin{proposition}
Suppose we have an irreducible representation $\tpi$ of $\overline{G}$ over $\CK$
whose mod $\ell$ reduction contains $\pi$ as a subquotient.  Let $s$ be the semisimple
conjugacy class corresponding to the supercuspidal support of $\tpi$.  Then $s^{\reg} = s'$.
\end{proposition}
\begin{proof}
The supercuspidal support of $\tpi$ is given by the pair $(M_s,\tpi_{M,s})$, where $M_s$ is the minimal
split Levi containing $s$ and $\tpi_{M,s}$ is the corresponding supercuspidal representation of $M_s$ over $\overline{\CK}$.
The mod $\ell$ reduction $\tpi$ is thus isomorphic to a subquotient of a parabolic induction $i_{\overline{P}}^{\overline{G}} \pi_{M,s}$,
where $\pi_{M,s}$ is the mod $\ell$ reduction of $\tpi_{M,s}$.  By the above discussion, $\pi_{M,s}$ is the cuspidal
representation of $M_s$ attached to $s^{\reg}$, which has supercuspidal support
$(M_{s^{\reg}}, \pi_{s^{\reg}})$.  Since $\pi$ is a subquotient of a parabolic induction of $\pi_{M,s}$,
the supercuspidal support of $\pi$ is also $(M_{s^{\reg}}, \pi_{s^{\reg}})$.  In particular $s^{\reg}$ and
$s'$ represent the same conjugacy class.
\end{proof}

Fix an irreducible cuspidal representation $\pi$ of $\overline{G}$ over $k$,
and let $s'$ be a representative of the corresponding $\ell$-regular semisimple
conjugacy class.  Let $\CP_{\pi}$ be a projective envelope of $\pi$. 

\begin{proposition} \label{prop:endomorphisms}
The $W(k)[\overline{G}]$-module $e_{s'} \Ind_{\overline{U}}^{\overline{G}} \Psi$ is a projective envelope
of $\pi$, and hence is isomorphic to $\CP_{\pi}$.  In particular
$\CP_{\pi} \otimes_{W(k)} \overline{\CK}$ is isomorphic to the direct sum 
$$\bigoplus_{s: s^{\reg} = s'} \St_s.$$
Moreover, the endomorphism
ring $\End_{W(k)[\overline{G}]}(\CP_{\pi})$ is a reduced, commutative, free $W(k)$-module of
finite rank.
\end{proposition}
\begin{proof}
The module $e_{s'} \Ind_{\overline{U}}^{\overline{G}} \Psi$ is projective, as induction takes
projectives to projectives.  Suppose $\pi'$ is an irreducible representation
of $\overline{G}$ over $k$ that admits a nonzero map from $e_{s'} \Ind_{\overline{U}}^{\overline{G}} \Psi$.
Then $\pi'$ is generic and has supercuspidal support given by $s'$, so
$\pi' = \pi$.  Thus $e_{s'} \Ind_{\overline{U}}^{\overline{G}} \Psi$ is a projective envelope of $\pi$.

As $\CP_{\pi}$ is free of finite rank over $W(k)$, so is $\End_{W(k)[\overline{G}]}(\CP_{\pi})$.
Thus $\End_{W(k)[\overline{G}]}(\CP_{\pi})$ embeds into
$\End_{W(k)[\overline{G}]}(\CP_{\pi}) \otimes_{W(k)} \overline{\CK}$.  The latter is
the endomorphism ring of $e_{s'} \Ind_{\overline{U}}^{\overline{G}} (\Psi \otimes_{W(k)} \overline{\CK})$.
This module is a direct sum of the generic representations $\tpi$ of $\overline{G}$ over
$\overline{\CK}$ whose mod $\ell$ reduction contains $\pi$, each with multiplicity one.
In particular its endomorphism ring is reduced and commutative.
\end{proof}

Let $\overline{P}$ be a parabolic subgroup of $\overline{G}$, let $\overline{U}$ be its unipotent radical, and let
$\overline{M}$ be the corresponding Levi subgroup.  We will need to understand the restriction
$r_{\overline{G}}^{\overline{P}} \CP_{\pi}$.  The following lemma is an immediate consequence of 
Proposition~\ref{prop:steinberg restriction}:

\begin{lemma} \label{lemma:projective restriction}
We have an isomorphism:
$$r_{\overline{G}}^{\overline{P}} \CP_{\pi} \otimes \overline{\CK} = 
\bigoplus_{s: s^{\reg} = s'} \bigoplus_{t \sim s} \St_{\overline{M},t}.$$
In particular,
the endomorphism ring $\End_{W(k)[\overline{M}]}(r_G^P \CP_{\pi})$ is reduced and commutative.
\end{lemma}

Suppose that each block
of $\overline{M}$ has size divisible by $\frac{n}{m}$.  Then $s'$ is conjugate 
to an element of $\overline{M}$, and this element is unique up to $\overline{M}$-conjugacy.
Then $\St_{\overline{M},s'}$ is the unique generic representation of $\overline{M}$ determined
by this conjugacy class in $\overline{M}$.  Let $\CP_{\overline{M},s'}$ be a
projective envelope of $\St_{\overline{M},s'}$.

\begin{proposition} \label{prop:projective restriction}
The restriction $r_{\overline{G}}^{\overline{P}} \CP_{\pi}$ is zero unless each block of
$\overline{M}$ has size divisible by $\frac{n}{m}$.  When the latter occurs there is an
isomorphism:
$$r_{\overline{G}}^{\overline{P}} \CP_{\pi} \cong \CP_{\overline{M},s'}.$$
\end{proposition}
\begin{proof}
Every Jordan-H\"older constituent of $\CP_{\pi}$ has supercuspidal support corresponding to
a tensor product of $m$ copies of the representation with supercuspidal support $s'_0$, and thus
$r_{\overline{G}}^{\overline{P}} \CP_{\pi} = 0$ unless the condition on the block size of $M$ holds.
When this condition does hold, $r_{\overline{G}}^{\overline{P}} \CP_{\pi}$ is projective, and hence
a direct sum of indecomposable projectives.  Moreover, $i_{\overline{P}}^{\overline{G}} \St_{\overline{M},s'}$
contains a unique generic Jordan-H\"older constituent; this constituent is necessarily isomorphic to $\pi$,
and thus yields a map $\CP_{\pi} \rightarrow i_{\overline{P}}^{\overline{G}} \St_{\overline{M},s'}$.
By Frobenius reciprocity we obtain a map:
$$r_{\overline{G}}^{\overline{P}} \CP_{\pi} \rightarrow \St_{\overline{M},s'},$$
and therefore $r_{\overline{G}}^{\overline{P}} \CP_{\pi}$ has a summand isomorphic to
$\CP_{\overline{M},s'}$.  It is then easy to see that this is the only summand, by tensoring with $\overline{\CK}$
and applying Lemma~\ref{lemma:projective restriction}.
\end{proof}

\begin{corollary} \label{cor:finite endomorphisms}
The map $$e_{s'} Z(W(k)[\overline{G}]) \rightarrow \End_{W(k)[\overline{G}]}(\CP_{\pi})$$
admits a section:
$$\End_{W(k)[\overline{G}]}(\CP_{\pi}) \rightarrow e_{s'} Z(W(k)[\overline{G}])$$
such that the composition:
$$\End_{W(k)[\overline{G}]}(\CP_{\pi}) \rightarrow e_{s'} Z(W(k)[\overline{G}]) \rightarrow
\End_{W(k)[\overline{G}]}(\CP_{\pi})$$
is the identity.
\end{corollary}
\begin{proof}
The $W(k)[\overline{G}]$-module $e_{s'} W(k)[\overline{G}]$ is a projective module,
and is therefore a direct sum (with multiplicities)
of projective envelopes of modules in the block containing $\pi$.  Any such module $\pi'$ has
supercuspidal support corresponding to $(s'_0)^m$, and hence cuspidal support
of the form $(\overline{M},\pi_{\overline{M},s'})$ for some Levi $\overline{M}$ of $G$.
It is easy to see that $i_{\overline{P}}^{\overline{G}} \CP_{\overline{M},s'}$ is then a
projective envelope of $\pi'$.  Thus $e_{s'} W(k)[\overline{G}]$ is isomorphic
to a direct sum of modules of the form $i_{\overline{P}}^{\overline{G}} \CP_{\overline{M},s'}$
for various $\overline{M}$.

As $\CP_{\overline{M},s'}$ is isomorphic to a parabolic restriction of $\CP_{\pi}$,
the endomorphisms of $\CP_{\pi}$ act naturally on $\CP_{\overline{M},s'}$, and hence
on $i_{\overline{P}}^{\overline{G}} \CP_{\overline{M},s'}$.  This gives an action
of $\End_{W(k)[\overline{G}]}(\CP_{\pi})$ on $e_{s'} W(k)[\overline{G}]$, via
elements of the center of $e_{s'} W(k)[\overline{G}]$.  The resulting map
$\End_{W(k)[\overline{G}]}(\CP_{\pi}) \rightarrow e_{s'} Z(W(k)[\overline{G}])$ is
the desired section.
\end{proof}

We now obtain more precise results about $\CP_{\pi}$ under the additional hypothesis that
$\ell > n$.

Let $\Phi_e(x)$ be the $e$th cyclotomic polynomial.
\begin{lemma}
Let $r$ be the order of $q$ mod $\ell$, and suppose there exists an
$e$ not equal to $r$ such that $\ell$ divides $\Phi_e(q)$.  Then $\ell$ divides $e$.
\end{lemma}
\begin{proof}
As $q$ has exact order $r$ mod $\ell$, it is clear that $\ell$ divides $\Phi_r(q)$,
and also that $r$ divides $e$.  The polynomial $x^e - 1$ is divisible by 
$\Phi_r(x)\Phi_e(x)$; as $q$ is a root of both of these polynomials
mod $\ell$ we see that $x^e - 1$ has a double root mod $\ell$, and so $\ell$ divides $e$ as
required.
\end{proof}

It follows that for $\ell > n$, there exists at most one $e$ dividing $n$ with
$\Phi_e(q)$ divisible by $\ell$.

As a result, one has:
\begin{lemma}
Let $s'$ be an $\ell$-regular semisimple element of $\GL_n(\FF_q)$,
whose characteristic polynomial is reducible, and
suppose that there exists a semisimple element $s$ of $\GL_n(\FF_q)$
with irreducible characteristic polynomial such that $s' = s^{\reg}$.
Suppose also that $\ell > n$.  Then for any $s$ such that $s' = s^{\reg}$,
either $s = s'$ or the characteristic polynomial of $s$ is irreducible.
\end{lemma}
\begin{proof}
Fix a semisimple element $s$ with irreducible characteristic polynomial such that
$s' = s^{\reg}$.  Then $s$ is contained in a subgroup of $G_n$ isomorphic to
$\FF_{q^n}^{\times}$, and (when considered as an element of $\FF_{q^n}^{\times}$),
$s$ generates $\FF_{q^n}$ over $\FF_q$.  We regard $s$ and $s'$ as elements of
$\FF_{q^n}$.  Write $s = s's_{\ell}$, where $s_{\ell}$ is $\ell$-power torsion.
Then there are integers $e_{\ell}$ and $e'$ dividing $n$ such that
$s'$ generates $\FF_{q^{e'}}$ over $\FF_q$ and $s_{\ell}$ generates $\FF_{q^{e_{\ell}}}$
over $\FF_q$.  The least common multiple of $e'$ and $e_{\ell}$ is equal to $n$.
Note that $e'$ is not equal to $n$, as this would imply that the characteristic polynomial
of $s'$ was irreducible.

In particular, $e_{\ell}$ cannot be equal to $1$, so $e_{\ell}$ is a nontrivial root
of unity.  Thus $\ell$ divides $\Phi_{e_{\ell}}(q)$, and hence $e_{\ell}$ is the unique
$m < n$ such that $\ell$ divides $\Phi_m(q)$.

Now if $s$ is an arbitrary element with $s^{\reg} = s$, we can write $s$ as
$s's_{\ell}$ for some $\ell$-power root of unity $s_{\ell}$.  
Either $s_{\ell}$ is trivial or $s_{\ell}$ lies in $\FF_{q^{e_{\ell}}}$;
as we know that the least common multiple of $e'$ and $e_{\ell}$ is $n$ the latter
case implies that $s$ generates $\FF_{q^n}$ over $\FF_q$, as required.
\end{proof}

In particular, if $\ell > n$, and $\sigma$ is a representation that is cuspidal but not
supercuspidal, corresponding to a semisimple element $s'$, then every irreducible
representation $\tsigma$ of $G_n$ over $\overline{\CK}$ whose mod $\ell$ reduction contains
$\sigma$ is either a supercuspidal lift of $\sigma$, or has supercuspidal support given by
$s'$.  We thus have:

\begin{corollary}
For $\ell > n$, and $\pi$ cuspidal but not supercuspidal, we have:
$$\CP_{\sigma} \otimes \overline{\CK} \cong \St_{s'} \oplus \bigoplus_{\tsigma} \tsigma,$$
where $\tsigma$ runs over the supercuspidal representations lifting $\sigma$.
\end{corollary}

\section{Structure of $\CP_{K,\tau}$ in depth zero} \label{sec:zero generic}

The goal of the next four sections will be to apply the finite group theory
of section~\ref{sec:finite} to understand the structure of $\CP_{K,\tau}$.
Recall that, by definition, we have

$$\CP_{K,\tau} = \cInd_K^G \tkappa \otimes \CP_{\sigma}.$$

The representation $\sigma$ is inflated from a cuspidal representation
of $\GL_{\frac{n}{ef}}(\FF_{q^f})$; such a representation is of the
form $\St_{s'}$ for some semisimple $\ell$-regular element
$s'$ of $\GL_{\frac{n}{ef}}(\FF_{q^f})$.  For conciseness we abbreviate
$\GL_{\frac{n}{ef}}(\FF_{q^f})$ by $\overline{G}$ for this section.

The decomposition
$$\CP_{\sigma} \otimes \overline{\CK} \cong \bigoplus_s \St_s$$
of Proposition~\ref{prop:endomorphisms} gives rise to a decomposition:
$$\CP_{K,\tau} \otimes \overline{\CK} \cong \bigoplus_s \cInd_K^G \tkappa \otimes \St_s,$$
where $s$ runs over a set of representatives for the $\overline{G}$-conjugacy 
classes of semisimple elements $s$ whose $\ell$-regular part is conjugate to $s'$.

\begin{definition} Let $(K,\kappa \otimes \sigma)$ be a maximal distinguished cuspidal
$k$-type.  A {\em generic pseudo-type} attached to $(K,\kappa \otimes \sigma)$
is a $\overline{\CK}$-type of the form $(K,\tkappa \otimes \St_s)$, where $s$ is
a semisimple element of $\overline{G}$ whose $\ell$-regular part is conjugate to $s'$.
\end{definition}

If $\St_s$ is cuspidal, then the generic pseudo-type $(K,\tkappa \otimes \St_s)$
is a maximal distinguished $\overline{\CK}$-type, whose mod $\ell$ reduction is the type
$(K,\kappa \otimes \sigma)$.  In general a generic pseudo-type
is not a type in the usual sense of the word; we will see that in some sense the Hecke
algebras attached to generic pseudo-types are analogues of spherical Hecke algebras.
(In particular, they can be identified with centers of Hecke algebras attached to types.)

In this section and section~\ref{sec:generic} we will prove several
basic results about the structure of representations induced from generic pseudo-types, as well
as certain natural $W(k)[G]$-submodules of these representations.

Our approach makes use of the theory of $G$-covers, which can be found
in~\cite{BK} for a field of characteristic zero, and~\cite{vig98} 
over a general base ring.  We largely follow
the presentation of~\cite{vig98}, section II.
Let $P = MU$ be a parabolic subgroup of $G$, with Levi subgroup $M$, unipotent
radical $U$, and opposite parabolic subgroup $P^{\circ} = MU^{\circ}$.
Let $(K,\tau)$ be an $R$-type of $G$, and let $K_M, K^+, K^-$ denote the intersections
$K \cap M$,  $K \cap U$; $K \cap U^{\circ}$, respectively.
Let $\tau_M$ be the restriction of $\tau$ to $K_M$.  

\begin{definition} The pair $(K,\tau)$ is {\em decomposed with respect to $P$}
if $K = K^- K_M K^+$, and $\tau$ is trivial on $K^+$ and $K^-$.  If $\tau$ is
the trivial representation we will sometimes say that $K$ is {\em decomposed with respect to $P$}.
\end{definition} 

The key example that motivates this definition is as follows: Let $P$ be a standard
parabolic subgroup of $G$, with Levi subgroup $M$, and let $\overline{P}$ and $\overline{M}$
be the images of $P \cap \GL_n(\OO_F)$ and $M \cap \GL_n(\OO_F)$ under the reduction map
$\GL_n(\OO_F) \rightarrow \overline{G} = \GL_n(\FF_q)$.  If we then take
$K_P$ to be the preimage of $\overline{P}$ in $\GL_n(\OO_F)$, and let $\tau$ be any representation
of $K_P$ inflated from $\overline{M}$ via the maps:
$$K_P \rightarrow \overline{P} \rightarrow \overline{M},$$
then one checks easily that the pair $(K_P,\tau)$ is decomposed with respect to $P$.  We will be
concerned primarily with this example for the majority of this section and the next.

Returning for a moment to
the general situation,
suppose that $(K,\tau)$ is decomposed with respect to $P$, and let $\lambda$ be
an element of $M$.  Following~\cite{vig98}, II.4, we say that $\lambda$ is
{\em positive} if $\lambda K^+ \lambda^{-1}$ is contained in $K^+$, and $\lambda^{-1} K^- \lambda$
is contained in $\lambda$.  We say $\lambda$ is {\em negative} if $\lambda^{-1}$ is positive.

We say that $\lambda$ is {\em strictly positive} if the following two conditions hold:
\begin{itemize}
\item For any pair of open compact subgroups $U_1, U_2$ of $U$, there exists a positive $m$ such that
$\lambda^m U_1 \lambda^{-m}$ is contained in $U_2$.
\item For any pair of open compact subgroups $U^{\circ}_1, U^{\circ}_2$ of 
$U^{\circ}$, there exists a positive $m$ such that
$\lambda^{-m} U^{\circ}_1 \lambda^{m}$ is contained in $U^{\circ}_2$.
\end{itemize}

The discussion of~\cite{vig98}, II.3 shows that when $(K,\tau)$ is decomposed with respect
to $P$, there is a natural $R$-linear injection: $T': H(M,K_M,\tau_M) \rightarrow H(G,K,\tau)$,
that takes every element of $H(M,K_M,\tau_M)$ supported on $K_M \lambda K_M$ to an element
supported on $K\lambda K$.  Moreover, let $H(M,K_M,\tau_M)^+$ be the subalgebra of $H(M,K_M,\tau_M)$ 
consisting of elements supported on double cosets $K_M \lambda K_M$ with $\lambda$ positive.  
Then the restriction $T^+$ of $T'$ to $H(M,K_M,\tau_M)^+$ is an algebra homomorphism.

We now have the following result:
\begin{proposition}[\cite{vig98}, II.6] \label{prop:G-cover}
Suppose there exists a central, strictly positive element $\lambda$ of $M$ such that
$T^+(1_{K_M \lambda K_M})$ is an invertible element of $H(G,K,\tau)$, where
$1_{K_M \lambda K_M}$ is the unique element of $H(M,K_M,\tau_M)$ that is supported on $K_M \lambda K_M$, and
whose values at $\lambda$ is the identity endomorphism of $\tau_M$.  Then $T^+$ extends uniquely to an algebra map:
$$T: H(M,K_M,\tau_M) \rightarrow H(G,K,\tau).$$
\end{proposition}
 
If $(K,\tau)$ is decomposed with respect to $P$, and the hypothesis of Proposition~\ref{prop:G-cover}
is satisfied, we say that $(K,\tau)$ is a $G$-cover of $(K_M,\tau_M)$.  We then have:

\begin{theorem}[\cite{BK-semisimple}, p. 55] \label{thm:G-cover}
Suppose that $(K,\tau)$ is a $G$-cover of $(K_M,\tau_M)$, with $\tau$ and $\tau_M$
defined over $\overline{\CK}$, and that the functor $\Hom_K(\tau,-)$ (resp. $\Hom_{K_M}(\tau_M,-)$)
is an equivalence of categories between a block of $\Rep_{\overline{\CK}}(G)$ (resp. $\Rep_{\overline{\CK}}(M)$)
and the category of right $H(G,K,\tau)$-modules (resp. $H(M,K_M,\tau_M)$-modules.)

Then for any $\overline{\CK}[M]$-module $\Pi$, we have an isomorphism of $H(G,K,\tau)$-modules:
$$\Hom_K(\tau, i_P^G \Pi) \cong \Hom_{H(M,K_M,\tau_m)}(H(G,K,\tau), \Hom_{K_M}(\tau_M, \Pi)).$$
\end{theorem}

The first goal of this section is to understand the $\overline{\CK}[G]$-module
$\cInd_K^G \tkappa \otimes \St_s$, in the so called {\em depth zero} case where $K$ is
the subgroup $\GL_n(\OO_F)$ of $G$ and $\tkappa$ is trivial.  Let $\overline{M}$
be the minimal split Levi of $\overline{G}$ containing $s$; conjugating $s$ if necessary we
assume that $\overline{M}$ is standard.  Then there is a unique standard Levi subgroup $M$ of $G$
such that the image of $M \cap \GL_n(\OO_F)$ in $\overline{G}$ is $\overline{M}$.  We also fix a
standard parabolic $P$ with Levi subgroup $M$, and let $\overline{P}$ be the image of
$P \cap \GL_n(\OO_F)$ in $\overline{G}$.  We then set $K_M = K \cap M$ and let $K_P$ be
the preimage of $\overline{P}$ in $K$.  We regard $\St_s$ as a representation of $K$ via inflation,
and $\St_{\overline{M},s}$ as a representation of $K_M$ or $K_P$ (depending on context) via
inflation from $\overline{M}$. 

There is then a close relationship between the pair $(K_P,\St_{\overline{M},s})$ and
the pair $(K,\St_s)$.  Indeed, we have an isomorphism:
$$\cInd_{K_P}^G \St_{\overline{M},s} \cong \cInd_K^G i_{\overline{P}}^{\overline{G}} \St_{\overline{M},s}$$
and (since $\St_s$ is a direct summand of $i_{\overline{P}}^{\overline{G}} \St_{\overline{M},s}$) it follows
that $\cInd_K^G \St_s$ is a direct summand of $\cInd_{K_P}^G \St_{\overline{M},s}$.

On the other hand, the pair $(K_P,\St_{\overline{M},s})$ is decomposed with respect to $P$, and indeed,
one has:

\begin{theorem}
The pair $(K_P,\St_{\overline{M},s})$ is a $G$-cover of the pair $(K_M,\St_{\overline{M},s})$.  Moreover, these
pairs satisfy the hypotheses of Theorem~\ref{thm:G-cover}.
\end{theorem}
\begin{proof}
That the pair in question is a $G$-over follows from~\cite{BK-semisimple}, Theorem 7.2, or, alternatively,
from our description of the Hecke algebra $H(G,K_P,\St_{\overline{M},s})$ below.

By construction, $(K_M,\St_{\overline{M},s})$ is a maximal distinguished cuspidal $M$-type,
and hence $\Hom_{K_M}(\St_{\overline{M},s},-)$ is an equivalence of categories from 
the block of $\Rep_{\overline{\CK}}(M)$ corresponding to the pair $(K_M,\St_{\overline{M},s})$
to the category of right $H(M,K_M,\St_{\overline{M},s})$-modules.  It then follows
from~\cite{BK-types}, Theorem 8.3 that $Hom_{K_P}(\St_{\overline{M},s},-)$
is an equivalence on a block of $\Rep_{\overline{\CK}}(G)$. 
\end{proof}

Let $\pi$ be an irreducible cuspidal $\overline{\CK}$-representation of $M$ containing
the type $(K_M,\St_{\overline{M},s})$.  As an immediate consequence of the above theorem
and Theorem~\ref{thm:G-cover} we deduce:

\begin{corollary} \label{cor:depth zero bernstein component}
The representations $\cInd_{K_P}^G \St_{\overline{M},s}$ and $\cInd_K^G \St_s$ are objects of the block 
$\Rep_{\overline{\CK}}(G)_{(M,\pi)}$.
\end{corollary}

We now turn to the question of computing endomorphism ring of
$\cInd_{K_P}^G \St_{\overline{M},s}$.  This ring is the Hecke algebra
$H(G,K_P,\St_{\overline{M},s})$.  The Mackey formula gives an isomorphism:
$$H(G,K_P,\St_{\overline{M},s}) \cong \bigoplus_g I_g(\St_{\overline{M},s})$$
where $g$ runs over a set of representatives for the double cosets $K_P g K_P$.
Without loss of generality we may restrict our attention to $g$ of the form
$w t$, where $w$ is a permutation matrix and $t$ lies in the standard maximal torus $T$ of $G$.
We may further assume that every entry of $t$ is a power of a fixed uniformizer $\unif$ of $F$.
Let $Z$ be the subgroup of $T$ consisting of elements $t$ such that every entry of $t$ is a power
of $\unif$.

We then have:
\begin{prop} Let $W_{\overline{M}}(s)$ be the group of permutation matrices $w$ that normalize $\overline{M}$, and such that
$w s w^{-1}$ is $\overline{M}$-conjugate to $s$.  Let $g$ have the form $w t$, with $w$ a permutation matrix and $t \in Z.$
Then $I_g(\St_{\overline{M},s})$ is one dimensional if $w$ is in $W_{\overline{M}}(s)$ and $t$ is in the center $Z_M$ of $M$,
and zero-dimensional otherwise.
\end{prop}
\begin{proof}
Suppose first that $t$ does not lies in $Z_M$, or that $w$ does not normalize $M$.  Then there is a unipotent
subgroup $\overline{U}$ of $\overline{M}$, contained in the image of the composition:
$$K_P \cap g K_P g^{-1} \rightarrow K \rightarrow \overline{G},$$ 
that acts
trivially on $\St_{\overline{M},s}^g$.  An element of $I_g(\St_{\overline{M},s})$ is, by definition, 
a $K_P \cap g K_P g^{-1}$-equivariant
map from $\St_{\overline{M},s}$ to $\St_{\overline{M},s}^g$; such a map must factor through the $\overline{U}$-coinvariants
of $\St_{\overline{M},s}$.  As $\St_{\overline{M},s}$ is a cuspidal representation of $\overline{M}$, this space
of coinvariants vanishes.

Now suppose that $t$ lies in $Z_M$ and $w$ normalizes $M$.  Then the image of $K_P \cap g K_P g^{-1}$
under reduction from $\GL_n(\OO_F)$ to $\overline{G}$ contains $\overline{M}$, and so we have an isomorphism:
$$\Hom_{K_P \cap g K_P g^{-1}}(\St_{\overline{M},s},\St_{\overline{M},s}^g) \cong
\Hom_{\overline{M}}(\St_{\overline{M},s},\St_{\overline{M},wsw^{-1}}).$$  The latter
is one-dimensional if, and only if, $w$ lies in $W_{\overline{M}}(s)$ and zero otherwise.
\end{proof}

The characteristic polynomial of $s$ factors as a product of irreducible polynomials $f_1^{m_1} \dots f_r^{m_r}$;
this induces a factorization of $W_{\overline{M}}(s)/W(M)$ as a product of symmetric groups $S_{m_i}$, and a
parallel factorization of $Z$ as a product of (free abelian) groups $Z_i$, such that 
the action of $W_{\overline{M}}(s)/W(M)$ on $Z_i$ factors through $S_{m_i}$ for all $i$.  We may regard elements of $Z_i$
in two ways: first, they are elements of $Z$, and can thus be considered as $n$ by $n$ diagonal matrices whose elements
are powers of $\unif$.  On the other hand, if we fix an identification of the quotient $S_{m_i}$ of $W_{\overline{M}}(s)$
with the Weyl group of $\GL_{m_i}(F_i)$, where $F_i$ is the unramified extension of $F$ of degree $d_i$,
and let $Z'_i$ denote the group 
of diagonal matrices in $\GL_{m_i}(F_i)$ whose entries are powers of $\unif$, 
then there is a unique identification of $Z_i$ with $Z'_i$ 
that is compatible with the actions of
$S_{m_i}$ on $Z_i$ and the Weyl group of $\GL_{m_i}(F_i)$ on $Z'_i$.

If we make these identifications, then the subspace $H_i$ of $H(G,K_P,\St_{\overline{M},s})$ supported on
double cosets of the form $K_P w z K_P$, with $z \in Z_i$ and $w \in S_{m_i}$ is a subalgebra, and in fact a familiar one.
Indeed, the construction of~\cite{BK}, 5.6, gives an isomorphism from the affine Hecke algebra $H(q,m_i)$
(regarded as the Hecke algebra $H(\GL_{m_i}(F_i), I)$, where $I$ is the Iwahori subgroup of $\GL_{m_i}(F_i)$)
to $H_i$.  This isomorphism is support-preserving, in the sense that it identifies the one-dimensional subspace
of $H(\GL_{m_i}(F_i), I)$ supported on $I w z I$, for $z \in Z_i$ and $w$ in the Weyl group of $\GL_{m_i}(F_i)$
with the subspace of $H_i$ supported on $K_P w z K_P$.  (This isomorphism depends on certain choices, but these
will be irrelevant for our purposes.)

In this way we find that $H(G,K_P,\St_{\overline{M},s})$ is a tensor product of Iwahori Hecke algebras $H_i$.
We will now relate this to $H(G,K,\St_s)$.  We have seen that $\cInd_K^G \St_s$ is a summand of
$\cInd_{K_P}^G \St_{\overline{M},s}$; in particular any central endomorphism of the latter commutes with
projection onto $\cInd_K^G \St_s$, and hence induces an endomorphism of $\cInd_K^G \St_s$.  We thus
have a natural map:
$$Z(H(G,K_P,\St_{\overline{M},s})) \rightarrow H(G,K,\St_s).$$
Moreover, since we can identify $H(G,K_P,\St_{\overline{M},s})$ with
$H(G,K,i_{\overline{P}}^{\overline{G}} \St_{\overline{M},s})$, this natural map
is support preserving, in the sense that an element of the domain supported on a union of double cosets
$K_P g_i K_P$ maps to an element of $H(G,K,\St_s)$ supported on the union of the double cosets $K g_i K$.

Moreover, we have:
\begin{proposition} \label{prop:depth zero center injectivity}
The map $Z(H(G,K_P,\St_{\overline{M},s})) \rightarrow H(G,K,\St_s)$ is injective.
\end{proposition}
\begin{proof}
As $\cInd_K^G \St_s$ is a direct summand of $\cInd_K^G i_{\overline{P}}^{\overline{G}} \St_{\overline{M},s}$,
any endomorphism of the former extends by zero to an endomorphism of the latter.  This allows us
to view $H(G,K,\St_s)$ as a $Z(H(G,K,i_{\overline{P}}^{\overline{G}} \St_{\overline{M},s})$-submodule 
of $H(G,K,i_{\overline{P}}^{\overline{G}} \St_{\overline{M},s})$; 
from this point of view the claim is that this submodule is not annihilated by any nonzero
element of $Z(H(G,K,i_{\overline{P}}^{\overline{G}} \St_{\overline{M},s}))$.  But 
$H(G,K,\tkappa \otimes i_{\overline{P}}^{\overline{G}} \St_{\overline{M},s})$ is
a tensor product of affine Hecke algebras; in particular (for instance, by Bernstein's presentation
of $H(q,n)$~\cite{Lu}), it is free over its center and
its center is a domain.  Thus {\em no} element of $H(G,K, i_{\overline{P}}^{\overline{G}} \St_{\overline{M},s})$ is
annihilated by any element of the center of $H(G,K,i_{\overline{P}}^{\overline{G}} \St_{\overline{M},s})$.
\end{proof}

In fact, we will show that this map is an isomorphism.  In light of the injectivity shown above,
and the fact that this map is support-preserving, this amounts to a dimension count. 

Let $z$ be an element of $Z$, and let $\tP_z$ be the intersection: $\GL_n(\OO_F) \cap z \GL_n(\OO_F) z^{-1}$.
Then the image of $\tP_z$ under the projection $\GL_n(\OO_F) \rightarrow \overline{G}$ is a parabolic subgroup
$\overline{P}_z$ of $\overline{G}$; this parabolic subgroup contains a standard Levi $\overline{M}_z$.

Since $G$ is the union of the double cosets $K z K$, for $z \in Z$, it suffices to compute
$H(G,K,\St_s)_{K z K} \cong I_z(\St_s)$ for all such $z$.  By definition, we have
$$I_z(\St_s) = \Hom_{\tP_z}(\St_s, \St_s^z) = \Hom_{\overline{P}_z}(\St_s, \St_s^z),$$
where the last equality follows by noting that the kernel of the map $\tP_z \rightarrow \overline{P}_z$
acts trivially on both $\St_s$ and $\St_s^z$. 
Let $\overline{U}_z$ be the unipotent radical of $\overline{P}_z$.  Then for any $u \in \tP_z$ that maps
to $\overline{U}_z$, the element $z^{-1} u z$ acts trivially on $\St_s$, and so $u$ acts trivially on
$\St_s^z$.  In particular any $\overline{P}_z$-equivariant map from $\St_s$ to $\St_s^z$
factors through the $\overline{U}_z$-coinvariants of $\St_s$; this space of coinvariants is
precisely $r_{\overline{G}}^{\overline{P}_z} \St_s$.  A parallel argument shows that the
image of this map lies in a submodule of $\St_s^z$ that is isomorphic to
$r_{\overline{G}}^{\overline{P}_z} \St_s$.  In this way an element of $I_z(\St_s)$
gives rise to an element of $\End_{\overline{M}_z}(r_{\overline{G}}^{\overline{P}_z}\St_s)$, and the resulting map
$$I_z(\St_s) \rightarrow \End_{\overline{M}_z}(r_{\overline{G}}^{\overline{P}_z} \St_s)$$
is an isomorphism.  

By Proposition~\ref{prop:steinberg restriction} we have a direct sum decomposition:
$$r_{\overline{G}}^{\overline{P}_z}(\St_s) \cong \bigoplus_w \St_{\overline{M}_z, w s w^{-1}},$$
where $w$ runs over a set of representatives for 
$W(\overline{M}_z) \backslash W(\overline{M},\overline{M}_z) / W_{\overline{M}}(s)$,
and $\St_{\overline{M}_z, w s w^{-1}}$ is isomorphic to $\St_{\overline{M}_z, w' s (w')^{-1}}$
if, and only if $w$ and $w'$ lie in the same double coset.  We thus find:

\begin{proposition} For $z \in Z$, the dimension of $H(G,K,\St_s)_{K z K}$ is equal
to the cardinality of $W(\overline{M}_z) \backslash W(\overline{M},\overline{M}_z) / W_{\overline{M}}(s)$.
\end{proposition}

It remains to compute the dimension of $Z(H(G,K_P,\St_{\overline{M},s}))$ supported on $K z K$.
As $H(G,K_P,\St_{\overline{M},s})$ is a tensor produce of Iwahori Hecke algebras, it is first useful to observe:

\begin{lemma} \label{lemma:Iwahori central support}
For any element $z$ of $Z'_i$, 
the subspace elements of $H(\GL_{m_i}(F_i),I)$ that are central and supported on the union of
the double cosets $I w z w' I$, for $w,w'$ in the Weyl group of $\GL_{m_i}(F_i)$,
is one-dimensional.  Moreover, the sum of these spaces as $z$ varies is the entire center
of this Hecke algebra.
\end{lemma}
\begin{proof}
We are grateful to an anonymous referee for pointing out that this is essentially the content of Corollary 3.1 
of~\cite{dat-bernstein}.  We include an alternative argument for completeness.

Let $J$ be the subgroup $\GL_{m_i}(\OO_{F_i})$ of $\GL_{m_i}(F_i)$.  Then the induction $\cInd_I^J 1$ contains
the trivial character of $J$, and so $\cInd_J^{\GL_{m_i}(F_i)} 1$
is a direct summand of $\cInd_I^{\GL_{m_i}(F_i)} 1 = \cInd_J^{\GL_{m_i}(F)} \cInd_I^J 1$.  In particular the center
of $H(\GL_{m_i}(F_i),I)$ preserves the summand $\cInd_J^{\GL_{m_i}(F_i)} 1$; this gives a support-preserving map
from the center of $H(\GL_{m_i}(F_i),I)$ to the spherical Hecke algebra $H(\GL_{m_i}(F_i),J)$.  (This map
simply gives the action of the center of the unipotent block on $\cInd_J^{\GL_{m_i}(F_i)} 1$.)
One verifies easily that the center of $H(\GL_{m_i}(F_i),1)$ acts faithfully on $\cInd_J^{\GL_{m_i}(F_i)} 1$
(this is the same injectivity argument as in the proof of the injectivity of $Z(H(G,K_P,\St_{\overline{M},s}))$
into $H(G,K,\St_s)$.)

On the other hand, the standard description of the action of the spherical Hecke algebra on the irreducible
quotients of $\cInd_J^{\GL_{m_i}(F_i)} 1$, together with Bernstein-Deligne's description of the action
of the center of $\Rep_{\overline{\CK}}(\GL_{m_i}(F_i))$ on these irreducibles, shows that for any element $x$ of
the spherical Hecke algebra, there is an element of the center of $\Rep_{\overline{\CK}}(\GL_{m_i}(F_i))$
whose action on $\cInd_J^{\GL_{m_i}(F_i)} 1$ coincides with that of $x$.  Since this element also
gives an element of $Z(H(\GL_{m_i}(F_i), I))$, we find that the map $Z(H(\GL_{m_i}(F_i),I)) \rightarrow H(\GL_{m_i}(F_i),J)$
is surjective, and hence an isomorphism.

Now the dimension of $H(\GL_{m_i}(F_i),J)_{J z J}$ is one for all $z$, so there is a one-dimensional subspace
of the center of $H(\GL_{m_i}(F_i),I)$ supported on $J z J$.  Since $J z J$ is the union of $I w z w' I$,
as $w$ and $w'$ range over elements of the Weyl group of $\GL_{m_i}(F_i)$, the remaining claims follow.
\end{proof}

Our support preserving isomorphism of $H(G,K_P,\St_{\overline{M},s})$ with the tensor product of the spaces
$H(\GL_{m_i}(F_i),I)$ shows that for any $z$, there is a one-dimensional subspace of $Z(H(G,K_P,\St_{\overline{M},s})$
supported on the union of double cosets $K_P w z w' K_P$, for $w$ and $w'$ in $W_{\overline{M}}(s)$, and that the full
center is the sum of these subspaces.  Fix $z$, and let $S_z$ be the the set of $z' \in Z$ such that $K z' K = K z K$.
For any such $z'$ we have $z' = w^{-1} z w$ for some $w$ in the Weyl group of $G$.  On the other hand, if $w z w^{-1}$
is an element $z'$ of $Z$, then $w^{-1} M_z w = M_{z'} \subseteq M$, so $w$ lies in $W(\overline{M},\overline{M}_z)$.
We thus have a surjection $W(\overline{M},\overline{M}_z) \rightarrow S_z$, defined by $w \mapsto w^{-1} z w$.  This
surjection descends to a bijection of $W(\overline{M}_z) \backslash W(\overline{M},\overline{M}_z)$ with $S_z$.
Define an equivalence relation on $S_z$ by setting $z' \sim z''$ if the collection of cosets $K_P w z' w' K_P$ ($w,w' \in W_{\overline{M}}(s)$)
coincides with the collection $K_P w z'' w' K_P$ ($w,w' \in W_{\overline{M}}(s)$).  
Then $z' \sim z''$ if, and only if, $K_P z' K_P = K_P w z'' w' K_P$, for some $w$ and $w'$ in $W_{\overline{M}}(s)$.
This happens if, and only if, we have $z' = w z'' w^{-1}$ for some $w$ in $W_{\overline{M}}(s)$.

We thus obtain:

\begin{proposition} \label{prop:cosets}
The map $w \mapsto w^{-1} z w$ induces a bijection between: i
$$W(\overline{M}_z) \backslash W(\overline{M},\overline{M}_z) / W_{\overline{M}}(s) \rightarrow S_z / \sim.$$
\end{proposition}

\begin{corollary}
The map $Z(H(G,K_P,\St_{\overline{M},s})) \rightarrow H(G,K,\St_s)$ is an isomorphism.
\end{corollary}
\begin{proof}
We have shown that this map is an injection, and that the dimension of the subspace of the target supported on $K z K$
is equal to the cardinality of $W(\overline{M}_z) \backslash W(\overline{M},\overline{M}_z) / W_{\overline{M}}(s)$.
On the other hand, for each equivalence class $[z']$ in $S_z$, we have a one-dimensional subspace of $Z(H(G,K_P,\St_{\overline{M},s}))$
supported on the union of cosets $K_P w z' w' K_P$; for distinct equivalence classes these spaces have disjoint supports contained in $K z K$.  
Thus the dimension of $Z(H(G,K_P,\St_{\overline{M},s}))_{K z K}$ is at least the cardinality of $S_z$, and the result is immediate
from Proposition~\ref{prop:cosets}.
\end{proof}

On the other hand, we have shown that the map $\Hom_{K_P}(\St_{\overline{M},s}, -)$ is an equivalence of categories
between $\Rep_{\overline{\CK}}(G)_{M,\pi}$ and the category of $H(G,K_P,\St_{\overline{M},s})$-modules.  It follows that
the map $A_{M,\pi} \rightarrow H(G,K_P,\St_{\overline{M},s})$ giving the action of $A_{M,\pi}$ on $\cInd_{K_P}^G \St_{\overline{M},s}$
identifies $A_{M,\pi}$ with the center of $H(G,K_P,\St_{\overline{M},s})$.  From this point of view, the corollary above
asserts that the map $A_{M,\pi} \rightarrow H(G,K,\St_s)$ giving the action of $A_{M,\pi}$ on $\cInd_K^G \St_s$ is an isomorphism.

\section{Endomorphisms of $\CP_{K,\tau}$ in depth zero} \label{sec:zero endomorphisms}

Our next objective is to refine the results of the previous section, in order to obtain results that hold
over $W(k)$ rather than $\overline{\CK}$.  We retain the assumption that we are in depth zero, so that $K$ is a maximal
compact and $\tau$ is inflated from a cuspidal representation of the $k$-points of a general linear group.

For our purposes it will be necessary to work inductively, with a sequence of cuspidal types.  Fix an integer $n_1$ and 
let $\overline{G}_1$ be the group $\GL_{n_1}(k)$.  If we fix a supercuspidal representation $\sigma_1$ of $\overline{G}_1$,
corresponding to an irreducible $\ell$-regular conjugacy class $s'_1$ in $\overline{G}_1$
then, for each $m$ in the set $\{1, e_q, \ell e_q, \ell^2 e_q, \dots\}$ (where $e_q$ is the order of $q$ mod $\ell$) we have
a cuspidal representation $\sigma_m$ of $\overline{G}_m := \GL_{n_1m}(k)$ corresponding to the conjugacy class $(s'_1)^m$ in
$\overline{G}_m$. 

For each $m$, we obtain a cuspidal type $(K_m,\tau_m)$ in $G_m := \GL_{n_1m}(F)$ by setting $K_m = \GL_{n_1m}(\OO_F)$ and taking
$\tau_m$ to be the inflation of $\sigma_m$ to $K_m$.  Then if we write $\CP_m$ for the projective envelope of $\sigma_m$, we
have $\CP_{K_m,\tau_m} = \cInd_{K_m}^{G_m} \CP_m.$ 

Let $E_m$ denote the endomorphism ring of $\CP_{K_m,\tau_m}$.  The isomorphism:
$$\CP_m \otimes \overline{\CK} \cong \bigoplus_s \St_s,$$
where $s$ runs over the conjugacy classes in $\overline{G}_m$ with $\ell$-regular part $(s'_1)^m$, induces an isomorphism:
$$\CP_{K_m,\tau_m} \otimes \overline{\CK} \cong \bigoplus_s \cInd_{K_m}^{G_m} \St_s.$$

By the previous section, the summand $\cInd_{K_m}^{G_m} \St_s$ lies in $\Rep_{\overline{\CK}}(G_m)_{M_s,\pi_s}$, and the pairs
$(M_s,\pi_s)$ are not inertially equivalent for distinct $s$.  Thus any endomorphism of $\CP_{K_m,\tau_m}$ preserves each of the summands
$\cInd_{K_m}^{G_m} \St_s$, and we have an isomorphism:
$$E_m \otimes \overline{\CK} \cong \prod_s A_{M_s,\pi_s},$$
where for each $s$, $\pi_s$ is a supercuspidal representation of $M_s$ containing the type $(K_{M_s},\St_{\overline{M}_s,s})$.

We will inductively construct a family of endomorphisms $\Theta_{1,m}, \dots, \Theta_{m,m}$ of $\CP_{K_m,\tau_m}$.  These endomorphisms
will be characterized by their images in $A_{M_s,\pi_s}$ for each $s$.  We describe the rings $A_{M_s,\pi_s}$ more concretely as follows:
fix a uniformizer $\unif$ of $F$, and let $Z_s$ denote the subgroup of $M_s$ consisting of elements of the center of $M_s$ whose characteristic
polynomial has the form $(t - 1)^a(t - \unif)^b$ for integers $a$ and $b$.  We then have a map $\overline{\CK}[Z_s] \rightarrow \overline{\CK}[M_s/(M_s)^0]$ that takes
an element of $Z_s$ to its class modulo $(M_s)^0$, and its image is the subalgebra $\overline{\CK}[M_s/(M_s)^0]^{H_s}$, where $H_s$ is the group of 
unramified characters $\chi$ of $M_s$ such that $\pi_s \otimes \chi$ is isomorphic to $\pi_s$.

A particular choice of $\pi_s$ gives rise to an isomorphism: 
$$A_{M_s,\pi_s} \cong \left(\overline{\CK}[M_s/(M_s)^0]^{H_s}\right)^{W_{M_s}(s)}$$
as in~\ref{thm:B-D presentation}.
It will therefore be important to have a systematic way of choosing the $\pi_s$.  In depth zero this is rather straightforward:
fix a uniformizer $\unif$ of $F$.  Then for a given $m$, and any irreducible conjugacy class $s$ in $\overline{G}_m$ with $\ell$-regular part $(s_0)^m$,
we extend $\St_s$ to a representation $\Lambda_s$ of $F^{\times}K_m$ by letting $\unif$ act via the identity.  Then
$\cInd_{F^{\times}K_m}^{G_m} \Lambda_s$ is an irreducible supercuspidal representation of $G_m$, which we denote by $\pi_{m,s}$.
We call the collection $\{ \pi_{m,s} \}$ the {\em compatible family of cuspidals} attached to $s'_1$ and our choice of $\unif$.

For an arbitrary $s$ with $\ell$-regular part $(s'_1)^m$, we let $s_1, \dots, s_r$ be the irreducible factors of $s$, of block sizes $n_1m_1, \dots, n_1m_r$,
and let $M_s$ be standard Levi of $\GL_{n_0m}(F)$ with block sizes $n_1m_1, \dots, n_1m_r$, and take $\pi_{m,s}$ to be the tensor product
of the $\pi_{m_i,s_i}$.  This fixes, for all $s$, a corresponding isomorphism:
$$A_{M_s,\pi_s} \cong \left(\overline{\CK}[M_s/(M_s)^0]^{H_s}\right)^{W_{M_s}(s)}.$$

We next define distinguished elements of $A_{M_s,\pi_s}$:
\begin{definition} 
For $1 \leq i \leq m$, and $s$ a semisimple conjugacy class in $\overline{G}_m$ with $\ell$-regular part $(s'_1)^m$, let $\theta_{i,s}$
be the image, in $\overline{\CK}[M_s/(M_s)^0]$ of the sum of the elements $z$ in $Z_s$ with characteristic polynomial $(t-1)^{n_1(m-i)}(t-\unif)^{n_1i}$.
(Note that this sum may be empty, in which case $\theta_{i,s} = 0$.)
\end{definition}

The elements $\theta_{i,s}$ are clearly invariant under the actions of both $H_s$ and $W_{M_s}(s)$, and we may thus regard them as elements of $A_{M_s,\pi_s}$
via our chosen isomorphism.  This allows us to give a characterization of the endomorphisms $\Theta_{i,m}$.
\begin{theorem} \label{thm:depth zero Theta}
For $i \leq i \leq m$, there exists a unique endomorphism $\Theta_{i,m}$ in $E_m$ such that for all $s$ with $\ell$-regular part $(s'_1)^m$, the composed map
$$E_m \rightarrow E_m \otimes \overline{\CK} \rightarrow A_{M_s,\pi_s}$$
takes $\Theta_{i,m}$ to $\theta_{i,s}$.  Moreover, $\Theta_{m,m}$ is invertible.
\end{theorem}

It is clear that such a $\Theta_{i,m}$ is unique if it exists; the construction of the $\Theta_{i,m}$ will occupy the remainder of the section.

The construction of $\Theta_{m,m}$ for any $m$ is straightforward: it is simply the endomorphism of $\cInd_{K_m}^{G_m} \CP_m$ given by the action
of the central element $\unif$ of $G_m$.  This is clearly invertible, and its image in $A_{M_s,\pi_s}$ is equal to $\theta_{m,s}$.

The $\Theta_{i,m}$ for $i < m$ will be constructed by an inductive argument.  Let $m' < m$ be two consecutive elements of the set
$\{1, e_q, \ell e_q, \ell^2 e_q, \dots \}$, and set $j = \frac{m}{m'}$, so that $j = \ell$ or $j = e_q$.  Let $\overline{M}$ be the standard Levi
of $\overline{G}_m$ given by $j$ blocks of size $m'$, and let $\overline{P}$ be the standard (upper triangular) parabolic with Levi $\overline{M}$.
Further let $\CP_{\overline{M}}$ be the representation of $\CP_{m'}^{\otimes j}$ of $\overline{M}$,
and let $K'$ be the preimage of $\overline{P}$ in $K_m$.  Also let $M$ denote the standard Levi of $G_m$ given by $j$ blocks of size $m'$ and
let $K_M$ be the maximal compact subgroup $K' \cap M$ of $M$.  We may regard $\CP_{\overline{M}}$ as a representation of $K'$ or of $K_M$.
The key point in our inductive argument is then:

\begin{theorem} \label{thm:depth zero G-cover}
The pair $(K',\CP_{\overline{M}})$ is a $G$-cover of $(K_M,\CP_{\overline{M}})$.
\end{theorem}

The proof of this is somewhat complicated and will be postponed to the end of the section.  An immediate consequence is that we have a sequence
of maps:
$$E_{m'}^{\otimes j} \cong H(M,K_M,\CP_{\overline{M}}) \rightarrow H(G_m,K',\CP_{\overline{M}}) \cong 
H(G_m,K_m,i_{\overline{P}}^{\overline{G}_m} \CP_{\overline{M}})$$
where $\overline{P}$ is the standard (block upper triangular) parabolic with Levi $\overline{M}$.  Our first step is to compare the
latter with $E_m$.

By Proposition~\ref{prop:projective restriction}, we have an isomorphism: $r_{\overline{G}_m}^{\overline{P}_M} \CP_m \cong \CP_{\overline{M}}$.
Frobenius reciprocity then gives us a map:
$$i_{\overline{P}}^{\overline{G}_m} \CP_{\overline{M}} \rightarrow \CP_m.$$
Let $\CP'_m$ be the image of this map.

\begin{lemma} \label{lemma:cuspidal discrepancy}
The quotient $\CP_m/\CP'_m$ is cuspidal.  Moreover, every endomorphism of $\CP_m$ preserves $\CP'_m$.
\end{lemma}
\begin{proof}
Note that the composition:
$$\CP_{\overline{M}} \rightarrow r_{\overline{G}_m}^{\overline{P}} i_{\overline{P}}^{\overline{G}_m} \CP_{\overline{M}} \rightarrow 
r_{\overline{G}_m}^{\overline{P}} \CP_m$$
is an isomorphism by construction.  In particular the map
$$r_{\overline{G}_m}^{\overline{P}} i_{\overline{P}}^{\overline{G}_m} \CP_{\overline{M}} \rightarrow r_{\overline{G}_m}^{\overline{P}} \CP_m$$
is surjective, and factors through $r_{\overline{G}_m}^{\overline{P}} \CP'_m$.  The inclusion of $\CP'_m in \CP_m$ thus inducts
an isomorphism:
$$r_{\overline{G}_m}^{\overline{P}} \CP'_m \cong r_{\overline{G}_m}^{\overline{P}} \CP_m.$$

Now let $\pi$ be a Jordan-H\"older constituent of $\CP_m/\CP'_m$.  Then $r_{\overline{G}_m}^{\overline{P}} \pi = 0$
by the above calculation.  Suppose $\pi$ were not cuspidal.  Then it would have cuspidal support $(\overline{M}',\pi')$ for some
proper Levi subgroup $\overline{M}'$ of $\overline{G}_m$, with $\pi'$ a product of cuspidal representations $\sigma_{m''}$, with
$m'' <  m$.  In particular $\overline{M}'$ would then be conjugate to a subgroup of $\overline{M}$, and $r_{\overline{G}_m}^{\overline{P}} \pi$
could not be trivial.

Let $f$ be an endomorphism of $\CP_m$, and suppose $f$ does not preserve $\CP'_m$.  Then the image of $f(\CP'_m)$ in $\CP_m/\CP'_m$ is nonzero,
so we have a nonzero map from $\CP'_m$ to a cuspidal representation of $\overline{G_m}$.  This induces a nonzero map
of $i_{\overline{P}}^{\overline{G_m}} \CP_{\overline{M}}$ to a cuspidal representation of $\overline{G_m}$, which is impossible.
\end{proof}

Our next step is to compare $H(G_m,K_m,\CP_m)$ with $H(G_m,K_m,\CP'_m)$.  Let $Z_m$ be the subgroup of $G_m$ consisting of diagonal
matrices whose entries are powers of $\unif$.  Then $G_m$ is the union of double cosets $K_m z K_m$ for $z \in Z_m$.  Moreover, for any $z$, 
we have isomorphisms:
$$H(G_m, K_m, \CP_m)_{K_m z K_m} \cong I_z(\CP_m) \cong \End_{\overline{M}_z}(r_{\overline{G}_m}^{\overline{P}_z} \CP_m)$$
$$H(G_m, K_m, \CP'_m)_{K_m z K_m} \cong I_z(\CP'_m) \cong \End_{\overline{M}_z}(r_{\overline{G}_m}^{\overline{P}_z} \CP'_m)$$
where $\overline{P}_z$ is the image in $\overline{G}_m$ of $K_m \cap z K_m z^{-1}$, and $\overline{M}_z$ is the associated Levi.

If $z$ is not central in $G_m$, then $\overline{M}_z$ is a proper subgroup of $\overline{G}_m$.  Hence parabolic restriction to
$\overline{M}_z$ annihilates the cokernel of the map $\CP'_m \rightarrow \CP_m$, and so this inclusion induces an isomorphism
of $r_{\overline{G}_m}^{\overline{P}_z} \CP'_m$ with $r_{\overline{G}_m}^{\overline{P}_z} \CP_m$.  By contrast, if $z$ is central,
then $H(G_m,K_m,\CP_m)_{K_m z K_m}$ is given by the endomorphism ring of $\CP_m$, and hence, by Lemma~\ref{lemma:cuspidal discrepancy}
maps naturally to the endomorphism ring of $\CP'_m$ (and thus to $H(G_m,K_m,\CP'_m)$.)  Combining these maps for all $z$ (central or otherwise)
we find:

\begin{prop} \label{prop:surjection}
There is a surjection:
$$H(G_m,K_m,\CP_m) \rightarrow H(G_m,K_m,\CP'_m)$$
compatible with the inclusion of $\cInd_{K_m}^{G_m} \CP'_m$ in $\cInd_{K_m}^{G_m} \CP_m$ and the actions of the respective Hecke algebras
on these spaces.  Moreover, this map is an isomorphism away from double cosets of the form $K_m z K_m$ for $z$ central in $G_m$.
\end{prop}

Our next objective will be to compare $H(G_m,K_m,\CP'_m)$ with $H(G_m,K_m,i_{\overline{P}}^{\overline{G}_m} \CP_{\overline{M}})$.
By construction there is a surjection:
$$i_{\overline{P}}^{\overline{G}_m} \CP_{\overline{M}} \rightarrow \CP'_m$$
and this surjection induces a surjection of $\cInd_{K_m}^{G_m} i_{\overline{P}}^{\overline{G}_m}$ onto $\cInd_{K_m}^{G_m} \CP'_m$.
We then have:

\begin{lemma} \label{lemma:depth zero center}
Let $c$ be a central element of $H(G_m,K_m,i_{\overline{P}}^{\overline{G}_m} \CP_{\overline{M}})$.  Then $c$ preserves the
kernel of the map: 
$$\cInd_{K_m}^{G_m} i_{\overline{P}}^{\overline{G}_m} \rightarrow \cInd_{K_m}^{G_m} \CP'_m$$
and therefore descends to an element of $H(G_m,K_m,\CP'_m)$ with the same support as $c$.
\end{lemma}
\begin{proof}
After inverting $\ell$, $\CP'_m$ becomes a direct summand of $i_{\overline{P}}^{\overline{G}_m} \CP_{\overline{M}}$,
and so the module $\cInd_{K_m}^{G_m} \CP'_m \otimes \CK$ is a direct summand of
$\cInd_{K_m}^{G_m} i_{\overline{P}}^{\overline{G}_m} \CP_{\overline{M}} \otimes \CK$.  The element $c$
commutes with projection onto this summand, and so preserves the kernel of this projection.  On the other hand,
$\cInd_{K_m}^{G_m}  \CP'_m$ is $\ell$-torsion free, and so the kernel of the map
from $\cInd_{K_m}^{G_m} i_{\overline{P}}^{\overline{G}_m} \CP_{\overline{M}}$ onto
$\cInd_{K_m}^{G_m} \CP'_m$ consists of those elements of the kernel of the projection
that lie in $\cInd_{K_m}^{G_m} i_{\overline{P}}^{\overline{G}_m} \CP_{\overline{M}}$.  If $x$ is such an element,
it is clear that $cx$ is as well.
\end{proof}

We can now sketch an inductive construction of the $\Theta_{i,m}$.  Assume that for each $1 \leq i \leq m'$, we have
constructed elements $\Theta_{i,m'}$ as in Theorem~\ref{thm:depth zero Theta}.  Assume further that for $i < m'$
the element $\Theta_{i,m'}$ of $H(G_{m'}, K_{m'}, \CP_{m'})$ is supported away from double cosets of the form $K_{m'} z K_{m'}$
for $z \in Z_{m'}$ central.  Then we can construct the $\Theta_{i,m}$ for $i < m'$, as follows:

First, let $\tTheta_{i,m}$ be the element of $E_{m'}^{\otimes j}$ by the formula:
$$\tTheta_{i,m} = \sum\limits_{r_1 + \dots + r_j = i} \Theta_{r_1,m'} \otimes \dots \otimes \Theta_{r_j,m'},$$
where $0 \leq r_k \leq i$ for all $k$ and by convention $\Theta_{0,m'} = 1$.  For any $s_1, \dots, s_j$, each with $\ell$-regular
part $(s'_1)^{m'}$, the maps $E_{m'} \rightarrow \overline{\CK}[Z_{s_k}]$ takes $\Theta_{i,m'}$ to $\theta_{i,s_k}$.  Taking the
tensor product of each of these maps gives, for $s$ the conjugacy class in $\overline{G}_m$ with ``blocks'' $s_k$, a map:
$$E_{m'}^{\otimes j} \rightarrow \overline{\CK}[Z_s]$$
that takes $\tTheta_{i,m}$ to $\theta_{i,s}$.

Let $\overline{\Theta}_{i,m}$ be the image of $\tTheta_{i,m}$ under the map
$E_{m'}^{\otimes j} \rightarrow H(G_m,K'_m, \CP_{\overline{M}})$ coming from Theorem~\ref{thm:depth zero G-cover}.  Note that,
for $i < m$, the element $\overline{\Theta}_{i,m}$ is supported away from double cosets of the form $K_m z K_m$ with $z$ central in $G_m$.
(This is because $\overline{\Theta}_{i,m} = u T^+(\tTheta_{i,m}) u^{-1}$ for an invertible element $u$ of
$H(G_m,K_m,i_{\overline{P}}^{\overline{G_m}} \CP_{\overline{M}})$.  Since for $z$ central, $z$ normalizes $K_m$, so
conjugation by $u$ preserves $H(G_m,K_m,i_{\overline{P}}^{\overline{G}_m})_{K_m z K_m}.$  Since $T^+(\tTheta_{i,m})$ has no
support on $K_m z K_m$, neither can $\overline{\Theta}_{i,m}$.)

It now suffices to show that $\overline{\Theta}_{i,m}$ lies in the center of $H(G_m,K'_m,\CP_{\overline{M}})$.
Suppose this is the case.  Then $\overline{\Theta}_{i,m}$ gives rise to an element of $H(G_m,K_m,\CP'_m)$ supported away from
$K_m z K_m$ for $z$ central.  There is thus a unique lift of the image of $\overline{\Theta}_{i,m}$ to an element 
$\Theta_{i,m}$ of $H(G_m,K_m,\CP_m)$ supported away from $K_m z K_m$ for $z$ central.  

Now let $s$ be a semisimple conjugacy class in $\overline{G}_m$ with $\ell$-regular part $(s')^{m}$.  If $s$ is irreducible,
then $\theta_{i,s} = 0$ for $i < m$, and any element of $H(G_m,K_m,\CP_m)$ supported away from the cosets $K_m z K_m$ for
$z$ central maps to zero in $H(G_m,K_m,\St_s)$.  Thus $\Theta_{i,m}$ maps to $\theta_{i,s}$ for such $s$.

On the other hand, if $s$ is not irreducible, we may choose semisimple conjugacy classes $s_1, \dots, s_j$ in $G_{m'}$
such that $s$ is conjugate to the element of $G_m$ with ``blocks'' $s_1, \dots, s_j$, and let $\St_{\overline{M}, \vec{s}}$
denote the tensor product of the representations $\St_{s_i}$, considered as a representation of $\St_{\overline{M}}$.
We have a map $i_{\overline{P}}^{\overline{G}_m} \St_{\overline{M}, \vec{s}} \rightarrow \St_s$, and the action of
$\Theta_{i,m}$ on $\cInd_{K_m}^{G_m} \St_s$ is compatible, via this map, with the action of $\overline{\Theta}_{i,m}$ on
the summand $\cInd_{K_m}^{G_m} i_{\overline{P}}^{\overline{G}_m} \St_{\overline{M}, \vec{s}}$ of
$\cInd_{K_m}^{G_m} i_{\overline{P}}^{\overline{G}_m} \CP_M \otimes \overline{\CK}$. 

On the other hand, because the maps of Hecke algebras corresponding to $G$-covers are compatible with parabolic induction, the
action of $\overline{\Theta}_{i,m}$ on the compact induction $\cInd_{K_m}^{G_m} i_{\overline{P}}^{\overline{G}_m} \St_{\overline{M}, \vec{s}}$
is compatible with the action of $\tTheta_{i,m}$ on $\cInd_{K_M}^M \St_{\overline{M}, \vec{s}}$, in the following sense:
We have commutative diagram:
$$
\begin{array}{ccc}
E_{m'}^{\otimes j} & \rightarrow & \bigotimes_k H(G_{m'}, K_{m'}, \St_{s_k})\\
\downarrow & & \downarrow\\
H(G_m,K_m,i_{\overline{P}}^{\overline{G}_m} \CP_{M'}) & \rightarrow & H(G_m,K_m, i_{\overline{P}}^{\overline{G}} \bigotimes_k \St_{s_k}).
\end{array}
$$

Recall that we have a fixed identifications of $A_{M_s,\pi_s}$ with a subalgebra of $\overline{\CK}[Z_s]$.  Similarly we may
identify the center of $\Rep_{\overline{\CK}}(G_{m'})_{M_{s_k},\pi_{s_k}}$ with a subalgebra of $\overline{\CK}[Z_{s_k}]$,
and $Z_s$ is isomorphic to the product of the $Z_{s_k}$.  We then have a commutative diagram:
$$
\begin{array}{ccc}
A_{M_s,\pi_s} & \rightarrow & \overline{\CK}[Z_s]\\
\downarrow & & \downarrow\\
\bigotimes_k A_{M_{s_k},\pi_{s_k}} & \rightarrow & \bigotimes_k \overline{\CK}[Z_{s_k}]
\end{array}
$$
where the left-hand vertical map is given by Proposition~\ref{prop:Bernstein induction}.

The element $\theta_{i,s}$ of $\overline{\CK}[Z_s]$ corresponds to an element of $A_{M_s,\pi_s}$ and this acts
on $\cInd_{K_m}^{G_m} i_{\overline{P}}^{\overline{G}_m} \St_{\overline{M}, \vec{s}}$; the above commutative diagram
shows that this action is induced, via the map $T$, from the action of $\theta_{i,s}$ (considered as an element of
$\bigotimes_k A_{M_{s_k},\pi_{s_k}}$) on $\cInd_{K_M}^M \bigotimes_k \St_{s_k}$.  Since the action of $\theta_{i,s}$
coincides with that of $\tTheta_{i,m}$, and $T(\tTheta_{i,m})$ is equal to $\overline{\Theta}_{i,m}$ by definition,
we conclude that the action of $\overline{\Theta}_{i,m}$ on 
$\cInd_{K_m}^{G_m} i_{\overline{P}}^{\overline{G}_m} \St_{\overline{M}, \vec{s}}$
coincides with $\theta_{i,s}$.  Together with the compatibility of $\overline{\Theta}_{i,m}$ and $\Theta_{i,m}$,
this shows that $\Theta_{i,m}$ maps to $\theta_{i,s}$ for all $s$.

It thus remains to show that the element $\overline{\Theta}_{i,m}$ is central in $H(G_m,K'_m,\CP_{\overline{M}})$.
Since $\cInd_{K'_m}^{G_m} \CP_{\overline{M}}$ is $\ell$-torsion free, it suffices to prove this after tensoring
with $\overline{\CK}$.  We will show that the action of $\overline{\Theta}_{i,m}$ coincides with the action
of an element of the center of $\Rep_{\overline{\CK}}(G_m)$, proving the claim.

Indeed, let $x_i$ be an element of this center that acts via $\theta_{i,s}$ on the block
$\Rep_{\overline{\CK}}(G_m)_{M_s,\pi_s}$
for each semisimple conjugacy class $s$ with $\ell$-regular part $(s'_1)^m$.  Since the pairs $(M_s,\pi_s)$
give distinct blocks for distinct conjugacy classes $s$, such elements $x_i$ exist.

On the other hand, we have a direct sum decomposition:
$$\cInd_{K_m}^{G_m} i_{\overline{P}}^{\overline{G}_m} \CP_{\overline{M}} \otimes \CK \cong 
\bigoplus\limits_{s_1, \dots, s_j} \cInd_{K_m}^{G_m} i_{\overline{P}}^{\overline{G}_m} \otimes_k \St_{s_k}$$
and (for $s$ the semisimple conjugacy class given by a block diagonal matrix whose blocks are the $s_k$) the action
of $A_{M_s,\pi_s}$ on the summand corresponding to the $s_k$ is given by $\theta_{i,s}$.  Thus the action of
$\overline{\Theta}_{i,m}$ on $\cInd_{K_M}^{G_m} i_{\overline{P}}^{\overline{G}_m} \CP_{\overline{M}}$ coincides with
that of $x_i$, and is in particular central.  Our construction of the $\Theta_{i,m}$ is thus complete, modulo the proof
of Theorem~\ref{thm:depth zero G-cover}, which will occupy the remainder of this section.

\begin{lemma} \label{lemma:G-cover reduction}
Let $P = MU$ be a parabolic subgroup of $G_m$ with Levi $M$, and let $(K,\tau)$ be a pair consisting of
a compact open subgroup $K$ of $G_m$ and a $W(k)[K]$-module $\tau$ that is decomposed with respect to $P$.  Set $K_M = K \cap M$,
and let $\tau_M$ be the restriction of $\tau$ to $K_M$.  
Suppose that there are positive central elements $\lambda_1, \dots, \lambda_r$ of
$M$, whose product is strictly positive, and elements $x_1, \dots, x_r$ of $H(M,K_M,\tau_M)$ such that for each $i$,
$x_i$ is supported on $\lambda_i K_M$, and $T^+x_i$ is invertible in
$H(G_m,K,\tau)$.  Then $(K,\tau)$ is a $G_m$-cover of $(K_M,\tau_M)$.
\end{lemma}
\begin{proof}
Let $\lambda$ be the product of the $\lambda_i$, and $x$ the product of the $x_i$.  Then $x$ is supported on $K_m \lambda K_m$,
and $T^+ x$ is the product of the $T^+ x_i$, hence invertible.
We must show that there is a strictly positive central element $\lambda'$ of $M$ such that $1_{K_m \lambda' K_m}$ maps,
via $T^+$, to an invertible element of $H(G_m,K,\tau)$.  Fix any strictly positive central $\lambda'$ in $M$.  Then for
some sufficiently large $r$, $\lambda^r (\lambda')^{-1}$ is strictly positive.  We have:
$$T^+ [(x^r) 1_{K_m (\lambda')^{-1} K_m}] T^+ [1_{K_m \lambda' K_m}] = T^+ x^r,$$
and since $T^+ x^r$ is invertible one must have $T^+ [1_{K_m \lambda' K_m}]$ invertible as well.
\end{proof}

For $1 \leq i \leq j$, let $V_i$ be the span of the basis vectors $e_{m'(i-1) + 1}, \dots, e_{m'i}$ in $F^n$, so that
$M$ is the Levi of $G_m$ preserving the $V_i$.  Let $z$ be the central element of $M$ that acts by multiplication
by $\unif$ on $V_1$ and the identity on all other $V_i$, and let $w$ be the permutation matrix that maps
$e_i$ to $e_{(j-1)m' + i}$ for $1 \leq i \leq m'$ and $e_i$ to $e_{i-m'}$ for $m'+1 \leq i \leq jm'.$  Set $\Pi = wz$.
Then $\Pi$ normalizes $K'$.  Thus, if we let $\beta$ denote the automorphism of $\CP_{\overline{M}} = \CP_{m'}^{\otimes j}$ that 
cyclically permutes the tensor factors, we have a unique element $y$ of $H(G_m,K',\CP_{\overline{M}})$ supported on $K' \Pi K'$
that takes the value $\beta$ at $\Pi$.  Moreover, left multiplication by $y$ induces an isomorphism:
$$H(G_m,K',\CP_M)_{K' x K'} \rightarrow H(G_m,K',\CP_M)_{K' \Pi x K'}$$
for any $x$.

For $1 \leq i \leq j$, let $\lambda_i = w^{-i} \Pi^i$.  Then $\lambda_i$ is a central, positive element of $M$
(but is not strictly positive).  Let $x_i$ be the element of $H(M,K_M,\CP_{\overline{M}})$ giving
the action of $\lambda_i$ on $\cInd_{K_M}^M \CP_{\overline{M}}$.  Then $x_i$ is supported on $K_M \lambda_i K_M$.
It thus suffices to show that for each $i$, $T^+ x_i$ invertible in
$H(G_m,K',\CP_{\overline{M}})$.  

Equivalently, it suffices to show that the element $v_i$ defined by $v_i = T^+(x_i) y^{-i}$
is invertible.  Note that $v_i$ is supported on $K' w^{-i} K' \subset K$.  It follows that the
endomorphism $v_i$ of $\cInd_{K'}^{G_m} \CP_{\overline{M}} \cong \cInd_K^{G_m} i_{\overline{P}}^{\overline{G}_m} \CP_{\overline{M}}$
is induced by an endomorphism $\overline{v}_i$ of $i_{\overline{P}}^{\overline{G}_m} \CP_{\overline{M}}$.  We will show that
$\overline{v}_i$ is an isomorphism.

Indeed, unwinding the definition of $v_i$, we see that when considered as an element of the Hecke algebra
$H(\overline{G}_m, \overline{P}, \CP_{\overline{M}})$, the support of $\overline{v}_i$ is on the double coset
$\overline{P} w^{-i} \overline{P}$, and $\overline{v}_i$ is the unique element supported on this double coset that is given
at $w^{-i}$ by the isomorphism of $\CP_{\overline{M}}$ with $\CP_{\overline{M}}^{w^{-i}}$ that permutes the tensor factors. 

Note that the map $\overline{U} \rightarrow (\overline{P} \cap w^i \overline{P} w^{-i}) \setminus \overline{P}$ is bijective.
Thus the double coset $\overline{P} w^{-i} \overline{P}$ decomposes as the disjoint union of cosets $\overline{P} w^{-i} u$
as $u$ runs over the elements of $\overline{U}$.  It follows that, if we regard elements of $i_{\overline{P}}^{\overline{G}} \CP_{\overline{M}}$
as functions from $G$ to $\CP_{\overline{M}}$, then the action of $\overline{v}_i$ on a function $f$ this space is given (up to a power of $p$)
by first 
applying the ``permute the tensor factors'' automorphism of $\CP_{\overline{M}}$ to the values of $f$, and then
summing over all the translates of this new function by elements of the form $w^{-i} u$, for $u \in \overline{U}$.  This is precisely
the map described in Theorem 2.4 of~\cite{howlett-lehrer}, where it is shown to be an isomorphism.  Thus $v_i$ is an isomorphism for all $i$,
completing the proof.

\section{Generic pseudo-types} \label{sec:generic}

In this section and the next we will extend the techniques and arguments of the previous two sections to
arbitrary generic pseudo-types.  In every case the final results we obtain will be direct analogues of those from
the depth zero case, but there are many more technicalities that need to be addressed.  

Let $(K,\tau)$ be a maximal distinguished cuspidal $k$-type, with $\tau$ of the form $\kappa \otimes \sigma$.  For
suitable choice of $s$ we then obtain a generic pseudo-type $(K,\tkappa \otimes \St_s)$.  Attached to the type $(K,\tau)$
is a finite extension $E$ of $F$, of ramification index $e$ and residue class degree $f$ over $F$, and an embedding
$\GL_{\frac{n}{ef}}(E)$ into $G$ such $K$ contains $\GL_{\frac{n}{ef}}(\OO_E)$ and the chain of maps:
$$\GL_{\frac{n}{ef}}(\OO_E) \rightarrow K \rightarrow K/K_1$$
identifies $K/K_1$ with the quotient $\overline{G} = \GL_{\frac{n}{ef}}(\FF_{q^f})$ of $\GL_{\frac{n}{ef}}(\OO_E)$.

We will use these identifications to construct analogues of the Levi subgroups $M$ and $\overline{M}$ of
section~\ref{sec:zero generic}.  First,
let $\overline{M}$ be the minimal split Levi of $\overline{G}$ containing $s$.
We will make a choice of Levi subgroup
$M$ of $G$ depending on $\overline{M}$, as follows: let $V$ be an $n$-dimensional $F$-vector
space on which $G$ acts via the standard representation, so that $G \cong \GL(V)$.  
The distinguished subgroup $\GL_{\frac{n}{ef}}(E)$ of $G$ coming from
the type $(K,\tau)$ gives $V$ the structure of an $E$-vector space.  Let $L$ be an $\OO_E$-lattice in $V$
stable under $\GL_{\frac{n}{ef}}(\OO_E)$; then we have a map: $\GL_{\frac{n}{ef}}(\OO_E) \rightarrow \overline{G}$
coming from the isomorphism: $\overline{G} \cong \GL_{\FF_{q^f}}(L/\unif_E L)$.  The Levi subgroup
$\overline{M}$ of $\overline{G}$ then gives a direct sum decomposition:
$$L/\unif_E L = \oplus_i \overline{L}_i,$$ 
where the $\overline{L}_i$ are the minimal subspaces of $L/\unif_E L$ stable under $\overline{M}$.  Choose a lift
of this to a direct sum decomposition:
$$L = \oplus_i L_i$$ 
of $\OO_E$-modules, and hence a decomposition:
$$V = \oplus_i V_i$$ 
of $E$-vector spaces.  Let $M$ be the corresponding Levi subgroup of
$G$ (consisting of matrices that preserve each of the $V_i$).  
Then $M$ is a product of linear groups $M_i = \Aut_F(V_i)$.
Note that the image of $M \cap \GL_{\frac{n}{ef}}(\OO_E)$ in $\overline{G}$ is precisely $\overline{M}$.

Let $P$ be the parabolic subgroup of $G$ that preserves the subspaces $V_1$, $V_1 + V_2$, etc.,
and let $U$ be its unipotent radical.
Let $\overline{P}$ and $\overline{U}$ be the images of $P \cap \GL_{\frac{n}{ef}}(\OO_E)$
and $U \cap \GL_{\frac{n}{ef}}(\OO_E)$ in $\overline{G}$.  Then $\overline{P}$ is a parabolic
subgroup of $\overline{G}$ with Levi $\overline{M}$ and unipotent radical $\overline{U}.$  

As in section~\ref{sec:zero generic}, the next step is to relate the pair $(K,\tkappa \otimes \St_s)$
to a $G$-cover of a certain cuspidal $M$-type $(K_M,\tau_M)$.  Our construction of the pair $(K_M,\tau_M)$ 
closely parallels sections 7.1 and 7.2 of~\cite{BK}, and is more or less the ``reverse'' of the construction
of section 7 of~\cite{BK-semisimple}.

Recall that the maximal distinguished cuspidal type $(K,\tau)$
arises from a simple stratum $[{\mathfrak A},n,0,\beta]$, together with a character
$\theta$ in ${\mathcal C}({\mathfrak A}, 0, \beta)$.
Given this data, the group $K$ is the group $J(\beta,{\mathfrak A})$ in~\cite{BK}, and the representation
$\tkappa$ of $K$ is a $\beta$-extension of the unique irreducible representation of
$J^1(\beta,{\mathfrak A})$ whose restriction to the subgroup
$H^1(\beta,{\mathfrak A})$ contains the character $\theta$.

Let ${\mathfrak A}_i$ be order in $\End_F(V_i)$ induced by ${\mathfrak A}$; that is,
the image of the subring of ${\mathfrak A}$ that preserves $V_i$ in $\End_F(V_i)$.  Then conjugation by
$E^{\times}$ stabilizes ${\mathfrak A}_i$ and ${\mathfrak A}_i \cap \End_E(V_i) = \End_{\OO_E}(L_i)$.
Given these orders, the procedure at the beginning of~\cite{BK-semisimple}, 7.2 constructs
an order ${\mathfrak A}'$ in $\End_F(V)$ (this is the order denoted by ${\mathfrak A}$ in
section 7 of~\cite{BK-semisimple}.)  The order ${\mathfrak A}'$ is contained in the maximal order
${\mathfrak A}$.  Set $K' = J(\beta,{\mathfrak A}')$.

Let $K_P$ be the preimage of $\overline{P}$ in $K$, via our identification
of $K/K_1$ with $\overline{G}$.  Then by Theorem 5.2.3 (ii) of~\cite{BK}, there is a
unique $W(k)[K']$-module $\tkappa'$ of $K'$ such that we have:
$$\Ind_{K_P}^{({\mathfrak A}')^{\times}} \tkappa|_{K_P}
\cong \Ind_{K'}^{({\mathfrak A}')^{\times}} \tkappa'.$$
(Strictly speaking, this is proved in~\cite{BK} over an algebraically closed field of
characteristic zero, rather than over $W(k)$; in the proof of~\cite{vigbook}, III.4.21
Vigneras observes that the same result holds for the $W(k)$-representations we use here.)
Over $\CK$, the representation $\tkappa'$ can alternatively be described, up to twist, in
the following way: the character $\theta$ gives rise to an endo-class of ps-characters
$(\Theta,0,\beta)$ in the sense of~\cite{BK-semisimple}, section 4.  Then $(\Theta,0,\beta)$
in particular give rise to a character $\theta'$ in ${\mathcal C}({\mathfrak A}', 0, \beta)$.
From this perspective $\tkappa'$ is a $\beta$-extension of the unique irreducible representation
of $J^1(\beta,{\mathfrak A}')$ whose restriction to $H^1(\beta,{\mathfrak A}')$ contains 
the character $\theta'$.

We now apply the construction of~\cite{BK}, 7.2 to the representation $\tkappa'$.
That is, let $K''$ be the subset $(J(\beta,{\mathfrak A}') \cap P)H^1(\beta,{\mathfrak A}')$;
it is shown in~\cite{BK}, 7.1 and 7.2 that $K''$ is a group, and that the
$K' \cap U$-fixed vectors in $\tkappa'$ are stable under $K''$.  Thus these fixed vectors
give a representation $\tkappa''$ of $K''$.  Let $K_M$ be the intersection $K'' \cap M$,
and let $\tkappa_M$ be the restriction of $\tkappa''$ to $K_M$.  
We then have the following:

\begin{lemma}
The pair $(K'',\tkappa'')$ satisfies the conditions (7.2.1) of~\cite{BK-semisimple}.
Explicitly:
\begin{enumerate}
\item The restriction of $\tkappa''$ to $H^1(\beta,{\mathfrak A}')$ is a multiple of $\theta'$.
\item The representation $\tkappa''$ is trivial on $K'' \cap U$ and $K'' \cap U^{\circ}$.
\item The group $K_M$ is the product of the groups $K_i = K \cap M_i$, and
the representation $\tkappa_M$ is a tensor product of irreducible representations
$\tkappa_i$ of $K_i$ for each $i$.  Moreover, the subgroup $K_i$ of $M_i$ is equal
to $J(\beta,{\mathfrak A}_i)$, and each $\tkappa_i$ is a $\beta$-extension
of the unique representation of $J^1(\beta,{\mathfrak A}_i)$ whose
restriction to $H^1(\beta,{\mathfrak A}_i)$ contains the element
$\theta_i$ of ${\mathcal C}({\mathfrak A}_i, 0, \beta)$ determined by the ps-character
$(\Theta,0,\beta)$.
\end{enumerate}
Moreover, the map:
$$\Ind_{K''}^{K'} \tkappa'' \rightarrow \tkappa'$$
(obtained by Frobenius reciprocity from the realization of $\tkappa''$ as the $U \cap K''$-invariants
of $\tkappa'$) is an isomorphism.
\end{lemma}
\begin{proof}
The first claim follows from~\cite{BK}, 5.1.1 and our description of $\tkappa''$ as a $\beta$-extension.
The second is clear from the construction of $\tkappa''$.  The decomposition of $\tkappa$
as a tensor product of $\tkappa_i$ is~\cite{BK}, 7.2.14, as is the fact that each $\tkappa_i$
is a $\beta$-extension (the necessary intertwining property on the $\tkappa_i$ is verified
as part of~\cite{BK}, 7.2.15.)  The fact that the characters $\theta_i$ are the ones claimed in the theorem
follows from~\cite{BK}, 7.1.19.  The isomorphism is~\cite{BK}, 7.2.15.
\end{proof}

Let $s_i$ be the projection of $s$ to $\overline{M}_i$; then $s_i$ is a semisimple element of
$\overline{M}_i$ with irreducible characteristic polynomial.  
As $K_i$ contains $J^1(\beta,{\mathfrak A}_i)$ as a normal subgroup, and the quotient is naturally isomorphic
to $\overline{M}_i$, we may regard the cuspidal representation $\St_{s_i}$ of $\overline{M}_i$
as a representation of $K_i$.  We set $\tau_i = \tkappa_i \otimes \St_{s_i}$; the pair
$(K_i,\tau_i)$ is then a maximal distinguished cuspidal type in $M_i$.
Let $\tau_M$ be the tensor product of the $\tau_i$; it is then a representation of $K_M$, and we have
$\tau_M = \tkappa_M \otimes \St_{\overline{M},s}$,
where $\St_{\overline{M},s}$ is the tensor product of the cuspidal representations $\St_{s_i}$.
The pair $(K_M,\tau_M)$ is
a maximal distinguished cuspidal $M$-type.  

On the other hand, by construction, the quotient of $K''$ by $J^1(\beta,{\mathfrak A}')$
is naturally isomorphic to $\overline{M}$.  We can thus regard $\St_{\overline{M},s}$
as a representation of $K''$, and form the representation $\tau'' = \tkappa'' \otimes \St_{\overline{M},s}.$
We will show that $(K'',\tau'')$ is a $G$-cover of $(K_M,\tau_M)$.  

Note, however, that unlike in the level zero case, there is further work to be done: the pair $(K'',\tau'')$
must be related to the original pair $(K,\tau)$.  To this end, 
consider the representation $\tau_P = \tkappa|_{K_P} \otimes \St_{\overline{M},s},$
where we regard
$\St_{\overline{M},s}$ as a representation of $K_P$ via the surjection
$$K_P \rightarrow \overline{P} \rightarrow \overline{M}.$$

We then have:

\begin{theorem} \label{thm:type cover}
The pair $(K'',\tau'')$ is a $G$-cover of $(K_M,\tau_M)$.  Moreover,
there are natural isomorphisms:
$$\cInd_{K''}^{({\mathfrak A}')^{\times}} \tau'' \cong
\cInd_{K_P}^{({\mathfrak A}')^{\times}} \tau_P $$
$$\cInd_{K''}^{{\mathfrak A}^{\times}} \tau'' \cong 
\cInd_{K}^{{\mathfrak A}^{\times}} \tkappa \otimes I_s.$$
where we regard
$I_s$ as a representation of $K$
via the surjection of $K$ onto $\overline{G}$.
\end{theorem}
\begin{proof}
We have verified that $(K'',\tkappa'')$ satisfies the list of properties in~\cite{BK-semisimple},
7.2.1.  In particular, if one applies the procedure of~\cite{BK-semisimple}, section 7.2 to
the type $(K_M,\tau_M)$, one arrives at the representation $(K'',\tau'')$.  Thus $(K'',\tau'')$ is
a $G$-cover of $(K_M,\tau_M)$ by Theorem 7.2 of~\cite{BK-semisimple}.  (Notice that the procedure given there
in particular applies to the type $(K_M,\tau_M)$ because we have verified that each
of the types $(K_i,\tau_i)$ arises from the same endo-class $(\Theta,0,\beta)$ of ps-character.)

Whenever we have $H'$ a subgroup of $H$, and representations $A$ of $H$ and $A'$ of $H'$,
we have a general identity:
$$\Ind_{H'}^H A|_{H'} \otimes B \cong A \otimes \Ind_H^{H'} B.$$
It follows from this and the isomorphism:
$$\tkappa' \cong \Ind_{K''}^{K'} \tkappa'$$
that we have an isomorphism:
$$\cInd_{K''}^{{\mathfrak A}^{\times}} \tau'' \cong
\cInd_{K'}^{{\mathfrak A}^{\times}} \tau',$$
where $\tau' = \tkappa' \otimes \St_{\overline{M},s}$.  (Recall that the surjections of $K'$ and $K''$
onto $\overline{M}$ are compatible with the inclusion of $K''$ in $K'$, so we can regard
$\St_{\overline{M},s}$ as a representation of $K'$ here.)

Next, the isomorphism:
$$\Ind_{K'}^{({\mathfrak A}')^{\times}} \kappa' \cong
\Ind_{K_P}^{({\mathfrak A}')^{\times}} \kappa|_{K_P}$$
induces an isomorphism:
$$\Ind_{K'}^{({\mathfrak A}')^{\times}} \tau' \cong
\Ind_{K_P}^{({\mathfrak A}')^{\times}} \tau_P$$

Finally, we have an isomorphism:
$$
\Ind_{K_P}^{{\mathfrak A}^{\times}} \kappa|_{K_P} \otimes \St_{\overline{M},s} \cong
\Ind_K^{{\mathfrak A}^{\times}} \kappa \otimes I_s$$
obtained by inducing the previous isomorphism from $({\mathfrak A}')^{\times}$ to ${\mathfrak A}^{\times}$
and applying the tensor product identity.
\end{proof}

As in Corollary~\ref{cor:depth zero bernstein component},
the maximal distinguished cuspidal type $(K_M,\tau_M)$ gives rise to a unique inertial equivalence class
of cuspidal representations of $M$; let $\pi$ be an irreducible representation of $M$ over $\overline{\CK}$
that lies in this inertial equivalence class (or equivalently, that contains the type $(K_M,\tau_M)$.)  We then
have:
\begin{corollary} \label{cor:steinberg Bernstein component}
The $\overline{\CK}[G]$-modules $\cInd_K^G \tkappa \otimes I_s$
and 
$\cInd_K^G \tkappa \otimes \St_s$
are objects of $\Rep_{\overline{\CK}}(G)_{(M,\pi)}$.
\end{corollary}
\begin{proof}
Theorem~\ref{thm:type cover} implies that we have an isomorphism:
$$\cInd_K^G \tkappa \otimes I_s
\cong \cInd_{K''}^G \tau''.$$
As $(K'',\tau'')$ is a $G$-cover of the maximal distinguished cuspidal type $(K_M,\tau_M)$,
it follows immediately from~\cite{BK-types}, Theorem 8.3 that $\cInd_{K''}^G \tau''$ 
lies in $\Rep_{\overline{\CK}}(G)_{(M,\pi)}$.
As for $\cInd_K^G \tkappa \otimes \St_s$, note that, 
by definition, $\St_s$ is the (unique) generic summand of $I_s$,
and so this $\overline{\CK}[G]$-module is a direct summand of the first.
\end{proof}

We now turn to the question of understanding the Hecke algebra $H(G,K,\tkappa \otimes \St_s)$, or equivalently
the endomorphism ring $\End_{\overline{\CK}[G]}(\cInd_K^G \tkappa \otimes \St_s)$.  The isomorphisms of
Theorem~\ref{thm:type cover} induce isomorphisms:
$$H(G,K'',\tau'') \cong H(G,({\mathfrak A}')^{\times}, \Ind_{K''}^{({\mathfrak A}')^{\times}} \tau'')$$
$$H(G,K_P,\tau_P) \cong 
H(G,({\mathfrak A}')^{\times}, \Ind_{K''}^{({\mathfrak A}')^{\times}} \tau'')$$
These isomorphisms are compatible with support in the sense that an element of one of the left hand Hecke algebras
supported on a double coset $K'' g K''$ or $K_P g K_P$ gets sent to an element
of the right hand Hecke algebra supported on $({\mathfrak A}')^{\times} g ({\mathfrak A}')^{\times}$.

We would like to use this observation to compare the spaces $H(G,K'',\tau'')_{K''gK''}$ and
$H(G,K_P,\tau_P)_{K_P g K_P}$ for various $g$ in $G$.

On the one hand, the space $H(G,K'',\tau'')$ is well-understood; just as in the
depth zero case, the discussion in section 1 of~\cite{BK-semisimple}
shows that it is a tensor product of affine Hecke algebras.  As in the depth zero
case, there is a sense in which the isomorphism to a tensor product of affine Hecke algebras
is compatible with support, but (as we are no longer able to work with standard Levi and parabolic subgroups)
it will take somewhat more work to make this precies.  To begin with, recall that $M$ is the subgroup
of $G$ consisting of endomorphisms of $V$ that preserve each summand $V_i$ of $V$, and let $Z$ be the subgroup
of $M$ consisting of elements that act by a power of $\unif_E$ on each $V_i$.

Choose a maximal torus $T_E$ of $\GL_{\frac{n}{ef}}(E)$, and let $\overline{T}$ be its reduction mod $\unif_E$.
Then $\overline{T}$ is a maximal torus of $\overline{G}$; we assume it is contained in $\overline{M}$.
Let $W(\overline{G})$ be the Weyl group of $\overline{G}$ with respect to $\overline{T}$.
The choices of $T_E$ and $\overline{T}$ give an isomorphism of $W(\overline{G})$ with the
Weyl group of $\GL_{\frac{n}{ef}}(E)$.

Now choose a maximal torus $T_F$ of $G$.  We will say that $T_F$ is {\em compatible with $T_E$}
if every $T_E$-stable $E$-line in $F^n$ is a sum of $T_F$-stable $F$-lines.  A choice of $T_F$ compatible
with $T_E$ identifies the Weyl group of $\GL_{\frac{n}{ef}}(E)$ with a subgroup of $W(G)$, and thus
lets us consider $W(\overline{G})$ as a subgroup of $W(G)$.  In what follows, whenever we have a trio
of groups $G,\GL_{\frac{n}{ef}}(E),\overline{G}$, we will choose maximal tori of these groups
related in the sense described above, and implicitly make the corresponding identifications on
Weyl groups.

Let $W_{\overline{M}}$ be the subgroup of $W(\overline{G})$ normalizing $\overline{M}$,
and let $W_{\overline{M}}(s)$ be the subgroup of $W_{\overline{M}}$ consisting of $w$ such that
$w s w^{-1}$ is $M$-conjugate to $s$.
If $W_M$ is the subgroup of $W(G)$ normalizing $M$,
we can identify $W_{\overline{M}}$ with a subgroup of $W_M$.  Then
$H(G,K'',\tau'')$ is supported on the double cosets $K'' g K''$ for $g$ in $W_{\overline{M}}(s) Z$,
and each $H(G,K'',\tau'')_{K'' g K''}$ is a one-dimensional $\overline{\CK}$-vector space.
(Observe that if $w,w'$ are in $W_{\overline{M}}(s)$, and $z,z'$ lie in $Z$, then
$K'' w z K'' = K'' w' z' K''$ if, and only if, $z = z'$ and $w^{-1}w'$ lies in $W(M)$.)

We now decompose $Z$ into a product of factors $Z_j$, just as we did in the depth zero setting.
Specifically, if we write the characteristic polynomial of $s$ as a product of irreducible polynomials
$f_1^{m_1} \dots f_r^{m_r}$, with $\deg f_j = d_j$, then the quotient of $W_{\overline{M}}(s)$ by the subgroup $W(M)$ 
of $W_{\overline{M}}(s)$ is a product of permutation groups $W_j \cong S_{m_j}$, where 
$W_j$ permutes the ``blocks'' of $s$ with characteristic polynomial $f_j$.  If we
let $Z_j$ be the subgroup of $Z$ consisting of elements that are the identity away from those blocks,
then $Z_j$ is a subgroup of $Z$ invariant under the conjugation action of $W_{\overline{M}}(s)$,
and the conjugation action of $W_{\overline{M}}(s)$ on $Z_j$ factors through $S_{m_j}$.  Moreover,
$W_j Z_j$ is a subgroup of $\GL_n$, and the subspace $H_j$ of $H(G,K'',\tau'')$ supported
on cosets of the form $K g K$ for $g$ in $W_j Z_j$ is a subalgebra of $H(G,K'',\tau'')$ isomorphic
to the affine Hecke algebra $H(\GL_n(E_j),I)$, where $E_j$ is the unramified extension of $E$ of
degree $d_j$, and $I$ is the Iwahori subgroup.  (This isomorphism depends on certain choices and
is therefore not canonical; we refer the reader to~\cite{BK}, 5.6, for its construction.)

The algebra $H(G,K'',\tau'')$ is then the tensor product
of the $H_j$.  Moreover, the map from $H(\GL_n(E_j),I)$ to $H_j$ is compatible with supports in a 
manner exactly analogous to the depth zero case.  We may embed $\GL_{m_j}(E'_j)$ in
$M_j$ in such a way that the image of $\GL_{m_j}(\OO_{E_j})$ is equal to the intersection of ${\mathfrak A}$
with the image of $\GL_{m_j}(E_j)$, and so that
the maximal tori of $M_j$ and $M_j \cap \GL_{\frac{n}{ef}}(E_j)$ arising from $T_F$ and $T_E$
are compatible with the standard maximal torus of $\GL_{m_j}(E_j)$.  Then the reduction mod $\unif_E$
of $\GL_{m_j}(\OO_{E_j})$ is a subgroup of $\overline{M}_j$ isomorphic to $\GL_{m_j}(\FF_{q^{fd_j}})$;
we assume we have chosen our embedding so that the standard maximal torus of
$\GL_{m_j}(\FF_{q^{fd_j}})$ is contained in the Levi $\overline{M}_{s_j}$.  

This embedding allows
us to identify $W_j$ with the (standard) Weyl group $W_j'$ of $\GL_{m_j}(E_j)$.  Our choices identify $Z_j$ with
a subgroup $Z'_j$ of the diagonal matrices in $\GL_{m_j}(E_j)$, and then $\GL_{m_j}(E_j)$ is a union of
double cosets $I_j w' z' I_j$, with $w'$ in $W'_j$ and $z'$ in $Z'_j$.  The identification of
$H(\GL_{m_j}(E_j),I)$ with $H_j$ then takes the subspace $H(\GL_{m_j}(E_j),I)_{I w' z' I}$ to $H(G,K'',\tau'')_{K'' w z K''}$,
where $w$ and $z$ are the elements of $W_j$ and $Z_j$ corresponding to $w'$ and $z'$.

Our identification of $H(G,K'',\tau'')$ with a tensor product of Iwahori Hecke algebras
allows us to establish an ``integral'' version of the fact that $(K'',\tau'')$ is a $G$-cover
of $(K_M,\tau_M)$.  More precisely, we have:

\begin{lemma} \label{lemma:lattice G-cover}
Let $L$ be a $W(k)$-lattice in $\St_{\overline{M},s}$, and let $L_M$ be the lattice $\tkappa_M \otimes L$ in
$\tau_M$.  Similarly, let $L''$ be the lattice $\tkappa'' \otimes L$ in $\tau''$.  Then $(K'',L'')$
is a $G$-cover of $(K_M,L_M)$.  Moreover, if $H_j^{\circ}$ denotes the intersection of $H_j$
with $H(K'',L'')$, and $H^{\circ}(\GL_{m_j}(E_j),I)$ is the Iwahori Hecke algebra over $W(k)$,
then there is an isomorphism of $H_j$ with $H(\GL_{m_j}(E_j),I)$ that
restricts to an isomorphism of $H_j^{\circ}$ with $H^{\circ}(\GL_{m_j}(E_j),I)$.
\end{lemma}
\begin{proof}
This is closely related to the arguments of~\cite{vig98}, IV.2.5 and IV.2.6.

It is clear that the only condition that needs to be verified is that there
is an element of $H(K_M,L_M)$ with totally positive support whose image in
$H(K'',L'')$ is invertible.  This follows immediately if there is an isomorphism of $H_j^{\circ}$
with $H^{\circ}(\GL_{m_j}(E_j),I)$ compatible with supports, as in the latter algebra the
characteristic function of any totally positive double coset is invertible.  It thus suffices
to establish this second claim.

Let $\overline{H}(\GL_{m_j}(E_j),I)$ be the subalgebra of $H(\GL_{m_j}(E_j),I)$
consisting of elements supported on double cosets of the form $I w' I$ with $w' \in W_j$,
and $\overline{H}_j$ the subalgebra of $H_j$ supported on double cosets of the form $K'' w K''$
with $w \in W_j$.  Define $\overline{H}^{\circ}(\GL_{m_j}(E_j),I)$ and
$\overline{H}_j^{\circ}$ analogously.  The construction of~\cite{BK}, 5.6.1-5.6.4, yields a canonical isomorphism
of $\overline{H}(\GL_{m_j}(E_j),I)$ with $\overline{H}_j$ compatible with supports, and their argument
shows that this isomorphism sends $\overline{H}^{\circ}(\GL_{m_j}(E_j),I)$ to $\overline{H}^{\circ}_j$.

The extension of this isomorphism to an isomorphism of $H(\GL_{m_j}(E_j),I)$ with
$H_j$ depends on a choice of element $x$ of $H_j$ supported on $K'' \Pi_j K''$, where $\Pi_j$ is
the element of the extended affine Weyl group denoted by $\Pi({\mathcal B})$ in~\cite{BK}, section 5.5
(here we consider $\Pi_j$ to be an element of $M_j$.)
Indeed, the main result of~\cite{BK}, section 5.6 shows that given any such $x$ there is a unique
isomorphism of $H(\GL_{m_j}(E_j),I)$ with $H_j$ sending the characteristic function of $I \Pi_j I$
to $x$ and extending the canonical isomorphism of $\overline{H}(\GL_{m_j}(E_j))$ with $\overline{H_j}$.

Choose $x$ to lie in $H_j^{\circ}$, and to generate the $W(k)$-submodule of $H_j^{\circ}$ supported on
$K'' \Pi_j K''$ (note that the latter is free of rank one over $W(k)$; c.f. the proof of~\cite{BK}, 5.6.7.)
By the argument of~\cite{BK}, 5.6.8, such an $x$ is invertible in $H_j^{\circ}$.

As $H^{\circ}(\GL_{m_j}(E_j),I)$ is generated by $\overline{H}^{\circ}(\GL_{m_j}(E_j))$, together with
the characteristic function of $I \Pi_j I$ and its inverse, the
isomorphism $H(\GL_{m_j}(E_j),I) \rightarrow H_j$ takes $H^{\circ}(\GL_{m_j}(E_j),I)$ to $H_j^{\circ}$.  It thus
remains to show that the image of $H^{\circ}(\GL_{m_j}(E_j),I)$ is all of $H_j^{\circ}.$

Note that for any element $w$ of $W_j$, and any integer $a$, left multiplication by the characteristic function
of $I \Pi_j I$ is an isomorphism:
$$H(\GL_{m_j}(E_j),I)_{I \Pi_j^a w I} \rightarrow H(\GL_{m_j}(E_j),I)_{I \Pi_j^{a+1} w I}.$$
Thus left multiplication by $x$ maps $(H^{\circ}_j)_{K'' \Pi^a_j w' K''}$ to $(H^{\circ}_j)_{K'' \Pi_j^{a+1} w K''}$,
and this map is bijective since $x$ is invertible.  On the other hand, every double coset on which $H^{\circ}_j$
is nonzero has the form $\Pi_j^a w$ for some $a$ and some $w$.  Thus $H^{\circ}_j$ is generated by $x$
and $\overline{H}^{\circ}_k$ and the result follows.
\end{proof}

Our next step is to translate this detailed structure theory for $H(G,K'',\tau'')$
into a corresponding theory for $H(G,K_P,\tau_P)$,
via the isomorphisms of Theorem~\ref{thm:type cover}.  (This is not necessary in the depth zero
case, as in that case we have $K_P = K''$ and $\tau_P = \tau''$.)
To do so we must study the intertwining
of $\tau_P$ and of $\tau$.

\begin{lemma} \label{lemma:intertwining}
Let $\tP$ be the preimage of $\overline{P}$ under the map 
$$\GL_{\frac{n}{ef}}(\OO_E) \rightarrow \GL_{\frac{n}{ef}}(\FF_{q^f}).$$
\begin{enumerate}
\item Let $\xi$ be a $W(k)$-representation of $K$ trivial on $K_1$, and let $g$ be an element
of $G$ that intertwines $\tkappa \otimes \xi$.  Then $g$ lies in
$K \GL_{\frac{n}{ef}}(E) K$.  Moreover, if $g$ lies in $\GL_{\frac{n}{ef}}(E)$,
then $g$ intertwines $\xi$ when $\xi$ is considered as a representation of
$\GL_{\frac{n}{ef}}(\OO_E)$ inflated from $\GL_{\frac{n}{ef}}(\FF_{q^f})$,
and there is a natural isomorphism:
$$I_g(\tkappa \otimes \xi) \cong I_g(\tkappa) \otimes I_g(\xi).$$
\item Let $\xi$ be a $W(k)$-representation of $K_P$ trivial on $K_1$,
and let $g$ be an element of $G$ that intertwines $\tkappa|_{K_P} \otimes \xi$.
Then $g$ lies in $K_P \GL_{\frac{n}{ef}}(E) K_P$.  Moreover,
if $g$ lies in $\GL_{\frac{n}{ef}}(E)$, and then $g$ intertwines $\xi$ when $\xi$ is
considered as a representation of $\tP$ inflated from $\overline{P}$,
and there is a natural isomorpism:
$$I_g(\tkappa|_{K_P} \otimes \xi) \cong \I_g(\tkappa) \otimes I_g(\xi).$$
\end{enumerate}
\end{lemma}
\begin{proof}
Case (1) is almost precisely~\cite{BK}, Proposition 5.3.2, except that the coefficient space
here is $W(k)$ rather than $\CC$.  In spite of this the argument of~\cite{BK} adapts
without difficulty, and the argument in case (2) is identical.
\end{proof}

\begin{corollary} \label{cor:P-intertwining}
Let $g$ be an element of $G$ that intertwines $\tau_P$.  Then
$g$ lies in the double coset $K_P \GL_\frac{n}{ef}(E) K_P$.  Moreover, if
$g$ lies in $K_P W_{\overline{M}}(s) Z K_P$, then
$I_g(\tau_P)$ is one-dimensional.
\end{corollary}
\begin{proof}
The previous lemma shows that $g$ lies in $K_P \GL_{\frac{n}{ef}}(E) K_P$,
so the first statement is clear.  For the second statement, it suffices to
consider $g$ in $W_{\overline{M}}(s) Z$.  For such $g$,
$I_g(\tau_P)$ is equal to $I_g(\St_{\overline{M},s})$, where $\St_{\overline{M},s}$ is considered as a
representation of the parahoric subgroup $\tP$ of $\GL_{\frac{n}{ef}}(E)$ inflated from
$\overline{M}$.  We thus have
$$I_g(\St_{\overline{M},s}) = \Hom_{\tP \cap g \tP g^{-1}}(\St_{\overline{M},s},\St_{\overline{M},s}^g).$$
It is easy to see that for $g$ in $W_{\overline{M}}(s) Z$, we have a surjection
$$\tP \cap g \tP g^{-1} \rightarrow \overline{M}$$
induced by the surjection of $\tP$ onto $\overline{M}$,
and that the representations $\St_{\overline{M},s}$ and $\St_{\overline{M},s}^g$, when considered
as representations of $\tP \cap g \tP g^{-1}$, are both inflated from $\St_{\overline{M}}$ via this surjection.
The result follows.
\end{proof}

As a result we obtain:
\begin{proposition} \label{prop:support}
The isomorphism:
$$ H(G,K'',\tau'') \cong H(G,K_P,\tau_P) $$ 
of Theorem~\ref{thm:type cover} is support-preserving, in the sense that for all $g$ in $\GL_{\frac{n}{ef}}(E)$,
it induces an isomorphism:
$$ H(G,K'',\tau'')_{K''gK''} \cong H(G,K_P,\tau_P)_{K_P g K_P}.$$
\end{proposition}
\begin{proof}
For any $g$ in $G$, we have an isomorphism:
$$\bigoplus_{g'} H(G,K'',\tau'')_{K'' g' K''} \cong 
H(G,({\mathfrak A'})^{\times},\cInd_{K''}^{({\mathfrak A}')^{\times}} \tau'')_{({\mathfrak A}')^{\times} g ({\mathfrak A}')^{\times}}$$ 
where $g'$ lies in a set of representatives for the double cosets $K'' g' K''$ in
$({\mathfrak A}')^{\times} g ({\mathfrak A}')^{\times}$.
An easy calculation shows that if $g$ and $g'$ lie in $W_{\overline{M}}(s) Z$, and
$$({\mathfrak A}')^{\times} g ({\mathfrak A}')^{\times} = 
({\mathfrak A}')^{\times} g' ({\mathfrak A}')^{\times},$$
then $K'' g K'' = K'' g' K''$ and $K_P g K_P = K_P g' K_P.$
As $H(G,K'',\tau'')$ is supported on $K'' W_{\overline{M}}(s) Z K''$, it follows that 
$H(G,({\mathfrak A}')^{\times},\cInd_{K''}^{({\mathfrak A}')^{\times}} \tau'')$ is supported
on the double cosets $({\mathfrak A}')^{\times} W_{\overline{M}}(s) Z ({\mathfrak A}')^{\times}$.  Moreover,
each $H(G,K'',\tau'')_{K'' g K''}$ is one-dimensional for $g$ in $W_{\overline{M}}(s) Z$.  It follows that
$H(G,({\mathfrak A}')^{\times},\cInd_{K''}^{({\mathfrak A}')^{\times}} \tau'')_{({\mathfrak A}')^{\times} g
({\mathfrak A}')^{\times}}$ is as well. 

On the other hand, for $g$ in $G$ we also have an isomorphism
$$\bigoplus_{g'} H(G,K_P,\tau_P)_{K_P g' K_P} \cong
H(G,({\mathfrak A'})^{\times},\cInd_{K_P}^{({\mathfrak A}')^{\times}} \tau_P)_{({\mathfrak A}')^{\times} g ({\mathfrak A}')^{\times}},$$
where $g'$ runs over a set of representatives for the cosets $K_P g' K_P$
in $({\mathfrak A}')^{\times} g ({\mathfrak A}')^{\times}$.  For $g$ in $W_{\overline{M}}(s) Z$,
the right-hand side is one-dimensional, and the summand on the left corresponding to $g' = g$
is also one-dimensional.  On the other hand, if $g$ doesn't lie in $({\mathfrak A}')^{\times} W_{\overline{M}}(s) Z
({\mathfrak A}')^{\times}$ then the right hand side is zero, so all of the summands on the left vanish as well.
It follows that $H(G,K_P,\tau_P)$ is supported on $K_P W_{\overline{M}}(s) Z
K_P$, and that for $g$ in $W_{\overline{M}}(s) Z$, the isomorphisms above
identify $H(G,K_P,\tau_P)_{K_P g K_P}$
with $H(G,K'',\tau'')_{K'' g K''}$ as required.
\end{proof}

With Proposition~\ref{prop:support} established, the remainder of the argument goes more or less exactly
as in Section~\ref{sec:zero generic}.  The next step is
to understand the relationship between
$H(G,K_P,\tau_P)$ and $H(G,K,\tkappa \otimes \St_s)$.  As
$\St_s$ is a direct summand of $I_s$,
$\cInd_K^G \tkappa \otimes \St_s$ is a direct summand of 
$\cInd_K^G \tkappa \otimes I_s.$  We therefore obtain a map:
$$Z(H(G,K_P,\tau_P)) \rightarrow H(G,K,\tkappa \otimes \St_s).$$
that is support-preserving in the sense that elements supported on a double
coset $K_P g K_P$ for some $g \in \GL_{\frac{n}{ef}}(E)$
get sent to elements of $H(G,K,\tkappa \otimes \St_s)$
supported on $K g K$.  As in depth zero, we have:

\begin{proposition} \label{prop:center injectivity}
The map $Z(H(G,K,\tkappa \otimes I_s)) \rightarrow H(G,K,\tkappa \otimes \St_s)$ is injective.
\end{proposition}
\begin{proof}
The proof of this is identical to the proof of Proposition~\ref{prop:depth zero center injectivity}.
\end{proof}

To show that this map is an isomorphism, it remains to compare
the dimensions
of $H(G,K,\tkappa \otimes \St_s)_{K g K}$ and $Z(H(G,K,\tkappa \otimes I_s))_{K g K}$ for
a suitable set of $g$.  This is essentially an intertwining calculation, analogous to the
one done in section~\ref{sec:zero generic}, except that we must now respect the distinction
between $\GL_{\frac{n}{ef}}(E)$ and $\GL_n(F)$, which introduces a few subtleties:

Fix an element $g$ of $M \cap \GL_{\frac{n}{ef}}(E)$.
For such a $g$, let $\overline{P}_g$ be the image of $g \GL_{\frac{n}{ef}}(\OO_E) g^{-1} \cap \GL_{\frac{n}{ef}}(\OO_E)$
in $\overline{G}$.  Then
$\overline{P}_g$ is a parabolic subgroup with unipotent radical $\overline{U}_g$, and Levi
$\overline{M}_g$.  

Let $\tP_g$ be the subgroup $g \GL_{\frac{n}{ef}}(\OO_E) g^{-1} \cap \GL_{\frac{n}{ef}}(\OO_E)$ of
$\GL_{\frac{n}{ef}}(\OO_E)$.  Then
$g^{-1} \tP_g g$ is also a subgroup of $\GL_{\frac{n}{ef}}(\OO_E)$, and its
image under the reduction map to $\overline{G}$ is opposite parabolic $\overline{P}_g^{\circ}$ of $\overline{P}_g$.
Let $\tU_g$ be the preimage
of $\overline{U}_g$ in $\tP_g$; then for any $u$ in $\tU_g$, the conjugate $g^{-1} u g$ maps to
the identity under the composition $g^{-1} \tP_g g \rightarrow \overline{P}_g \rightarrow \overline{M}_g$.
The conjugation map $\tP_g \rightarrow g^{-1} \tP_g g$
thus descends to a map $\overline{M}_g \rightarrow \overline{M}_g$; we say $g$
is $\overline{M}_g$-central if the resulting automorphism of $\overline{M}_g$ is trivial.
It is clear that for any $g$, there exists a $k$ in $\GL_{\frac{n}{ef}}(\OO_E)$
such that $gk$ is $\overline{M}_g$-central.
We then have:

\begin{lemma} \label{lemma:xi-intertwining}
Let $g$ be an element of $M \cap \GL_{\frac{n}{ef}}(E)$, and let
$\xi$ be a $W(k)[\overline{G}]$-module considered as a module over
$W(k)[\GL_{\frac{n}{ef}}(\OO_E)]$ by inflation.  Then $I_g(\xi)$ is free of
rank one over $\End_{W(k)[\overline{M}_g]}(r_{\overline{G}}^{\overline{P}_g} \xi).$
Moreover, if $g$ is $\overline{M}_g$-central, we have a natural isomorphism:
$$I_g(\xi) \cong \End_{W(k)[\overline{M}_g]}(r_{\overline{G}}^{\overline{P}_g} \xi).$$
\end{lemma}
\begin{proof}
By definition, $I_g(\xi) = \Hom_{W(k)[\tP_g]}(\xi,\xi^g)$.  
Note that for $u$ in $\tU_g$, the element $g^{-1} u g$ acts trivially on $\xi$.  Thus $u$ acts trivially on $\xi^g$.
In particular any element of $I_g(\xi)$ gives rise to a map:
$$\Hom_{W(k)[\tP_g]}(r_{\overline{G}}^{\overline{P}_g} \xi, r_{\overline{G}}^{\overline{P}_g} \xi^g).$$
Conversely, any such map gives rise to an element of $I_g(\xi)$.
(Here we are identifying the $\overline{U}_g$-invariants of $\xi$ with the $\overline{U}_g$-coinvariants
via the natural map from invariants to coinvariants.  Moreover, the representation
$r_{\overline{G}}^{\overline{P}_g} \xi$ is a representation of $\overline{M}_g$ considered
as a representation of $\tP_g$ by inflation.) 
As conjugation by $g$ descends to an inner automorphism of $\overline{M}_g$,
$r_{\overline{G}}^{\overline{P}_g} \xi^g$ is isomorphic to $r_{\overline{G}}^{\overline{P}_g} \xi$,
and we can take this isomorphism to be the identity when $g$ is $\overline{M}_g$-central.  The result is
then clear.
\end{proof}

Let $z$ be an element of $Z$.  Then our fixed maximal torus of $\overline{G}$ is a maximal torus of
$\overline{M}_g$.
As in section~\ref{sec:zero generic},  we define $W(\overline{M},\overline{M}_z)$
as the set of $w$ in $W(\overline{G})$ such that $w \overline{M} w^{-1}$ is contained in $\overline{M}_z$.  
We then have:

\begin{proposition} \label{prop:K-intertwining}
The spaces $H(G,K,\tkappa \otimes \St_s)$ and $H(G,K,\tkappa \otimes i_{\overline{P}}^{\overline{G}} \St_{\overline{M},s})$
are supported on $K Z K$.  Moreover, for $z$ in $Z$, we have:
$$H(G,K,\tkappa \otimes I_s)_{K z K}
\cong I_z(\tkappa) \otimes \End_{\overline{\CK}[\overline{M}_z]}(\oplus_{w} I_{\overline{M}_z, w s w^{-1}}) $$
$$H(G,K,\tkappa \otimes \St_s)_{K z K}
\cong I_z(\tkappa) \otimes \End_{\overline{\CK}[\overline{M}_z]}(\oplus_{w'} \St_{\overline{M}_z,w' s (w')^{-1}})$$
where $w$ runs over a set of representatives for 
$W(\overline{M}_z) \backslash W(\overline{M},\overline{M}_z)$,
and $w'$ runs over a set of representatives for
$W(\overline{M}_z) \backslash W(\overline{M},\overline{M}_z)/W_{\overline{M}}(s)$,
\end{proposition}
\begin{proof}
By Lemma~\ref{lemma:intertwining}, it suffices to compute the spaces $I_z(I_s)$ and $I_z(\St_s)$. 
Note that for any $z$ in $Z$, it is clear that $z$ is $\overline{M}_z$-central.
The result is thus immediate from Lemma~\ref{lemma:xi-intertwining}, together with Propositions
\ref{prop:steinberg restriction} and~\ref{prop:induction restriction}.
\end{proof}

As a result, exactly as in section~\ref{sec:zero generic}, we have:
\begin{corollary} \label{cor:dim count}
For $z \in Z$, the dimension of $H(G,K,\tkappa \otimes St_s)_{K z K}$ is equal to
the cardinality of $W(\overline{M}_z) \backslash W(\overline{M},\overline{M}_z)/W_{\overline{M}}(s)$.
\end{corollary}
\begin{proof}
The modules $\St_{\overline{M}_z, w s w^{-1}}$ and $\St_{\overline{M}_z, (w') s (w')^{-1}}$
are isomorphic if, and only if, $w$ and $w'$ represent the same class in
$W(\overline{M}_z) \backslash W(\overline{M},\overline{M}_z)/W_{\overline{M}}(s)$.
It follows that
$$\dim_{\overline{\CK}} \End_{\overline{\CK}[\overline{M}_z]}(\oplus_{w} \St_{\overline{M}_z, w s w^{-1}})
= \# W(\overline{M}_z) \backslash W(\overline{M},\overline{M}_z)/W_{\overline{M}}(s)$$
and the result follows by Proposition~\ref{prop:K-intertwining}.
\end{proof}

It remains to compute the dimension of the subspace of $H(G,K,\tkappa \otimes I_s)$ that is
central and supported on $K g K$.  The argument here is more or less identical to
that of section~\ref{sec:zero generic}, via the identification of $H(G,K,\tkappa \otimes I_s)$ with
the tensor product of the $H_j$.  This identification, together with Lemma~\ref{lemma:Iwahori central support},
allows us to immediately deduce:

\begin{prop} \label{prop:central support}
For any $z$ in $Z$, the space of central elements of $H(G,K_P,\tau_P)$
supported on the union of double cosets of the form $K_P w z w' K_P$
with $w,w'$ in $W_{\overline{M}}(s)$ is one-dimensional.  Moreover, the sum of these spaces as $z$
varies is the entire center of $H(G,K_P,\tau_P)$.
\end{prop}

We can thus conclude:
\begin{theorem} \label{thm:generic hecke}
The map: 
$$Z(H(G,K_P,\tau_P)) \rightarrow
H(G,K,\tkappa \otimes \St_s)$$
is an isomorphism.
\end{theorem}
\begin{proof}
By the injectivity of this map, and Corollary~\ref{cor:dim count}, it suffices
to show that for each $z$ in $Z$, the subspace of $Z(H(G,K_P, \tau_P))$
supported on $K z K$ has dimension equal to the cardinality of 
$W(\overline{M}_z \backslash W(\overline{M},\overline{M}_z)/W_{\overline{M}}(s)$.
This is immediate from Proposition~\ref{prop:central support} and Proposition~\ref{prop:cosets}.
\end{proof}

As in~\ref{sec:zero generic}, we may deduce from this that the map $A_{M,\pi} \rightarrow H(G,K, \tkappa \otimes \St_s)$
giving the action of $A_{M,\pi}$ on $\cInd_K^G \tkappa \otimes \St_s$ is an isomorphism.

\section{Structure of $\End_G(\CP_{K,\tau})$} \label{sec:endomorphisms}

Our next step will be to understand the endomorphism ring 
$\End_{W(k)[G]}(\CP_{K,\tau}) = H(G,K,\CP_{\kappa \otimes \sigma})$ for an arbitrary $\kappa$.
Mostly, we will follow the arguments of section~\ref{sec:zero endomorphisms} in the depth zero case, with
some necessary modifications.  However, it will be useful to make some additional observations, which
we omitted from section~\ref{sec:zero endomorphisms} because their proof does not appreciably simplify
in the depth zero setting.

To simplify notation, let $E_{\sigma}$ be the ring $\End_{W(k)[K]}(\CP_{\sigma})$,
and let $E_{K,\tau}$ be the ring $\End_{W(k)[G]}(\CP_{K,\tau})$.  By Proposition~\ref{prop:endomorphisms},
$E_{\sigma}$ is a reduced commutative $W(k)$-algebra that is free of finite rank
over $W(k)$.

We have a direct sum decomposition:
$$\CP_{\sigma} \otimes \overline{\CK} \cong \bigoplus_{s: s^{\reg} = s'} \St_s,$$
where $s'$ is the $\ell$-regular semisimple element of $\overline{G}$ corresponding to the representation
$\sigma$.  This yields a decomposition:
$$\CP_{K,\tau} \otimes \overline{\CK} = \bigoplus_{s: s^{\reg} = s'} \cInd_K^G \tkappa \otimes \St_s.$$

\begin{proposition} \label{prop:P intertwining}
Let $g$ be an element of $G$ that intertwines $\CP_{K,\tau}$.  Then $g$ lies in $K \GL_{\frac{n}{ef}}(E) K$,
and we have a direct sum decomposition:
$$I_g(\CP_{K,\tau}) \otimes \overline{\CK} = \bigoplus_{s: s^{\reg} = s'} I_g(\tkappa \otimes \St_s).$$
Moreover, if $g$ lies in $\GL_{\frac{n}{ef}}$ and is $\overline{M}_g$-central, then there is a natural isomorphism:
$$I_g(\CP_{K,\tau}) \cong I_g(\tkappa) \otimes \End_{W(k)[\overline{M}_g]}(r_{\overline{G}}^{\overline{P}_g} \CP_{\sigma}).$$
\end{proposition}
\begin{proof}
Every statement other than the decomposition is a direct consequence of Lemma~\ref{lemma:xi-intertwining}.
As for the direct sum decomposition, the decomposition:
$$P_{\sigma} \otimes \overline{\CK} = \bigoplus_{s: s^{\reg} = s'} \St_s,$$
yields a decomposition
$$r_{\overline{G}}^{\overline{P}} \CP_{\sigma} \otimes \overline{\CK} = 
\bigoplus_{s: s^{\reg} = s'} r_{\overline{G}}^{\overline{P}} \St_s.$$
From Proposition~\ref{prop:steinberg restriction} we see that no two irreducible summands of the right hand side
are isomorphic to each other, so any endomorphism of the right hand side preserves each of the summands.  We thus have
a decomposition:
$$\End_{W(k)[\overline{M}_g]}(r_{\overline{G}}^{\overline{P}_g} \CP_{\sigma}) = \bigoplus_{s: s^{\reg} = s'}
\End_{W(k)[\overline{M}_g]}(r_{\overline{G}}^{\overline{P}_g} \St_s).$$
Tensoring both sides with $I_g(\tkappa)$ yields the desired result.
\end{proof}

\begin{corollary} \label{cor:s decomposition}
The action of $E_{K,\tau} \otimes \overline{\CK}$ on the direct sum decomposition
$$\CP_{K,\tau} \otimes \overline{\CK} = \bigoplus_{s: s^{\reg} = s'} \cInd_K^G \tkappa \otimes \St_s$$
preserves each summand, and thus yields an isomorphism:
$$E_{K,\tau} \otimes \overline{\CK} \cong \prod_{s: s^{\reg} = s'} H(G,K,\tkappa \otimes \St_s).$$
In particular $E_{K,\tau}$ is reduced, commutative, and $\ell$-torsion free.  Moreover,
$\Spec E_{K,\tau}$ is connected.
\end{corollary}
\begin{proof}
The direct sum decomposition follows immediately from Proposition~\ref{prop:P intertwining} by summing
over the elements supported on $K g K$ for each $g$.  The embedding follows from the absence
of $\ell$-torsion in $E_{K,\tau}$ (which is immediate from the fact that $\CP_{K,\tau}$ is
projective.) Reducedness and commutativity of $E_{K,\tau}$
then follow from the corresponding results for $H(G,K,\tkappa \otimes \St_s)$.

For the last claim, suppose that we had a nontrivial idempotent in $E_{K,\tau}$. 
Then, since the spectrum of $H(G,K,\tkappa \otimes \St_s)$ is connected for all $s$,
such an idempotent maps to either zero or one in each $H(G,K,\tkappa \otimes \St_s)$,
and hence lies in $E_{\sigma} \otimes \CK$.  But $E_{\sigma}$ is saturated in
$E_{K,\tau}$; that is, if $\ell x$ lies in $E_{\sigma}$, then so does $x$.  Thus we
would obtain a nontrivial idempotent in $E_{\sigma}$, which is impossible since $E_{\sigma}$
is local.
\end{proof}

We now turn to questions that have already been addressed in the depth zero case in section~\ref{sec:zero endomorphisms}.
As in that setting, it will be necessary to understand families of cuspidal representations of $\GL_i(F)$
for various $i$, all of which have supercuspidal support inertially equivalent to some power of
a fixed supercuspidal representation over $k$.  Fix an integer $n_1$, and let $(K_1,\tau_1)$
be a maximal distinguished {\em supercuspidal} $k$-type for $\GL_{n_1}(F)$.  We have $\tau_1 = \kappa_1 \otimes \sigma_1$,
where $\sigma_1$ is the supercuspidal representation of $\GL_{\frac{n_1}{ef}}(\FF_{q^f})$ 
attached to an $\ell$-regular element $s'_1$ of $\GL_{\frac{n_1}{ef}}(\FF_{q^f})$ with
irreducible characteristic polynomial.

Recall that $e_{q^f}$ is the order of $q^f$ modulo $\ell$, and let
$m$ lie in $\{1, e_{q^f}, \ell e_{q^f}, \ell^2 e_{q^f}, \dots \}$.
For precisely such $m$, the generic representation $\sigma_m$ of $\GL_{\frac{n_1m}{ef}}(\FF_{q^f})$
corresponding to a block matrix consisting of $m$ copies of $s'_1$ is cuspidal.

In section~\ref{sec:zero endomorphisms} the choice of such a $\sigma_m$ sufficed to define a depth zero type
in $\GL_{n_1m}(F)$.  Here we must work harder:
for any suitable $m$ we must also define a representation $\kappa_m$ of a suitable compact open subgroup
of $\GL_{n_1m}(F)$.  We do so as follows: the data giving rise to the type $(K_1,\tau_1)$ consist of an
extension $E$ of $F$ of ramification index $e$ and residue class degree $f$, a stratum
$[{\mathfrak A_1}, n, 0, \beta]$, with $E = F[\beta]$, and a character $\theta$
in ${\mathcal C}({\mathfrak A_1}, 0, \beta)$.  Then $\theta$ gives rise to an endo-equivalence
class $(\Theta, 0, \beta)$.  Let ${\mathfrak A}_m$ be the maximal order $M_m({\mathfrak A}_1)$ of 
$M_{n_1m}(F)$; then the endo-class $(\Theta, 0, \beta)$ gives rise to a compact open subgroup
$K_m$ of $\GL_{n_1m}(F)$ and a representation $\kappa_m$ of $K_m$.  (This is slightly misleading:
the representation $\kappa_m$ is in fact only well-defined up to certain twists.  For the moment
we will let $\kappa_m$ be any representation of $K_m$ arising from the choice of ${\mathfrak A}_m$ and
the endo-class $(\Theta, 0, \beta)$.  Later we will make additional assumptions on $\kappa_m$ that
pin it down precisely.

Let $\tau_m = \kappa_m \otimes \sigma_m$;
then $(K_m,\tau_m)$ is a maximal distinguished cuspidal $k$-type in $\GL_{n_1m}(F)$.

For each positive integer $m$, let $G_m$ be the group $\GL_{n_1m}(F)$ and let $\overline{G}_m$ 
be the group $\GL_{\frac{n_1m}{ef}}(\FF_{q^f})$.
For each semisimple conjugacy class $s$ in $\overline{G}_m$ with $\ell$-regular part $(s'_1)^m$,
let $\overline{M}_s$ be the minimal split Levi of $\overline{G}_m$ containing $s$,
let $M_s$ be the corresponding Levi of $G_m$, and let $Z_s$ be the subgroup of the center of
$M_s \cap \GL_{\frac{n_1m}{ef}}(E)$ consisting of diagonal matrices whose entries are powers of $\unif_E$.

As in section~\ref{sec:zero endomorphisms}, we may define 
elements $\theta_{i,s}$ of $W(k)[Z_s]^{W_{\overline{M}_s}(s)}$ as follows:
\begin{definition}
Let $\theta_{i,s}$ be the element of $Z_s$ given by the sum of the elements of $Z_s$ 
whose characteristic polynomial
(as an element of $\GL_{\frac{n_1m}{ef}}(E)$) has the form 
$(t - \unif_E)^{\frac{n_1i}{ef}}(t-1)^{\frac{n_1(m-i)}{ef}}$.
\end{definition}
As in section~\ref{sec:zero endomorphisms}, it is possible that $\theta_{i,s} = 0$. 

Of particular interest to us will be the subgroups corresponding to $s = (s'_1)^m$; we denote these by
$\overline{M}_{0,m}$, $M_{0,m}$, and $Z_{0,m}$.  For $i$ between $1$ and $m$, let $z_{i,m}$
be an element of $Z_{0,m}$ with characteristic polynomial
$(t - \unif_E)^{\frac{n_1i}{ef}}(t-1)^{\frac{n_1(m-i)}{ef}}$.

For $m \in \{1, e_{q^f}, \ell e_{q^f}, \dots \}$, and any conjugacy class $s$
in $\overline{G}_m$ with $\ell$-regular part $(s'_1)^m$, set $E_m = E_{K_m,\tau_m}$.
We have a chain of maps:
$$E_m \rightarrow H(G_m,K_m,\tkappa_m \otimes \St_s) \cong A_{M_s,\pi_s} \cong \overline{\CK}[Z_s]^{W_{\overline{M}_s}(s)},$$
where the first map arises from Corollary~\ref{cor:s decomposition}, the pair $(M_s,\pi_s)$ is in the inertial
equivalence class determined by $\tkappa_m$ and $s$, and the last isomorphism depends on
a particular choice of $(M_s,\pi_s)$.  As in section~\ref{sec:zero endomorphisms} we will construct elements
of $E_m$ whose action we can describe explicitly in terms of the $\theta_{i,s}$; this will require
choosing a ``compatibile family of cuspidals'' analogous to the family constructed in section~\ref{sec:zero endomorphisms},
in order to suitably normalize the maps to $\overline{\CK}[Z_s]$ for all $s$. 

Fix, once and for all, an absolutely irreducible integral supercuspidal representation $\pi_{s'_1}$ of $\GL_{n_1}(F)$ 
containing $(K_1, \tkappa_1 \otimes \St_{s'_1})$.  We will construct,
for each $m \in \{1,e_{q^f},\ell e_{q^f}, \dots \}$, and 
each irreducible conjugacy class $s$ in $\overline{G}_m$ such that $s^{\reg} = (s'_1)^m$, 
an absolutely irreducible cuspidal representation $\pi_s$ of $\GL_{n_1m}(F)$ containing 
$(K_m,\tkappa_m \otimes \St_s)$.  Our construction will mostly parallel that of section~\ref{sec:zero endomorphisms},
but as usual the nontriviality of $\tkappa_m$ introduces complications.

We proceed as follows: note that $(K_m,\tkappa_m \otimes \St_s)$ is a maximal distinguished cuspidal $\overline{\CK}$-type.
In particular $H(G_m,K_m,\tkappa_m \otimes \St_s)$ is isomorphic to $\overline{\CK}[\Theta^{\pm 1}]$,
where $\Theta$ is an element of $H(G_m,K_m,\tkappa_m \otimes \St_s)$ supported on $K_m z_{m,m} K_m$.  We have an isomorphism:
$$H(G_m,K_m,\tkappa_m \otimes \St_s)_{K_m z_{m,m} K_m} \cong I_{z_{m,m}}(\tkappa_m) \otimes \End_{\overline{\CK}[\overline{G}_m]}(\St_s).$$

We also have an isomorphism:
$$H(G_m,K_m,\tkappa_m \otimes \St_{(s'_1)^m})_{K_m z_{m,m} K_m} \cong I_{z_{m,m}}(\tkappa_m) \otimes \End_{\overline{\CK}[\overline{G}_m]}(\St_{(s'_1)^m}).$$
Let $\pi_{(s'_1)^m}$ be the cuspidal representation $\pi_{s'_1}^{\otimes m}$ of $M_{(s'_1)^m}$.
Such a choice gives an isomorphism of $H(G_m,K_m,\tkappa_m \otimes \St_{(s'_1)^m})$
with $\overline{\CK}[Z_{0,m}]^{W_{\overline{M}_{0,m}}((s_1')^m)}$; the element $z_{m,m}$ of $\overline{\CK}[Z_{0,m}]$
maps to an element of $H(G_m,K_m,\tkappa_m \otimes \St_{(s'_1)^m})_{K_m z_{m,m} K_m}$.  This element has the form
$\phi \otimes 1$ for some $\phi \in I_{z_{m,m}}(\tkappa_m)$, where we regard $1$ as the identity endomorphism of $\St_{(s'_1)^m}$.

We may then consider the element $\phi \otimes 1$ of $I_{z_{m,m}}(\tkappa) \otimes \End_{\overline{\CK}[\overline{G}_m]}(\St_s)$.
This gives us an element $\Theta$ of $H(G_m,K_m,\tkappa \otimes \St_s)_{K_m z_{m,m} K_m}$.  There is then a unique irreducble cuspidal
representation $\pi_s$ of $G_m$ over $\overline{\CK}$ that contains $(K_m,\tkappa \otimes \St_s)$, on which $\Theta$
acts via the identity.  Note that $\pi_s$ is integral.

We will refer to the family of cuspidal representations $\{\pi_s\}$ as the {\em compatible family of cuspidals}
attached to our choice of $\pi_{s'_1}$.  The dependence on the choice of $\pi_{(s'_1)}$ is mild; indeed,
replacing $\pi_{(s'_1)}$ by an unramified twist $\pi_{(s'_1)} \otimes (\chi \circ \det)$ simply twists
each $\pi_s$ by $\chi \circ \det$ as well.  Note that (for the moment) there is also a dependence on the
$\kappa_m$ we have chosen; this will be resolved when we pin down the $\kappa_m$ precisely.

For each $m$, and each {\rm reducible} conjugacy class $s$ in $\overline{G}_m$, we define an irreducible cuspidal
representation $\pi_s$ of $M_s$, by taking $\pi_s$ to be the tensor product of the cuspidal representations $\pi_{s_i}$
defined above, where the $s_i$ are the irreducible factors of $s$.  

We can now state a key result of this section, analogous to Theorem~\ref{thm:depth zero Theta}:

\begin{theorem} \label{thm:compatibility}
For each $m$, there exists a subalgebra $C_{K_m,\tau_m}$ of $E_{K_m,\tau_m}$, generated over $W(k)$ by elements
$\Theta_{1,m}, \dots, \Theta_{m,m}$ and $(\Theta_{m,m})^{-1}$,
such that 
for any $s$ with $s^{\reg} = (s'_1)^m$, the composed map:
$$C_{K_m,\tau_m} \rightarrow A_{M_s,\pi_s} \cong \overline{\CK}[Z_s]^{W_{\overline{M}_s}(s)}$$
takes $\Theta_{i,m}$ to $\theta_{i,s}$.  (Here the left-hand 
map is the map giving the action of $E_{K_m,\tau_m}$ on the summand $(\CP_{K_m,\tau_m} \otimes \overline{\CK})_{M_s,\pi_s}$
of $\CP_{K_m,\tau_m} \otimes \overline{\CK}$.)
\end{theorem}

Note that this property characterizes the $\Theta_{i,m}$ uniquely.  The basic strategy, as in the depth
zero case, is to construct the $\Theta_{i,m}$ inductively via $G$-covers.
The case $m=1$ is easy: we have an isomorphism
$$H(G_1,K_1,\tkappa_1 \otimes \CP_{\sigma_1})_{K_1 z_{1,1} K_1} \cong I_{z_{1,1}}(\tkappa_1) \otimes \End_{W(k)[\overline{G}_1]}(\CP_{\sigma_1}).$$
We define $\Theta_{1,1}$ to be the unique element of the form $\phi \otimes 1$, for some $\phi$ in $I_{z_{1,1}}(\tkappa_1)$,
that acts on the quotient $\pi_{(s'_1)}$ of $\CP_{K_1,\tau_1} \otimes \overline{\CK}$ via the identity.  It is clear
from our construction of the $\pi_s$ that this $\Theta_{1,1}$ has the desired property.

We now turn to the inductive part of the argument.  This will proceed 
via a $G$-cover argument, along the lines of the construction
in section~\ref{sec:zero endomorphisms}.  However, the argument is considerably more delicate because of the presence of
the $\tkappa_m$.  In particular, the same complications that arose in section~\ref{sec:generic} (as compared
to section~\ref{sec:zero generic}) will arise here.

Let $m'$ be the largest element of $\{1,e_{q^f}, \ell e_{q^f}, \dots\}$ strictly less than
$m$, and set $j = \frac{m}{m'}$.  Let $V_{m'}$
be the $F$-vector space on which $G_{m'}$ acts, and identify $V_m$ with the direct sum of $j$ copies 
$V_{m',1}, \dots, V_{m',j}$ of $V_{m'}$.  Let $M$ be the Levi subgroup of $G_m$ preserving this direct sum
decomposition, and let $P$ be the parabolic preserving the flag 
$$V_{m',1} \subset V_{m',1} + V_{m',2} \subset \dots \subset V_{m',1} + \dots + V_{m',j}.$$
The groups $P = MU$ give rise to subgroups $\overline{P} = \overline{M}\overline{U}$ of
$\overline{G}_m$ in the usual way.

We have maximal orders ${\mathfrak A}_m$ of $G_m$ and ${\mathfrak A}_{m'}$ of $G_{m'}$ attached to the
types $(K_m,\tau_m)$ and $(K_{m'},\tau_{m'})$.  The procedure of~\cite{BK-semisimple}, 7.2 also
yields an order ${\mathfrak A}'_m$ contained in ${\mathfrak A}_m$ attached to the flag defined above.
Set $K'_m = J(\beta,{\mathfrak A}'_m),$ let $K''_m$ be the subgroup
$(J(\beta,{\mathfrak A}'_m) \cap P)H^1(\beta,{\mathfrak A}'_m)$ of $K'_m$, and let
$K_{m,P}$ be the preimage of $\overline{P}$ in $K_m$ under the map from
$K_m$ to $\overline{G}_m$.  
Then, just as in section~\ref{sec:generic} we have representations $\tkappa'_m$ of $K'_m$,
$\tkappa''_m$ of $K''_m$, and $\tkappa_{m,P}$ of $K_{m,P}$,
satisfying:
$$\tkappa'_{m,P} = (\tkappa'_m)|_{K_{m,P}}$$
$$\Ind_{K''_m}^{K'_m} \tkappa''_m \cong \tkappa'_m$$
$$\Ind_{K_{m,P}}^{({\mathfrak A}')^{\times}} \tkappa'_{m,P} \cong \Ind_{K'_m}^{({\mathfrak A}')^{\times}} \tkappa'_m.$$

Moreover, the intersection $K_M$ of $K''_m$ with $M$ is (under the identification of $M$ with a product of $j$ copies
of $G_{m'}$) equal to the product of $j$ copies of $K_{m'}$.  The restriction $\tkappa_M$ of $\tkappa''_m$ to $K_M$
factors as a product of $j$ copies of a representation $\tkappa'_{m'}$ containing a character attached to
${\mathfrak A}_m$ via the endo-class $(\Theta,0,\beta)$.  Thus $\tkappa'_{m'}$ differs from $\tkappa_{m'}$
by a twist.  

{\em We henceforth assume that, for each $m$, we have chosen $\tkappa_m$ so that $\tkappa'_{m'}$
is equal to $\tkappa_{m'}$.}  It is clear this can be done; the choice of $\kappa_1$ is arbitrary, and
for each pair $m',m$ of successive elements of $\{1,e_{q^f}, \ell e_{q^f}, \dots \}$,
changing $\tkappa_m$ by twist changes $\tkappa'_{m'}$ by the ``same'' twist.
With this assumption the representations $\tkappa_{m}$, depend only on our choice
of $\tkappa_1$.

Finally, let $\CP_{\overline{M}}$ denote the inflation of $\CP_{\sigma_{m'}}^{\otimes j}$ from $\overline{M}$ to
a representation of $K_M$, and also (somewhat abusively) the inflation of $\CP_{\sigma_{m'}}^{\otimes j}$
to a representation of $K''_m$ (via the surjection of $K''_m$ onto $\overline{P}$.)  

Exactly as in section~\ref{sec:generic}, we obtain isomorphisms:
$$\cInd_{K''_m}^{G_m} \tkappa''_m \otimes \CP_{\overline{M}} \cong \cInd_{K_{m,P}}^{G_m} \tkappa_{m,P} \otimes \CP_{\overline{M}}
\cong \cInd_{K_m}^{G_m} \tkappa_m \otimes i_{\overline{P}}^{\overline{G}_m} \CP_{\overline{M}}$$
$$H(G,K''_m,\tkappa''_m \otimes \CP_{\overline{M}}) \cong H(G_m,K_P,\tkappa_{m,P} \otimes \CP_{\overline{M}})
\cong H(G_m,K_m,\tkappa_m \otimes i_{\overline{P}}^{\overline{G}_m} \CP_{\overline{M}}).$$
Moreover, the discussion in section~\ref{sec:generic} following Theorem~\ref{thm:type cover} shows that these maps
are compatible with supports.

We are now in a position to prove an analog of Theorem~\ref{thm:depth zero G-cover}:
\begin{theorem} \label{thm:projective G cover}
The pair $(K''_m,\tkappa''_m \otimes \CP_{\overline{M}})$ is a $G$-cover of $(K_M,\tkappa_M \otimes \CP_{\overline{M}})$.
\end{theorem}
\begin{proof}
By Lemma~\ref{lemma:G-cover reduction}, it suffices to construct a sequence of central elements $\lambda_1, \dots, \lambda_r$
of $M$, whose product is strictly positive,
and for each $i$ an invertible element $x_i$ of $H(G_m,K''_m,\tkappa'' \otimes \CP_{\overline{M}})$
supported on $K''_m \lambda_i' K''_m$, such that $T^+ x_i$ is invertible. 

Let $w$ be an element of $W(G_m)$ that maps $V_{m',1}$ to $V_{m',j}$, and $V_{m',i}$ to $V_{m',i-1}$ for $1 < i \leq j$,
and let $z$ be the central element of $M$ that acts by multiplication by $\unif_E$ on $V_{m',1}$ and by the identity on
all other $V_{m',i}$.  Let $\Pi = w z$.  Then $\Pi$ normalizes $K''_m$, and intertwines $\tkappa''_m$.  Let $\alpha$ be
any isomorphism $\tkappa''_m \cong (\tkappa''_m)^{\Pi}$, and let $\beta$ be the automorphism of $\CP_{\overline{M}} = \CP_{m'}^{\otimes j}$
that cyclically permutes the tensor factors.  Then $\alpha \otimes \beta$ is an isomorphism of $\tkappa''_m \otimes \CP_{\overline{M}}$
with its $\Pi$-conjugate.  Let $y$ be the unique element of $H(G_m,K''_m,\tkappa''_m \otimes \CP_{\overline{M}})$ supported
on $K''_m \Pi K''_m$ that takes the value $\alpha \otimes \beta$ at $\Pi$.  Since $\Pi$ normalizes $K''_m$, and $\alpha \otimes \beta$
is an isomorphism, left multiplication by $y$ induces an isomorphism:
$$H(G_m,K''_m,\tkappa''_m \otimes \CP_{\overline{M}})_{K''_m x K''_m} \rightarrow H(G_m,K''_m,\tkappa''_m \otimes \CP_{\overline{M}})_{K''_m \Pi x K''_m}$$  
for all $x$, and an analogous statement holds for right multiplication.

Now for $1 \leq i \leq j$, let $\lambda_i = w^{-i} \Pi^i$.  Then $\lambda_i$ is a central and positive (but not strictly positive) element of $M$,
and the product of the $\lambda_i$ is strictly positive.  Let $x_i$ be the element of $H(M,K_M,\tkappa_M \otimes \CP_{\overline{M}})$
giving the action of $\lambda_i$ on $\cInd_{K_M}^M \tkappa_M \otimes \CP_{\overline{M}}$.  It suffices to show $T^+ x_i$ is invertible
for each $i$; equivalently, it suffices to show that $(T^+ x_i) y^{-i}$ is an invertible element $v_i$ of 
$H(G_m,K''_m,\tkappa''_m \otimes \CP_{\overline{M}})$.

Since $v_i$ is supported on $K''_m w^{-i} K''_m$, and the subalgebra of $H(G_m,K''_m,\tkappa''_m \otimes \CP_{\overline{M}})$ supported
on $K_m$ is isomorphic to the endomorphism ring of $i_{\overline{P}_m}^{\overline{G}_m} \CP_{\overline{M}}$, we may interpret
$v_i$ as being induced from an endomorphism $\overline{v}_i$ of $i_{\overline{P}_m}^{\overline{G}_m} \CP_{\overline{M}}$.
Exactly as in section~\ref{sec:zero endomorphisms}, one verifies that $\overline{v}_i$ is the automorphism of this space described in
Theorem 2.4 of~\cite{howlett-lehrer}, completing the argument.
\end{proof}

Our construction of the $\Theta_{i,m}$ now proceeds in a manner similar to that of 
section~\ref{sec:zero endomorphisms}.
In particular we obtain a map:
$$E_{m'}^{\otimes j} \rightarrow H(G_m,K''_m,\tkappa''_m \otimes \CP_{\overline{M}}) \cong 
H(G_m,K_m,\tkappa_m \otimes i_{\overline{P}}^{\overline{G}_m} \CP_{\overline{M}}).$$

Recall that $\CP'_m$ is the image of $i_{\overline{P}}^{\overline{G}_m}$ in
$\CP_m$ under the natural map from the former to the latter.  Exactly as in
section~\ref{sec:zero endomorphisms}, we have:

\begin{proposition} \label{prop:P and P'}
There is a surjection:
$$H(G_m,K_m,\tkappa_m \otimes \CP_m) \rightarrow H(G_m,K_m,\tkappa_m \otimes \CP'_m)$$
compatible with the inclusion of $\cInd_{K_m}^{G_m} \tkappa_m \otimes \CP_m$ in
$\cInd_{K_m}^{G_m} \tkappa_m \otimes \CP'_m$ and the actions of the respective Hecke operators
on these spaces.  Moreover, this surjection is an isomorphism away from double cosets of the
form $K_m z_{m,m}^r K_m$.
\end{proposition}
\begin{proof}
The proof is identical to that of Proposition~\ref{prop:surjection}, except that
elements of the form $z_{m,m}^r$ are no longer central, so the statement must be rephrased slightly.
\end{proof}

Similarly, in a manner identical to the proof of Lemma~\ref{lemma:depth zero center}, we have:

\begin{lemma} \label{lemma:G cover center}
Let $c$ be a central element of $H(G_m,K_m,\tkappa_m \otimes i_{\overline{P}}^{\overline{G}_m} \CP_{\overline{M}})$.
Then $c$ preserves the kernel of the surjection:
$$\cInd_{K_m}^{G_m} \tkappa_m \otimes i_{\overline{P}}^{\overline{G}_m} \CP_{\overline{M}}
\rightarrow \cInd_{K_m}^{G_m} \tkappa_m \otimes \CP'_{\sigma_m},$$
and thus descends to an element of $H(G_m,K_m,\tkappa_m \otimes \CP'_{\sigma_m})$.
\end{lemma}

With these results in hand we return to the inductive construction.
Fix $m'$ immediately preceding $m$,
and suppose that we have constructed elements $\Theta_{i,m'}$ as in Theorem~\ref{thm:compatibility}.
Suppose further that for $1 \leq i < m'$, the element $\Theta_{i,m'}$ is supported away from double cosets
of the form $K_{m'} z_{m',m'}^r K_{m'}$, 
and that $\Theta_{m',m'}$ is an element of $E_{K_{m'},\tau_{m'}}$
supported on $K_{m'} z_{m,m} K_{m'}$, and is of the form $\phi \otimes 1$, where $\phi$ lies in
$I_{z_{m',m'}}(\tkappa_{m'})$.  (These stipulations hold when $m' = 1$ by construction, and we will show that 
the inductive construction of the $\Theta_{i,m'}$ implies these conditions for each larger $m'$ as well.)
We now turn to constructing the elements $\Theta_{i,m}$.

As in section~\ref{sec:zero endomorphisms}, we first define $\tTheta_{i,m}$ be the element of $E_{m'}^{\otimes j}$
defined by the formula:
$$\tTheta_{i,m} = \sum\limits_{r_1 + \dots + r_j = i} \Theta_{r_1,m'} \otimes \dots \otimes \Theta_{r_j,m'},$$
and let $\overline{\Theta}_{i,m}$ be the image of $\tTheta_{i,m}$ in 
$H(G_m,K_m,\tkappa_m \otimes i_{\overline{P}}^{\overline{G}_m} \CP_{\overline{M}})$.  For $i < m$,
$\overline{\Theta}_{i,m}$ is supported away from cosets of the form $K_m z_{m,m}^r K_m$.
By contrast, $\overline{\Theta}_{m,m}$ is supported on $K_m z_{m,m} K_m$, and has the form $\phi \otimes 1$
for an element $\phi$ of $I_{z_{m,m}}(\tkappa)$.

The arguments of section~\ref{sec:zero endomorphisms} show that for any $s$ with $\ell$-regular part
$(s'_1)^m$, and any $s_1, \dots, s_j$ in $\overline{G}_{m'}$ such that the ``block'' matrix with
blocks $s_1, \dots, s_j$ is conjugate to $s$, the action of $\overline{\Theta}_{i,m}$ on the summand
$\cInd_{K_m}^{G_m} \tkappa_m \otimes i_{\overline{P}}^{\overline{G}_m} \bigotimes_r \St_{s_r}$
of $\cInd_{K_m}^{G_m} \tkappa_m \otimes i_{\overline{P}}^{\overline{G}_m} \CP_{\overline{M}} \otimes \overline{\CK}$
is via the element $\theta_{i,s}$ of $A_{M_s,\pi_s}$.  In particular $\overline{\Theta}_{i,m}$ is central in
$H(G_m,K_m,\tkappa_m \otimes i_{\overline{P}}^{\overline{G}_m} \CP_{\overline{M}})$.

The construction of the elements $\Theta_{i,m}$ as in Theorem~\ref{thm:compatibility} is
now straightforward.  Consider the image of $\overline{\Theta}_{i,m}$ in
$H(G_m,K_m,\tkappa_m \otimes \CP'_m)$.  For $i < m$, this image is supported away from
cosets of the form $K_m z_{m,m}^r K_m$ and hence lifts uniquely to an element of $E_m$ supported
away from $K_m z_{m,m}^r K_m$.  Denote this element by $\Theta_{i,m}$; it is clear from our calculations
that this has the claimed properties.  For $i = m$, the element $\overline{\Theta}_{m,m}$ is supported
on $K_m z_{m,m} K_m$ and has the form $\phi \otimes 1$ for some $\phi$ in $I_{z_{m,m}}(\tkappa)$;
we let $\Theta_{m,m}$ denote the element $\phi \otimes 1$ of $H(G_m,K_m,\tkappa_m \otimes \CP_m)_{K_m z_{m,m} K_m}$.
Again, it is clear that this element maps to $\theta_{m,s}$ for all $s$.

The proof of Theorem~\ref{thm:compatibility} is thus complete.\qed

We now turn to considering the implications of Theorem~\ref{thm:compatibility}.
Our construction of the $\Theta_{i,m}$ gives us control over the mod $\ell$ cuspidal supports
of the pairs $(M_s,\pi_s)$.  

\begin{theorem} \label{thm:cusp support}
\begin{enumerate}
\item Let $V$ be an irreducible cuspidal representation of $G_m$ over $k$ containing $(K_m,\tau_m)$.  Then the supercuspidal
support of $V$ is inertially equivalent to $(L_m,\pi_1^{\otimes m})$.
\item For each semisimple conjugacy class $s$ in $\overline{G}_m$ with $\ell$-regular part $(s'_1)^m$, the supercuspidal
$\overline{\CK}$-representation $\pi_s$ of $M_s$ has mod $\ell$ inertial supercuspidal support equivalent to $(L_m,\pi_1^{\otimes m})$.
\item Let $V$ be an irreducible cuspidal representation of $G_r$ over $k$ whose mod $\ell$ inertial supercuspidal support
is equivalent to $(L_r,\pi_1^{\otimes r})$.  Then $r$ lies in $\{1, e_{q^f}, \ell e_{q^f}, \dots \}$, and $V$ 
contains $(K_r,\tau_r)$.
\item Let $V$ be an irreducible supercuspidal representation of $G_r$ over $\overline{\CK}$ whose mod $\ell$ inertial supercuspidal support
is equivalent to $(L_r,\pi_1^{\otimes r})$.  Then $r$ lies in $\{1, e_{q^f}, \ell e_{q^f}, \dots \}$, and $V$ is an unramified twist
of $\pi_s$ for some irreducible semisimple conjugacy class $s$ in $\overline{G}_r$ with $\ell$-regular part $(s'_1)^r$.
\end{enumerate}
\end{theorem}
\begin{proof}
These claims all follow easily from results of Vigneras in~\cite{vig98}.  However, they also follow directly from what
we have proven above.  For completeness, and with an eye towards potential generalizations to situations not covered by~\cite{vig98},
we give this alternative argument:

For the first claim, note that we have a surjection $\CP_{K_m,\tau_m} \rightarrow V$.  Let $\tV$ be a cuspidal
lift of $V$ to a representation over $\overline{K}$; since $\tV$ admits a nonzero map from
$\CP_{K_m,\tau_m}$, $\tV$ must lie in $\Rep_{\overline{K}}(G)_{M_s,\pi_s}$ as above.  Moreover, since $\tV$ is
cuspidal we must have $M_s = G$ and $\pi_s$ cuspidal; in this case $\tV$ is an unramified twist of $\pi_s$.
We thus have an action of $E_{K_m,\tau_m}$ on $\tV$ via the maps:
$$E_{K_m,\tau_m} \rightarrow A_{M_s,\pi_s} \rightarrow \End_{\overline{K}[G]}(\tV) \cong \overline{K}.$$
This action preserves the image of $\CP_{K_m,\tau_m}$ in $\tV$ and thus reduces modulo $\ell$ to a map
$\alpha: E_{K_m,\tau_m} \rightarrow k$.  If we let $E_{K_m,\tau_m}$ act on $V$ via $\alpha$, then this action
is compatible with the surjection of $\CP_{K_m,\tau_m}$ onto $V$.

Let ${\mathfrak m}$ be the kernel of $\alpha$, and let $V'$ be the quotient
$\CP_{K_m,\tau_m}/{\mathfrak m} \CP_{K_m,\tau_m}$.  Every endomorphism of a quotient of $V'$ lifts to an element
of $E_{K_m,\tau_m}$, and is thus a scalar; this implies that the cosocle of $V'$ is absolutely irreducible.  Since
we have a surjection $V' \rightarrow V$, the cosocle of $V'$ is isomorphic to $V$.

On the other hand, we have a map $E_{K_m,\tau_m} \rightarrow W(k)[Z_{0,m}]^{W_{M_{0,m}}((s'_1)^m)}$
giving the action of $E_{K_m,\tau_m}$ on $\cInd_{K_m}^{G_m} \tkappa_m \otimes \St_{(s'_1)^m}$, and the target of
this map is generated by the images of the $\Theta_{i,m}$.  It is then straightforward to verify that
the map $\alpha$ coincides with the composition:
$$E_{K_m,\tau_m} \rightarrow W(k)[Z_{0,m}]^{W_{M_{0,m}}((s'_1)^m)} \rightarrow k,$$
where the second map takes the images of
$\Theta_{1,m}, \dots, \Theta_{m-1,m}$ to zero and takes the image of $\Theta_{m,m}$ to $\alpha(\Theta_{m,m})$.
(This amounts to verifying that $\Theta_{i,m}$ acts by zero on $V$ for $1 \leq i \leq m-1$, but this is clear since
$\tV$ is cuspidal: the action on $V$ thus factors through the action of $E_{K_m,\tau_m}$ on $\tV$
and for $1 \leq i \leq m-1$, the elements $\Theta_{i,m}$ annihilate the cuspidals.)

Let $W$ be the image of $\CP_{K_m,\tau_m}$ in $\cInd_{K_m}^{G_m} \tkappa_m \otimes \St_{(s'_1)^m}$.  We have
a surjection $V' \rightarrow W/{\mathfrak m} W$, and hence in particular $V$ is isomorphic to a subquotient
of $W$.  On the other hand, since $W \otimes \overline{\CK}$ is isomorphic to $\cInd_{K_m}^{G_m} \tkappa_m \otimes \St_{(s'_1)^m}$,
it follows that $W \otimes \overline{\CK}$ lies in $\Rep_{\overline{\CK}}(G)_{M_{(s'_1)^m},\pi_{(s'_1)^m}}$, and thus that
every subquotient of $W$ has mod $\ell$ inertial supercuspidal support $(L_m,\pi_1^{\otimes m})$.

For the second claim, note that it suffices to prove this when $s$ is irreducible, and in this case $\pi_s$ contains
$(K_m,\tkappa \otimes \St_s)$.  Twisting so that $\pi_s$ is integral, we find that the mod $\ell$ reduction of $\pi_s$
contains $(K_m,\tau_m)$ and so we are done by claim (1).

Now let $V$ be an irreducible cuspidal representation of $G_r$ over $k$ for some $r$.  Then $V$ contains a maximal distinguished
cuspidal $k$-type $(K',\tau')$.  Write $\tau' = \kappa' \otimes \sigma'$.  Then $\sigma'$ is attached to an $\ell$-regular
semisimple conjugacy class $t$ in $\overline{G}_r$; moreover, as $\sigma'$ is cuspidal, $t$ is the $\ell$-regular part of
an irreducible conjugacy class and thus has the form $(t')^{r'}$ for some irreducible $\ell$-regular conjugacy class $t'$ and
some positive integer $r'$ dividing $n_1r$.  Let $n'_1 = \frac{n_1r}{r'}$ and fix a maximal order in $\GL_{n'_1}(F)$.  Then
the endo-class $(\Theta',0,\beta')$ attached to $(K',\tau')$ associates to this maximal order a compact open subgroup
$K_1'$ of $\GL_{n'_1}(F)$ and a representation $\kappa_1$ of $K_1'$ (defined up to twist).  Set $\tau'_1 = \kappa'_1 \otimes \sigma_{t'}$,
where $\sigma_{t'}$ is the supercuspidal representation associated to $t'$; then $(K'_1,\tau'_1)$ is a maximal distinguished
supercuspidal type, and the construction of this section associated to this type a family $(K'_m,\tau'_m)$ for suitable $m$.
Moreover, we may
choose the twist of $\kappa_1$ so that $(K',\tau')$ is equivalent to the type $(K'_r,\tau'_r)$.  Then claim (1) above shows
that $V$ has supercuspidal support inertially equivalent to $(\pi')^{\otimes r'}$, where $\pi'$ is some supercuspidal $k$-representation
of $\GL_{n'_1}(F)$ containing $(K'_1,\tau'_1)$.  But then $r = r'$ and $\pi'$ and $\pi$ are unramfied twists of each other, so
$(K'_1,\tau'_1)$ is equivalent to $(K_1,\tau_1)$.  It now follows that $\sigma_{t'}$ is a unramified twist of $\sigma_1$ by a character,
so that $s'$ and $t'$ differ by an element of the center of $\overline{G}_1$.  Thus $r$ must lie in $\{1,e_{q^f}, \ell e_{q^f}, \dots\}$
so the type $(K_r,\tau_r)$ makes sense and is equivalent to $(K',\tau')$.  Thus $V$ contains $(K',\tau')$ as claimed.

Claim (4) follows from (3) easily by reducing an unramified twist of $V$ modulo $\ell$.
\end{proof}

We conclude this section by giving a complete description of $E_{K_m,\tau_m}$ for
small $m$; that is, for $m < \ell$.  There are two cases to consider; in the first,
$m = 1$, and in the second $m = e_{q^f} > 1$.

When $m=1$, the element $\Theta_{1,1}^r$ of $E_{K_1,\tau_1}$ generates
$(E_{K_1,\tau_1})_{K_1 z_{1,1}^r K_1}$ as an $E_{\sigma_1}$-module for all $r$, and hence we have
$$E_{K_1,\tau_1} = E_{\sigma_1}[\Theta_{1,1}^{\pm 1}].$$
This case was already studied by Dat in~\cite{dat}; in particular Dat shows that
$E_{\sigma_1}$ is a universal deformation ring of $\sigma_1$, and that, after completing
at any maximal ideal of characteristic $\ell$, $E_{K_1,\tau_1}$ becomes the universal
deformation ring of the corresponding supercuspidal representation.

Whem $m > 1$ but $m < \ell$, then any $s \neq (s'_1)^m$ with $s^{\reg} = (s'_1)^m$
is irreducible.  We have an ideal $I^{\cusp}$ of $E_{\sigma_m}$ that is the kernel of
the action of $E_{\sigma_m}$ on $\St_{(s'_1)^m}$.

\begin{proposition} \label{prop:e_q}
When $1 < m < \ell$, we have an isomorphism:
$$E_{K_m,\tau_m} \cong E_{\sigma_m}[\Theta_{1,m}, \dots, \Theta_{m,m}^{\pm 1}]/\<\Theta_{1,m}, \dots, \Theta_{m - 1,m}\>\cdot I^{\cusp}.$$
\end{proposition}
\begin{proof}
As $E_{\sigma_m}$ and the $\Theta_{i,m}$ are contained in $E_{K_m,\tau_m}$, we have a map:
$$E_{\sigma_m}[\Theta_{1,m}, \dots, \Theta_{m,m}^{\pm 1}] \rightarrow E_{K_m,\tau_m}.$$
It is easy to see that $\<\Theta_{1,m}, \dots, \Theta_{m-1,m}\>$ map to zero in $H(G_m,K_m,\tkappa_m \otimes \St_s)$
for $s \neq (s'_1)^m$, so that $\<\Theta_{1,m} \dots, \Theta_{m-1,m}\>\cdot I^{\cusp}$ is in the kernel
of the map to $E_{K_m,\tau_m}$.  Denote this ideal by $J$.

To see that this ideal is precisely the kernel, first note that
the restriction of this map to $E_{\sigma_m}[\Theta_{m,m}^{\pm 1}]$ identfies
$E_{\sigma_m}[\Theta_{m,m}^{\pm 1}]$ with the subring of $E_{K_m,\tau_m}$ supported on elements
of the form $K_m z_{m,m}^r K_m$.  Moreover, the ideal of $E_{\sigma_m}[\Theta_{1,m}, \dots, \Theta_{m,m}^{\pm 1}]/J$
generated by the $\Theta_{i,m}$ for $1 \leq i \leq m-1$ is a complementary $W(k)$-submodule
of the subring $E_{\sigma_m}[\Theta_{m,m}^{\pm 1}]$.  When restricted to the ideal generated by
$\Theta_{1,m}, \dots, \Theta_{m-1,m}$, the composed map:
$$E_{\sigma_m}[\Theta_{1,m}, \dots, \Theta_{m,m}^{\pm 1}]/J \rightarrow E_{K_m,\tau_m} \rightarrow H(G_m,K_m,\tkappa_m \otimes \St_{(s'_1)^m})$$
is injective, and its image is precisely the subring of $H(G,K,\tkappa_m \otimes \St_{(s'_1)^m})$ supported away from
elements of the form $K_m z_{m,m}^r K_m)$.  From this we see that the map
$$E_{\sigma_m}[\Theta_{1,m}, \dots, \Theta_{m,m}^{\pm 1}]/J \rightarrow E_{K_m,\tau_m}$$
is injective.

To see surjectivity, note that the map $E_{K_m,\tau_m} \rightarrow H(G_m,K_m,\tkappa_m \otimes \St_{(s'_1)^m})$ is an isomorphism
away from double cosets of the form $K_m z_{m,m}^r K_m$, so any element of $E_{K_m,\tau_m}$ supported away from the latter
is in the image of the map from $E_{\sigma_m}[\Theta_{1,m}, \dots, \Theta_{m,m}^{\pm 1}]$.  On the other hand we have already
seen that every element supported on a double coset $K_m z_{m,m}^r K_m$ is in the image, so the result is proven.
\end{proof}

\begin{remark} \rm When $\ell > n$, and $\sigma$ is not supercuspidal, Paige gives an explicit
description of $E_{\sigma}$ as a $W(k)$-algebra in~\cite{paige}, Theorem 4.11.
The above proposition thus gives a complete description of $E_{K,\tau}$ in this case.
\end{remark}

\section{Finiteness results} \label{sec:fg}

Our next goal is to establish fundamental finiteness results for $\CP_{K,\tau}$.  In order to do
so it will be necessary to work integrally with lattices inside a generic pseudo-type
$(K,\tkappa \otimes \St_s)$.  Choose
a finite extension $\CK'$ of $\CK$ such that $\St_s$ is defined over $\CK'$, and let $\OO'$ be
the ring of integers in $\CK'$.  We can then consider $\OO'$-lattices $L_s$ inside $\St_s$,
and consider the ``integral generic pseudo-type'' $(K,\tkappa \otimes L_s)$, and try to determine the structure
of $\cInd_K^G \tkappa \otimes L$ as a module over $H(G,K,\tkappa \otimes L_s)$.

There will be two lattices in $\St_s$ of particular interest to us.  We construct the first of these
as follows: denote by $\overline{M}$ the Levi subgroup $\overline{M}_s$ of $\overline{G}$.
The representation $\St_{\overline{M},s}$ is irreducible, cuspidal and defined over $\CK'$,
and remains irreducible when reduced mod $\ell$.  There is thus a $\overline{M}$-stable $\OO'$-lattice 
$L_{\overline{M}}$ in $\St_{\overline{M},s}$, and such an $L_{\overline{M}}$ is unique up to homothety.
Then $i_{\overline{P}}^{\overline{G}} L_{\overline{M},s}$ is an $\OO'$-lattice in $I_s$.  Let $L_s$ be the image
of this lattice in $\St_s$.

The second lattice we will make use of will be denoted $L'_s$, and is defined as follows: the representation
$\St_s$ is a direct summand of $\CP_{\sigma} \otimes_{W(k)} \CK'$.  Let $L'_s$ be the image of
$\CP_{\sigma} \otimes_{W(k)} \OO'$ under the projection to $\St_s$; this defines $L'_s$ up to homothety.
The lattice $L'_s$ is the one that is of interest to us in applications, but is more complicated;
we will study it via its relationship with $L_s$.  Note that there exist $a,b$ such that
$\ell^a L_s \subset L'_s \subset \ell^b L_s$.  Let $\tau_{L_s}$ and $\tau_{L'_s}$ denote the
representations $\tkappa \otimes L_s$, and $\tkappa \otimes L'_s$.

The pair $(K,\tau_{L_s})$ is not difficult to understand; indeed, the arguments of
section~\ref{sec:generic} apply.  In particular, consider the pair $(K_M,\tkappa_M \otimes L_{\overline{M}})$, where $M$, $K_M$,
and $\tkappa_M$ are as in section~\ref{sec:generic}.  It follows from Lemma~\ref{lemma:lattice G-cover} that
$(K'',\tkappa'' \otimes L_{\overline{M}})$
is a $G$-cover of $(K_M,\tkappa_M \otimes L_{\overline{M}})$, and that the center of
$H(G,K'',\tkappa'' \otimes L_{\overline{M}})$ is isomorphic to $\OO'[Z_s]^{W_{\overline{M}}(s)}$.  
(Indeed, a choice of a compatible family of cuspidals for the tower of types that contains $(K,\tau)$
gives rise to explicit isomorphisms: 
$$H(M,K_M,\tkappa_M \otimes L_{\overline{M}}) \cong \OO'[Z_s]$$
$$Z(H(G,K'',\tkappa'' \otimes L_{\overline{M}})) \cong \OO'[Z_s]^{W_{\overline{M}}(s)}.$$
Henceforth we fix such a choice.)
The intertwining calculations
of section~\ref{sec:generic} give rise to a support preserving isomorphism
$H(G,K'',\tkappa'' \otimes L_{\overline{M}})$ with $H(G,K_{\overline{P}},\tkappa_{\overline{P}} \otimes L_{\overline{M}})$, and an isomorphism
of the latter with $H(G,K,\tkappa \otimes i_{\overline{P}}^{\overline{G}} L_{\overline{M}})$.  

We now observe:
\begin{lemma} \label{lem:center}
Let $x$ be a central element of $H(G,K,\tkappa \otimes i_{\overline{P}}^{\overline{G}} L_{\overline{M}})$.  Then $x$
descends to an endomorphism of $\cInd_K^G \tkappa \otimes L_s$ via the surjection of $i_{\overline{P}}^{\overline{G}} L_{\overline{M}}$
onto $L_s$.
\end{lemma}
\begin{proof}
The results of section~\ref{sec:generic}, particularly the discussion after the proof
of Proposition~\ref{prop:support}, show that this holds after inverting $\ell$.  We have a surjection:
$$\cInd_K^G \tkappa \otimes (i_{\overline{P}}^{\overline{G}} L_{\overline{M}}) \rightarrow \cInd_K^G \tkappa \otimes L_s,$$
and it suffices to show that $x$ preserves the kernel of this surjection.  But as this holds after inverting $\ell$,
and both the left-hand and right-hand sides are $\ell$-torsion free, the result follows.
\end{proof}

We next turn to questions of admissibility:
\begin{lemma}
The module $\cInd_{K_M}^M \tkappa \otimes L_M$ is an admissible $H(M,K_M,\tkappa_M \otimes L_M)$-module.
\end{lemma}
\begin{proof}
Let $\pi$ be an $\OO'[M]$-module such that $\pi \otimes_{\OO'} \CK'$ is absolutely irreducible and
such that the restriction of $\pi$ to $K_M$ admits a nonzero map from $\tkappa \otimes L_M$.
Then $\Hom_{\OO'[K_M]}(\tkappa \otimes L_M, \pi)$ is a free $\OO'$-module of rank one.
Consider the $\OO'[M]$-module $\pi \otimes_{\OO'} \OO'[M/M_0]$, on which $M$ acts on $\OO'[M/M_0]$
via the natural character $M \rightarrow \OO'[M/M_0]$.  We have
$$\Hom_{\OO'[K_M]}(\tkappa \otimes L_M,\pi) \cong \OO'[M/M_0],$$
and $H(M,K_M,\tau) = \OO'[Z]$ acts on the right hand side via the inclusion of $\OO'[Z]$ in $\OO'[M/M_0]$. 
This yields an isomorphism:
$$(\cInd_{K_M}^M \tkappa \otimes L_M) \otimes_{\OO'[Z_s]} \OO'[M/M_0] \cong \pi \otimes_{\OO'} \OO'[M/M_0].$$
In particular the left hand side is admissible over $\OO'[M/M_0]$, and so
$\cInd_{K_M}^M \tkappa \otimes L_M$ is admissible over $\OO'[Z_s]$.
\end{proof}

\begin{lemma} \label{lem:induction restriction admissibility}
Let $R$ be commutative $W(k)$-algebra, let $P = MU$ be a parabolic
subgroup of $G$, and let $\pi$ be an admissible $R[M]$-module such that for any parabolic
subgroup $P' = M'U'$ of $M$, $r_M^{P'} \pi$ is admissible as an $R[M']$-module.  Then
$i_P^G \pi$ is an admissible $R[G]$-module, and, for any parabolic subgroup $P'' = M''U''$
of $G$, $r_G^{P''} i_P^G \pi$ is an admissible $R[M'']$-module.
\end{lemma}
\begin{proof}
This is an immediate consequence of Bernstein-Zelevinski's filtration of
of the composite functor $r_G^{P''} i_P^G$ (\cite{BZ}, 2.12) together with the
fact that parabolic induction takes admissible representations to admissible representations.
\end{proof}

Let $R$ be the center of
$H(G,K,\tkappa \otimes i_{\overline{P}}^{\overline{G}} L_{\overline{M}})$.  (We may
also regard $R$ as the center 
of $H(G,K'',\tkappa'' \otimes L_{\overline{M}})$.)
By Lemma~\ref{lem:center}, the ring $R$
acts on $\cInd_K^G \tau_{L_s}$.

\begin{proposition}
The module $\cInd_K^G \tau_{L_s}$ is admissible over $R$.
More generally, for any parabolic $P' = M'U'$ in $G$, the module
$r_G^{P'} \cInd_K^G \tau_{L_s}$ is admissible over $R$.
\end{proposition}
\begin{proof}

We have shown that we may identify the
center of $H(G,K'',\tkappa'' \otimes L_{\overline{M}})$
with the ring $\OO'[Z_s]^{W_{\overline{M}}(s)}$.  This ring acts on
$\cInd_K^G \tau_{L_s}$ via its identification with $R$.

Consider the module $\tkappa_M \otimes L_{\overline{M}}$; this is a lattice in a maximal distinguished
cuspidal $M$-type.  Thus 
the $W(k)[M]$-module $\cInd_{K_M}^{M} \tkappa_M \otimes L_{\overline{M}}$ is
admissible and cuspidal over the Hecke algebra $H(M,K_M,\tkappa_M \otimes L_{\overline{M}})$, and the latter is
isomorphic to $\OO'[Z_s]$.  It follows that $i_P^G \cInd_{K_M}^{M} \tkappa_M \otimes L_{\overline{M}}$
is an admissible $\OO'[Z_s]$-module, as is $r_G^{P'} i_P^G \cInd_{K_M}^M \tkappa_M \otimes L_{\overline{M}}$
for any $P'$.

The subalgebra $\OO'[Z_s]^{W_{\overline{M}}(s)}$ of $\OO'[Z_s]$
acts on both $i_P^G \cInd_{K_M}^M \tkappa_M \otimes L_{\overline{M}}$
and on $\cInd_{K''}^G \tkappa'' \otimes L_{\overline{M}}$.  In fact, 
any map between these two modules is $\OO'[Z_s]^{W_{\overline{M}}(s)}$-equivariant.  To see this, note that
we can identify $\OO'[Z_s]^{W_{\overline{M}}(s)} \otimes_{\OO'} \overline{\CK}$ with
the center $A_{M_s,\pi_s}$ of $\Rep_{\overline{\CK}}(G)_{M_s,\pi_s}$, and thus the action of $\OO'[Z_s]^{W_{\overline{M}}(s)}$
is compatible with every map between the two modules after inverting $\ell$.  Since both modules are $\ell$-torsion free the claim follows.

As $\OO'[Z_s]$ is a finitely generated $\OO'[Z_s]^{W_{\overline{M}}(s)}$-module, it now suffices to produce an
embedding of $\cInd_{K''}^G \tkappa'' \otimes L_{\overline{M}}$ in $i_P^G \cInd_{K_M}^M \tkappa_M \otimes L_{\overline{M}}$.

We first consider this question over $\overline{\CK}$.  In this situation the two modules
are actually isomorphic.  To see this, note that after tensoring with
$\overline{\CK}$, $\tkappa'' \otimes L_{\overline{M}}$ becomes $\tkappa'' \otimes \St_{\overline{M},s}$, which
is, in the language of~\cite{BK-semisimple}, a semisimple type.  In particular,
$\Hom_{K''}(\tkappa'' \otimes \St_{\overline{M},s}, -)$ yields an equivalence between the block of
$\Rep_{\overline{\CK}}(G)$ corresponding to the type $\tkappa'' \otimes \St_{\overline{M},s}$
and the category of $H(G,K'',\tkappa'' \otimes \St_{\overline{M},s})$-modules.  On the other hand, it is an
easy consequence of Theorem~\ref{thm:G-cover} that the space
$$\Hom_{K''}(\tkappa'' \otimes \St_{\overline{M},s}, i_P^G \cInd_{K_M}^M \tkappa_M \otimes \St_{\overline{M},s})$$
is a free $H(G,K'',\tkappa'' \otimes \St_{\overline{M},s})$-module of rank one.
We may thus identify $\cInd_{K''}^G \tkappa'' \otimes \St_{\overline{M},s}$ with 
$i_P^G \cInd_{K_M}^M \tkappa_M \otimes L_{\overline{M}}$.

Descending from $\overline{\CK}$ to $\CK'$, we find an isomorphism of
$\cInd_{K''}^G \tkappa'' \otimes L_{\overline{M}} \otimes \CK'$ into 
$i_P^G \cInd_{K_M}^M \tkappa_M \otimes L_{\overline{M}} \otimes \CK'$.  It suffices to show
that, after multiplying by a sufficiently large power of $\ell$, such a map
takes $\cInd_{K''}^G \tkappa'' \otimes L_{\overline{M}}$ into $i_P^G \cInd_{K_M}^M \tkappa_M \otimes L_{\overline{M}}.$
But this is clear since $G$-maps from the former to the latter are, by Frobenius reciprocity, the same as $K''$-maps
from $\tkappa'' \otimes L_{\overline{M}}$ into the latter, and $\tkappa'' \otimes L_{\overline{M}}$ is finitely
generated as an $\OO'$-module.
\end{proof}

We now compare $\cInd_K^G \tau_{L_s}$ and $\cInd_K^G \tau_{L'_s}$.  The inclusions
$\ell^a L_s \subset L'_s \subset \ell^b L$
give rise to inclusions:
$\ell^a \cInd_K^G \tau_{L_s} \subset \cInd_K^G \tau_{L'_s} \subset \ell^b \cInd_K^G \tau_{L_s}$.  

The endomorphism ring $E_{K,\tau}$ of $\CP_{K,\tau}$ preserves the factor
$\cInd_K^G \tkappa \otimes \St_s$ of $\CP_{K,\tau} \otimes \overline{\CK}$,
and hence preserves the image of $\CP_{K,\tau} \otimes \OO'$ in $\cInd_K^G \tkappa \otimes \St_s$.
This image is equal to $\cInd_K^G \tau_{L'_s}$.
In particular we obtain a map of $E_{K,\tau} \otimes \OO'$ into $H(G,K,\tau_{L'_s})$.

For some $m$, we have $(K,\tau) = (K_m,\tau_m)$, where $K_m$ and $\tau_m$
are as in section~\ref{sec:endomorphisms}.
Our choice of compatible family of cuspidals in section~\ref{sec:endomorphisms}
gives an isomorphism:
$$H(G,K,\tkappa \otimes \St_s) \cong
\overline{\CK}[Z_s]^{W_{\overline{M}_s}(s)}.$$  
The isomorphism of $R \otimes \overline{\CK}$ with $H(G,K,\tkappa \otimes \St_s)$ identifies the ring $R$
with the subalgebra $\OO'[Z_s]^{W_{\overline{M}_s}(s)}$.  We also identify
$H(G,K,\tau_{L'_s})$ with another, yet-to-be-determined $\OO'$-subalgebra that contains
$E_{K,\tau} \otimes \OO'$.  

We observe that as a subalgebra of $\overline{\CK}[Z_s]^{W_{\overline{M}_s}(s)}$, 
$H(G,K,\tau_{L'_s})$ contains the images of the elements $\Theta_{i,m}$ of $C_{K,\tau}$, 
as well as the image of $\Theta_{m,m}^{-1}$.  We may identify these with the elements
$\theta_{i,s}$ of $\overline{\CK}[Z_s]^{W_{\overline{M}}(s)}$, and hence with elements of $R$.
We then observe:

\begin{proposition} \label{prop:scalar fg}
The algebra $W(k)[Z_s]^{W_{\overline{M}}(s)}$ is a finitely generated module over
$W(k)[\theta_{1,s}, \dots, \theta_{m,s}, \theta^{-1}_{m,s}]$ In particular, $R$ is a finitely generated module
over this subalgebra.
\end{proposition}
\begin{proof}
Let $s_1, \dots, s_r$ be the irreducible constituents of $s$, and let
$z_1, \dots, z_r$ be the elements of $Z_s$ such that, when considered as an element
of $\GL_{\frac{n}{ef}}(E)$, $z_i$ is scalar with entries $\unif_E$ on the block
of $Z_s$ corresponding to $s_i$, and the identity on all other blocks.  Let $d_i$
be the degree of $s_i$ over $\FF_{q^f}$, and let $d$ be the degree of $s'$
over $\FF_{q^f}$, where $s^{\reg} = (s')^m$.  Then, by definition, we have:
$$\theta_{i,s} = \sum_{S} \prod_{j \in S} z_j, $$ 
where $S$ runs over those subsets of ${1, \dots, r}$ such that
$$\sum_{j \in S} d_j = di.$$

Now consider the polynomial
$$P(t) = \prod_{i=1}^r (t^{\frac{d_i}{d}} + z_i).$$
For $1 \leq i \leq r$, the coefficient of $t^{r-i}$ in $P(t)$
is $\theta_{i,s}$.  It follows that the elements $(-z_i)^{\frac{d}{d_i}}$
are integral over $W(k)[\theta_{1,s}, \dots, \theta_{m,s}]$, and so the
elements $z_i$ themselves are.  As $W(k)[Z_s]$ is generated by the $z_i$,
together with $\theta^{-1}_{m,s}$, it follows that
$W(k)[Z_s]$ is integral, and hence finitely generated as a module, over
$W(k)[\theta_{1,s}, \dots, \theta_{m,s}, \theta^{-1}_{m,s}]$,
and the result is immediate.
\end{proof}

We now show:

\begin{proposition} \label{prop:steinberg admissibility}
The module $\cInd_K^G \tau_{L'_s}$ is an admissible $E_{K,\tau}[G]$-module
Moreover, for any $P' = M'U'$
in $G$, $r_G^{P'} \cInd_K^G \tau_{L'}$ is admissible as a $E_{K,\tau}[M']$-module.
\end{proposition}
\begin{proof}
The module $r_G^{P'} \cInd_K^G \tau_{L_s}$ is admissible over $R$, 
which we have identified with $\OO'[Z_s]^{W_{\overline{M}}(s)}$.  It is thus also
admissible
over $\OO'[\theta_{1,s}, \dots, \theta_{m,s},\theta_{m,s}^{-1}]$, 
by Proposition~\ref{prop:scalar fg}.  The $\theta_{i,s}$ preserve both
$\cInd_K^G \tau_{L_s}$ and $\cInd_K^G \tau_{L'_s}$, and any embedding of
the latter in the former is equivariant for the $\theta_{i,s}$ (as the action of these elements is via
the Bernstein center.)  Fix such an embedding; this yields an embedding of
$r_G^{P'} \cInd_K^G \tau_{L'_s}$ 
in $r_G^{P'} \cInd_K^G \tau_{L_s}$ that is compatible with the action of the
elements $\theta_{i,s}$.  As the former is admissible over
$\OO'[\theta_{1,s}, \dots, \theta_{m,s},\theta^{-1}_{m,s}]$,
the latter must be as well, and the result follows.
\end{proof}

We now return to the study of $\CP_{K,\tau}$.  We choose $\CK'$ (and by extension
$\OO'$ sufficiently large that every map from $E_{\sigma}$ to $\overline{\CK}$
has image contained in $\OO'$.  Then $\St_s$ is defined over $\CK'$
for every $s$.  Moreover, the embeddings:
$$\CP_{K,\tau} \otimes \overline{\CK} \hookrightarrow \bigoplus_s \cInd_K^G \tkappa \otimes \St_s$$
$$E_{K,\tau} \otimes \overline{\CK} \hookrightarrow \prod_s H(G,K,\tkappa \otimes \St_s)$$
factor through embeddings:
$$\CP_{K,\tau} \otimes \OO' \hookrightarrow \bigoplus_s \cInd_K^G \tau_{L'_s}$$
$$E_{K,\tau} \otimes \OO' \hookrightarrow \prod_s H(G,K,\tau_{L'_s}).$$

We thus have:
\begin{theorem} \label{thm:P admissibility}
For any parabolic subgroup $P' = M'U'$ of $G$, $r_G^{P'} \CP_{K,\tau}$ is an admissible $C_{K,\tau}[M']$-module
(and hence also an admissible $E_{K,\tau}[M']$-module.)
\end{theorem}
\begin{proof}
For each $s$, $r_G^{P'} \cInd_K^G \tau_{L'_s}$ is admissible over $C_{K,\tau} \otimes \OO'$,
and hence also over $C_{K,\tau}$.  The result is thus immediate from the above decomposition.
\end{proof}

\begin{corollary} The $W(k)$-algebra $E_{K,\tau}$ is a finitely generated $C_{K,\tau}$-module.
\end{corollary}

We now turn to the study of the $W(k)[G]$-module
$\CP_{(M,\pi)}$ for an inertial equivalence class $(M,\pi)$ of irreducible
cuspidal representations of $M$ over $k$.  Write $\pi$ as a tensor product
of irreducible cuspidal representations $\pi_i$ of general linear groups $\GL_{n_i}(F)$,
and for each $i$, let $(K_i,\tau_i)$
be a maximal distinguished cuspidal $k$-type contained in $\pi_i$.

\begin{theorem} \label{thm:admissibility}
Let $R$ be the tensor product, over $W(k)$, of the rings $E_{K_i,\tau_i}$
for all $i$.  Then $\CP_{(M,\pi)}$ is an admissible $R[G]$-module,
and for any parabolic subgroup $P' = M'U'$ of $G$, $r_G^{P'} \CP_{(M,\pi)}$ is
an admissible $R[M']$-module.
\end{theorem}
\begin{proof}
This is an immediate consequence of Theorem~\ref{thm:P admissibility} and
Lemma~\ref{lem:induction restriction admissibility}.
\end{proof}




\section{The Bernstein decomposition and Bernstein's second adjointness}
\label{sec:bernstein}

Our next goal is to apply our results on the structure of $\CP_{K,\tau}$ to establish
a Bernstein decomposition for the category of smooth $W(k)[G]$-modules.
With the results of the previous section in hand, this is an easy consequence
of Bernstein's second adjointness for the category of smooth $W(k)[G]$-modules,
which is due, in this generality, to Dat in~\cite{dat-adjointness}.  
However, at this point it is not too much work to give
an alternative proof of Bernstein's second adjointness which is quite different in spirit
from Dat's approach.  We thus detour for a moment to show how Bernstein's second
adjointness follows from the results so far.

For technical reasons (namely, the fact that parabolic induction is naturally a
right adjoint, and therefore takes injectives to injectives), it will be useful for us
to work with injective objects rather than the projectives $\CP_{K,\tau}$.  To obtain
a suitable supply of injectives, we define:

\begin{definition} Let $\Pi$ be a smooth $W(k)[G]$-module.  We denote by
$\Pi^{\vee}$ the $W(k)[G]$-submodule of smooth vectors in $\Hom_{W(k)}(\Pi,\CK/W(k))$.
\end{definition}

As $\CK/W(k)$ is an injective $W(k)$-module, and the functor that takes a $W(k)[G]$-module
to the submodule consisting of its smooth vectors is exact, the functor $\Pi \mapsto \Pi^{\vee}$
is exact as well.  Moreover, we have:

\begin{lemma} Let $\Pi$ and $\Pi'$ be smooth $W(k)[G]$-modules.  Then there is a natural
isomorphism: 
$$\Hom_{W(k)[G]}(\Pi, (\Pi')^{\vee}) \rightarrow \Hom_{W(k)[G]}(\Pi',\Pi^{\vee}).$$
\end{lemma}
\begin{proof}
Both $\Hom_{W(k)[G]}(\Pi, (\Pi')^{\vee})$ and $\Hom_{W(k)[G]}(\Pi',\Pi^{\vee})$ are
in bijection with the set of $G$-equivariant pairings $\Pi \times \Pi' \rightarrow \CK/W(k)$
that are smooth with respect to the action of $G$.
\end{proof}

\begin{corollary} If $\Pi$ is a projective $W(k)[G]$-module, then $\Pi^{\vee}$ is injective.
\end{corollary}
\begin{proof}
The functors $\Hom_{W(k)[G]}(-,\Pi^{\vee})$ and 
$\Hom_{W(k)[G]}(\Pi,(-)^{\vee})$ are naturally equivalent, and the latter is exact.
\end{proof}

This duality is well-behaved with respect to normalized parabolic induction:
\begin{lemma} Let $P=MU$ be a Levi subgroup of $G$, and let $\pi$ be 
a smooth $W(k)[M]$-module.  Then there is a natural isomorphism:
$$(i_P^G \pi)^{\vee} \rightarrow i_P^G \pi^{\vee}.$$
In particular, if $\Pi$ is a simple $W(k)[G]$-module with cuspidal (resp. supercuspidal)
support $(M,\pi),$, then $\Pi^{\vee}$ has cuspidal (resp. supercuspidal)
support $(M,\pi^{\vee})$.
\end{lemma}
\begin{proof}
Integration over $G/P$ defines a $G$-equivariant bilinear map:
$$[i_P^G \pi] \times [i_P^G \pi^{\vee}] \rightarrow \CK/W(k),$$
and hence a $G$-equivariant map:
$$i_P^G \pi^{\vee} \rightarrow [i_p^G \pi]^{\vee}.$$
This map is easily seen to be an isomorphism by passing to $U$-invariants for a cofinal family
of sufficiently small compact open subgroups $U$ of $G$.
\end{proof}

Note that if $\Pi$ is an {\em admissible} $W(k)[G]$-module, (or, alternatively, an admissible
$\CK[G]$-module), then $(\Pi^{\vee})^{\vee}$
is naturally isomorphic to $\Pi$.  In particular this is true if $\Pi$ is simple.

If $(K,\tau)$ is a maximal distinguished cuspidal $k$-type, we define $I_{K,\tau}$
to be the injective $W(k)[G]$-module $\CP_{K,\tau^{\vee}}^{\vee}$.  Similarly, if $(M,\pi)$
is a pair consisting of a Levi subgroup of $G$ and an irreducible cuspidal representation $\pi$
of $M$ over $k$, we set $I_{(M,\pi)}$ to
be the $W(k)[G]$-module $\CP_{(M,\pi^{\vee})}^{\vee}$.

\begin{proposition} The $W(k)[G]$-module $I_{(M,\pi)}$ is injective.  Moreover,
every simple $W(k)[G]$-module $\Pi$ with mod $\ell$ inertial cuspidal support $(M,\pi)$
embeds in $I_{(M,\pi)}$.
\end{proposition}
\begin{proof}
For a suitable parabolic subgroup $P$, and suitable maximal distinguished cuspidal types
$(K_i,\tau_i)$, we have:
$$\CP_{(M,\pi^{\vee})}^{\vee} = 
[i_P^G \CP_{K_1,\tau_1^{\vee}} \otimes \dots \otimes \CP_{K_r,\tau_r^{\vee}}]^{\vee}$$
$$= i_P^G I_{K_1,\tau_1} \otimes \dots \otimes I_{K_r,\tau_r}.$$
As $i_P^G$ is a right adjoint of an exact functor (by Frobenius reciprocity), the
latter module is clearly injective.

Now given $\Pi$, $\Pi^{\vee}$ has mod $\ell$ inertial cuspidal support
$(M,\pi^{\vee})$; by Proposition~\ref{prop:cuspidal support} we have
a surjection $\CP_{(M,\pi^{\vee})} \rightarrow \Pi^{\vee}$.  Dualizing, and
using the fact that $(\Pi^{\vee})^{\vee} = \Pi$, we obtain our desired result.
\end{proof}

\begin{proposition}
Let $(M,\pi)$ and $(M',\pi')$ be two pairs consisting of a Levi subgroup of $G$ and
an irreducible cuspidal representation of that Levi subgroup over $k$.
Let $I = I_{(M,\pi)}$; $I' = I_{(M',\pi')}$.
If $\Hom_{W(k)[G]}(I,I')$ is nonzero, then the mod $\ell$ inertial supercuspidal supports
of $(M,\pi)$ and $(M',\pi')$ coincide.
\end{proposition}
\begin{proof}
Set $\CP = \CP_{(M,\pi^{\vee})}$; $\CP' = \CP_{(M',(\pi')^{\vee})}$, so that
$I = \CP^{\vee}$ and $I' = (\CP')^{\vee}$.  We have an injection of
$\Hom_{W(k)[G]}(I,I')$ into $\Hom_{W(k)[G]}((I')^{\vee},I^{\vee})$; the latter is equal
to $\Hom_{W(k)[G]}(\CP^{\vee\vee},(\CP')^{\vee\vee})$.  We also have an embedding:
$$\Hom_{W(k)[G]}(\CP^{\vee\vee},(\CP')^{\vee\vee}) \rightarrow
\Hom_{W(k)[G]}(\CP^{\vee\vee} \otimes \overline{\CK}, (\CP')^{\vee\vee} \otimes \overline{\CK}),$$
so it suffices
to show that the latter is zero unless the mod $\ell$ inertial supercuspidal supports
of $(M,\pi)$ and $(M',\pi')$ coincide.

We have an embedding of $\CP^{\vee\vee} \otimes \overline{\CK}$ into 
$(\CP \otimes \overline{\CK})^{\vee\vee}$.

If $(L,\pi')$ is a pair consisting of a Levi subgroup of $G$ and an irreducible supercuspidal
representation of $L$ over $\overline{\CK}$, and $\pi'$ has mod $\ell$ inertial supercuspidal
support different from $(M,\pi)$, then the projection
$(\CP \otimes \overline{\CK})_{L,\pi}$ of $\CP \otimes \overline{\CK}$ to 
$\Rep_{\overline{\CK}}(G)_{L,\pi}$ vanishes by part (2) of Theorem~\ref{thm:cusp support}, and therefore so does
$(\CP \otimes \overline{\CK})^{\vee\vee}_{L,\pi}$.  It follows that 
$(\CP \otimes \overline{\CK})^{\vee\vee}$
has a direct sum decomposition in which each summand lies in some block of $\Rep_{\overline{\CK}}(G)$
corresponding to an inertial supercuspidal support whose mod $\ell$ reduction is $(M,\pi)$.
Similarly, $(\CP' \otimes \overline{\CK})^{\vee\vee}$ has a direct summand decomposition in which
each summand lies in some block of $\Rep_{\overline{\CK}}(G)$ 
corresponding to an inertial supercuspidal support whose mod $\ell$ reduction
is $(M',\pi')$.

Thus, if the mod $\ell$ inertial supercuspidal supports of 
$(M,\pi)$ and $(M',\pi')\}$ differ, then no summand of
$(\CP \otimes \overline{\CK})^{\vee\vee}$ lies in the same block as any summand of 
$(\CP' \otimes \overline{\CK})^{\vee\vee}$,
and the result follows.
\end{proof}

\begin{corollary}
Every simple subquotient of $I_{(M,\pi)}$ has mod $\ell$ inertial supercuspidal
support equal to that of $(M,\pi)$.
\end{corollary}
\begin{proof}
Let $\Pi$ be a simple subquotient of $I_{(M,\pi)}$, with mod $\ell$ inertial
supercuspidal support
$(M',\pi')$.  Then $\Pi$ embeds in $I_{(M',\pi')}$;
as the latter is injective we obtain a nonzero map $I_{(M,\pi)} \rightarrow
I_{(M',\pi')}$.  The preceding proposition now implies that $(M',\pi')$ has
the same mod $\ell$ inertial supercuspidal support as $(M,\pi)$.
\end{proof}

An immediate corollary is the ``Bernstein decomposition'' for $\Rep_{W(k)}(G)$.
Let $M$ be a Levi subgroup of $G$, and let $\pi$ be an irreducible supercuspidal
representation of $M$ over $k$.
If $\Pi$ is a simple smooth $W(k)[G]$-module with mod $\ell$ inertial supercuspidal support
given by $(M,\pi)$, then the mod $\ell$ {\em cuspidal} support of $\Pi$ falls into
one of finitely many possibile mod $\ell$ inertial equivalence classes. 
Choose representatives $(M_j,\pi_j)$ for these inertial equivalence classes, and let
$I_{[M,\pi]} = I_{(M_1,\pi_1)} \oplus \dots \oplus I_{(M_r,\pi_r)}$.  Then every simple subquotient of 
$I_{[M,\pi]}$
has mod $\ell$ inertial supercuspidal support $(M,\pi)$.  On the other hand, any simple
smooth $W(k)[G]$-module $\pi$ with mod $\ell$ inertial supercuspidal support $(M,\pi)$
has mod $\ell$ inertial cuspidal support $(M_j,\pi_j)$ for some $j$, and hence embeds in $I_{(M_j,\pi_j)}$
(and thus also in $I_{[M,\pi]}$.)  

On the other hand, if $\Pi$ is a simple smooth $W(k)[G]$-module whose mod $\ell$ inertial
supercuspidal support is not in the inertial equivalence class $(M,\pi)$, then there
is an infinite collection of possible inertial equivalence classes into which the
mod $\ell$ inertial cuspidal support of $\Pi$ could fall.  
If we let $I_{\overline{[M,\pi]}}$
denote the direct sum of $I_{(M',\pi')}$ as $(M',\pi')$ runs over a set of representatives
for the inertial equivalence classes of pairs $(M',\pi')$ over $k$ whose supercuspidal support
is {\em not} in the inertial equivalence class $(M,\pi)$,
then no subquotient of $I_{\overline{[M,\pi]}}$ has
mod $\ell$ inertial supercuspidal support equal to $(M,\pi)$, and every simple object
of $\Rep_{W(k)}(G)$ whose mod $\ell$ inertial supercuspidal support is not equivalent to $(M,\pi)$
is a subobject of $I_{\overline{[M,\pi]}}$.

\begin{theorem} \label{thm:decomposition}
The full subcategory $\Rep_{W(k)}(G)_{[M,\pi]}$ of $\Rep_{W(k)}(G)$ consisting of smooth 
$W(k)[G]$-modules $\Pi$ such that every simple subquotient of $\Pi$ has mod $\ell$
inertial supercuspidal support given by $(M,\pi)$ is a block of $\Rep_{W(k)}(G)$.
Moreover, every element of $\Rep_{W(k)}(G)_{[M,\pi]}$ has a resolution by direct sums of
copies of $I_{[M,\pi]}$.
\end{theorem}
\begin{proof}
This is immediate from the above discussion and Proposition~\ref{prop:decompose}.
\end{proof}

Our first application of this Bernstein decomposition will be to establish Bernstein's
second adjointness for smooth $W(k)[G]$-modules.  This will allow us, at last, to conclude that
the modules $\CP_{(M,\pi)}$ are projective.  We follow the argument in the lecture notes
by Bernstein-Rumelhart~\cite{BR}, adapting it as necessary so that it will work in $\Rep_{W(k)}(G)$.

\begin{definition} Let $P = MU$ be a parabolic subgroup of $G$, let $K$ be a compact
open subgroup of $G$ that is decomposed with respect to $P$, and let $\lambda$ be a totally
positive central element of $M$.  Let $T_{\lambda}$ be the element of $H(G,K,1)$ given by
$T^+(1_{K_M\lambda K_M})$.  A smooth $W(k)[G]$-module $\Pi$ is $K,P$-stable,
with constant $c_{K,P,\lambda}$, if there exists a positive integer $c_{K,P,\lambda}$
such that $\Pi^K$ splits as a direct sum:
$$\Pi^K = \Pi^K[T_{\lambda}^{c_{K,P,\lambda}}] \oplus \Pi^K_{T_{\lambda}-{\rm invert}},$$
where $\Pi^K[T_{\lambda}^{c_{K,P,\lambda}}]$ is the $W(k)$-submodule of $\Pi^K$ consisting
of elements killed by $T_{\lambda}^{c_{K,P,\lambda}}$, and $\Pi^K_{T_{\lambda}-{\rm invert}}$
is the maximal $W(k)$-submodule of $\Pi^K$ on which $[\lambda]$ is invertible.
\end{definition}

The key to establishing Berstein's second adjointness will be proving that
for every pair $K,P$, and every supercuspidal inertial equivalence class $(L,\pi)$, all objects of 
$\Rep_{W(k)}(G)_{[L,\pi]}$
are $K,P$-stable.  We first make a few observations:

\begin{lemma} Let $\Pi$ be a smooth $K,P$-stable $W(k)[G]$-module.  Then $\Pi^{\vee}$
is also $K,P$-stable.
\end{lemma}
\begin{proof}
We have 
$$(\Pi^{\vee})^K \cong (\Pi^K)^{\vee} \cong \Pi^K[T_{\lambda}^{c_{K,P,\lambda}}]^{\vee}
\oplus [\Pi^K_{T_{\lambda}-{\rm invert}}]^{\vee};$$
the result follows immediately.
\end{proof}

\begin{lemma} Finite direct sums of $K,P$-stable modules are $K,P$-stable.  Infinite
direct sums of modules which are $K,P$-stable with a uniform constant $c_{K,P,\lambda}$
are $K$-stable.  Kernels and cokernels of maps of $K,P$-stable modules are $K,P$-stable.
\end{lemma}

\begin{lemma} \label{lemma:jacquet}
Let $\Pi$ be a smooth $W(k)[G]$-module.  Then the natural projection:
$$\Pi^K \rightarrow (\Pi_U)^{K_M}$$
identifies $(\Pi_U)^{K_M}$ with $\Pi^K \otimes_{W(k)[T_{\lambda}]} W(k)[T_{\lambda},T_{\lambda}^{-1}]$.
In particular, if $\Pi$ is $K,P$-stable, then the map
$\Pi^K \rightarrow (\Pi_U)^{K_M}$ is surjective, and one has
a direct sum decomposition:
$$\Pi^K = \Pi^K[T_{\lambda}^{c_{K,P,\lambda}}] \oplus (\Pi_U)^{K_M}.$$
This decomposition is independent of $\lambda$.
\end{lemma}
\begin{proof}
We make $(\Pi_U)^{K_M}$ into a $W(k)[T_{\lambda}]$-module by letting $T_{\lambda}$ act
on $(\Pi_U)^{K_M}$ via $\lambda$.  It is then clear that the map
$$\Pi^K \rightarrow (\Pi_U)^{K_M}$$
is $T_{\lambda}$-equivariant.  It thus suffices to show that every element of the kernel
of this map is killed by a power of $T_{\lambda}$, and that, for every element $x$
of $(\Pi_U)^{K_M}$, $\lambda^m x$ is in the image of this map for some sufficiently large $m$.

For the first claim, let $e_{K^+}$ be the idempotent projector
onto the $K^+$ invariants of a $K$-module.  As $K = K^- K_M K^+$, we have
$e_K = e_{K^+} e_{K_M} e_{K^-}$.  For each $m$, we have
$$e_K \lambda^m e_K = e_{K^+} e_{K_M} e_{K^-} \lambda^m e_K$$
$$= e_{K^+} \lambda^m e_{K_M} e_{\lambda^{-m} K^- \lambda^m} e_K =
\lambda^m e_{\lambda^{-m} K^+ \lambda^m} e_K.$$
(Here we have used that $\lambda$ is positive.)
Now if $\tx$ is an element of $\Pi^K$ that maps to zero in $\Pi_U$, then $e_K \tx = \tx$, and
there exists a compact open subgroup $U_1$ of $U$ such that $e_{U_1} \tx = 0$.  But as
$\lambda$ is strictly positive, there exists an $m$ such that $\lambda^{-m} K^+ \lambda^m$
contains $U_1$.  Thus $\tx$ is killed by $e_{\lambda^{-m} K^+ \lambda^m}$ and fixed by
$e_K$.  Then $\tx$ is killed by $e_K \lambda^m e_K$, as required.

As for the second claim, let $\tx$ be a lift of $x$ to $\Pi^{K_M}$.  There is a compact
open subgroup $K'$ of $G$ that fixes $\tx$; then $\tx$ is in particular invariant under
$K' \cap U^{\circ}$.  As $\lambda$ is strictly positive, there exists an $m$
such that $\lambda^m \tx$ is invariant under $K^-$.  Thus 
$$e_{K^+} \lambda^m \tx = e_{K^+} e_{K_M} e_{K^-} \lambda^m \tx = e_K \lambda^m \tx,$$
so $e_{K^+} \lambda^m \tx$ lies in $\Pi^K$.  As $e_{K^+}$ acts trivially on $\Pi_U$ (because
$K^+$ is contained in $U$), $e_{K^+} \lambda^m \tx$ maps to $\lambda^m \tx$ under the map
$$\Pi^K \rightarrow (\Pi_U)^{K_M},$$ as required.

Finally, if $\Pi$ is $K,P$-stable, then $\Pi^K$ surjects onto $\Pi^K \otimes_{W(k)[T_{\lambda}}
W(k)[T_{\lambda},T_{\lambda}^{-1}]$; this surjection identifies
$(\Pi_U)^{K_M}$ with the maximal $T_{\lambda}$-divisible submodule of $\Pi^K$ and thus
yields the asserted direct sum decomposition.  To see that this decomposition is independent of
$\lambda$, choose another strictly positive element $\lambda'$.  Then the maximal
$T_{\lambda\lambda'}$-divisible submodule of $\Pi^K$ is contained in both the maximal
$T_{\lambda}$-divisible submodule and the maximal $T_{\lambda'}$-divisible submodule;
since projection onto $(\Pi_U)^{\K_M}$ is an isomorphism on each of these submodules they must
all coincide.
\end{proof}

\begin{lemma} Let $R$ be a commutative Noetherian $W(k)$-algebra, and let $\Pi$ be an admissible
$R[G]$-module.  Suppose that $r_G^P \Pi$ is also admissible.  Then $\Pi$ is $K,P$-stable.
\end{lemma}
\begin{proof}
As $r_G^P \Pi$ is a twist of $\Pi_U$, the hypotheses imply that $\Pi^K$ and
$(\Pi_U)^{K_M}$ are finitely generated $R$-modules.  On the other hand,
we have an isomorphism 
$$(\Pi_U)^{K_M} \cong \Pi^K \otimes_{R[T_{\lambda}]} R[T_{\lambda},T^{-1}_{\lambda}].$$
The result is now an immediate consequence of~\cite{BR}, Lemma 33.
\end{proof}

\begin{corollary} For any Levi subgroup $L$ of $G$, and any irreducible cuspidal representation
$\pi$ of $L$ over $k$, 
the modules $\CP_{(M,\pi)}$ and $I_{(M,\pi)}$
are $K,P$-stable.  
\end{corollary}
\begin{proof}
For $\CP_{(M,\pi)}$, the result follows from the previous lemma, together with 
Theorem~\ref{thm:admissibility}.  We also have the result for $\CP_{[M,\pi]}$, as this
is a finite direct sum of modules $\CP_{[M',\pi']}$.
As $I_{(M,\pi)} = \CP_{(M,\pi^{\vee})}^{\vee}$, we also have the result
for $I_{(M,\pi)}$, and hence also for $I_{[M,\pi]}$.
\end{proof}

\begin{proposition} For any Levi subgroup $L$ of $G$, and any irreducible supercuspidal
$k$-representation $\pi$ of $L$, 
every object of $\Rep_{W(k)}(G)_{[L,\pi]}$ is $K,P$-stable.
\end{proposition}
\begin{proof}
Any object of $\Rep_{W(k)}(G)_{[L,\pi]}$ has a resolution by direct sums of $I_{[L,\pi]}$.
These are $K,P$-stable, so the result follows from the fact that kernels of maps of
$K,P$-stable modules are $K,P$-stable.
\end{proof}

It follows that for any object $\Pi$ of $\Rep_{W(k)}(G)_{[L,\pi]}$, the map
$\Pi^K \rightarrow (\Pi_U)^{K_M}$ is surjective.

\begin{proposition}
Let $\Pi$ be an object of $\Rep_{W(k)}(G)_{[L,\pi]}$.  There is a canonical isomorphism:
$$r_G^P \Pi^{\vee} \cong (r_G^{P^{\circ}} \Pi)^{\vee}.$$
\end{proposition}
\begin{proof}
We follow the proof of~\cite{localization}, Theorem 2.  First note that
$r_G^P \Pi^{\vee}$ is a twist of $(\Pi^{\vee})_U$ by our (fixed) square
root of the modulus character of $P$,
and $r_G^{P^{\circ}} \Pi$ is a twist of $\Pi_{U^{\circ}}$ by a square root of the modulus character of
$P^{\circ}$.  As these two modulus characters are inverses of each other, it suffices
to construct an isomorphism $(\Pi_{U^{\circ}})^{\vee} \cong (\Pi^{\vee})_U$ of $W(k)[M]$-modules.

For each $K$ that is decomposed with respect to $P$, we have an isomorphism:
$$(\Pi^K)^{\vee} \rightarrow (\Pi^{\vee})^K$$
coming from a perfect pairing $\Pi^K \times (\Pi^{\vee})^K \rightarrow \CK/W(k)$.
Under this pairing, the adjoint of $T_{\lambda}$ is $T_{\lambda^{-1}}$.  If
we take $\lambda$ to be strictly positive with respect to $P$, then
$\lambda^{-1}$ is strictly positive with respect to $P^{\circ}$.   Moreover,
$\Pi$ is both $(K,P)$-stable and $(K,P^{\circ})$-stable.  In particular,
$(\Pi_{U^{\circ}})^{K_M}$ is isomorphic to the maximal $T_{\lambda^{-1}}$-divisible
submodule of $\Pi^K$, and $((\Pi^{\vee})_U)^{K_M}$ is isomorphic to the maximal
$T_{\lambda}$-divisible submodule of $(\Pi^{\vee})^K$.  Moreover, under these identifications
$(\Pi_{U^{\circ}})^{K_M}$ and $((\Pi^{\vee})_U)^{K_M}$ are direct summands
of $\Pi^K$ and $(\Pi^{\vee})^K$, respectively.  The pairing on the latter thus descends
to a perfect pairing:
$$(\Pi_{U^{\circ}})^{K_M} \times ((\Pi^{\vee})_U)^{K_M} \rightarrow \CK/W(k).$$
By~\cite{localization}, Lemma 1, we can find a $K$ decomposed with respect to $P$ inside
any compact open subgroup of $G$.  Taking the limit over a cofinal system of such $K$
gives the desired perfect pairing
$$\Pi_{U^{\circ}} \times (\Pi^{\vee})_U \rightarrow \CK/W(k),$$
and hence the desired identification of $(\Pi^{\vee})_U$ with $(\Pi_{U^{\circ}})^{\vee}$.
\end{proof}

\begin{theorem}[Bernstein's second adjointness]
Let $\Pi_1$ and $\Pi_2$ be objects of $\Rep_{W(k)}(M)$ and
$\Rep_{W(k)}(G)_{[L,\pi]}$, respectively.  Then there is a canonical isomorphism:
$$\Hom_{W(k)[G]}(i_P^G \Pi_1, \Pi_2) \cong \Hom_{W(k)[M]}(\Pi_1, r_G^{P^{\circ}} \Pi_2).$$
\end{theorem}
\begin{proof}
We follow the argument of the "claim" after Theorem 20 of~\cite{BR}.  We first establish
the case in which $\Pi_2 = (\Pi_2')^{\vee}$ for some $\Pi_2'$ in $\Rep_{W(k)}(G)_{[L,\pi^{\vee}]}.$
We then have a sequence of functorial isomorphisms:
\begin{eqnarray*}
\Hom_{W(k)[M]}(\Pi_1, r_G^{P^{\circ}} \Pi_2) & \cong & \Hom_{W(k)[M]}(\Pi_1, (r_G^P \Pi_2')^{\vee}) \\
& \cong & \Hom_{W(k)[M]}(r_G^P \Pi_2', \Pi_1^{\vee})\\
& \cong & \Hom_{W(k)[G]}(\Pi_2', i_P^G \Pi_1^{\vee})\\
& \cong & \Hom_{W(k)[G]}(\Pi_2', (i_P^G \Pi_1)^{\vee})\\
& \cong & \Hom_{W(k)[G]}(i_P^G \Pi_1, \Pi_2)
\end{eqnarray*}
In particular the result holds for $\Pi_2 = I_{[L,\pi]}$.  If $\Pi_1$ is finitely generated,
then the functors $\Hom_{W(k)[G]}(i_P^G \Pi_1, -)$ and $\Hom_{W(k)[M]}(\Pi_1, r_G^{P^{\circ}} -)$
commute with arbitrary direct sums, so the result holds for $\Pi_1$ finitely generated
and $\Pi_2$ an arbitrary direct sum of copies of $I_{[L,\pi]}$.  As any $\Pi_1$ is the limit of
its finitely generated submodules, the result holds for an arbitrary $\Pi_1$, when $\Pi_2$ is
an arbitrary direct sum of copies of $I_{[L,\pi]}$.  Finally, we can resolve an arbitrary $\Pi_2$
by direct sums of copies of $I_{[L,\pi]}$, and the result then follows for all $\Pi_1,\Pi_2$.
\end{proof}

\begin{corollary}
The representations $\CP_{[L,\pi]}$ are projective and small.
\end{corollary}
\begin{proof}
It suffices to show that each $\CP_{(M,\pi')}$, with $\pi'$ irreducible and cuspidal over
$K$ is projective and small, as $\CP_{[L,\pi]}$ is a finite direct sum of representations of this
form.  For a suitable sequence of types $(K_i,\tau_i)$, we have
$$\CP_{(M,\pi')} = i_P^G \CP_{K_1,\tau_1} \otimes \dots \otimes \CP_{K_r,\tau_r}.$$
For each $i$ the representation
$\CP_{K_i,\tau_i}$ is projective and small (as $\CP_{K_i,\tau_i}$ is the compact induction
of a finite-length $W(k)$-module).  Thus the tensor product $\Pi$ of the $\CP_{K_i,\tau_i}$
is projective and small when considered as a $W(k)[M]$-module.
We have shown that
$\CP_{(M,\pi')}$ lies in $\Rep_{W(k)}(G)_{[L,\pi]}$.

Fix a smooth representation $\Pi'$ of $G$, and let $\Pi'_{[L,\pi]}$ be the direct summand
of $\Pi'$ that lies in $\Rep_{W(k)}(G)_{[L,\pi]}$.  Then
\begin{eqnarray*}
\Hom_{W(k)[G]}(\CP_{(M,\pi')},\Pi') & \cong & \Hom_{W(k)[G]}(\CP_{(M,\pi')},\Pi'_{[L,\pi]})\\
& \cong & \Hom_{W(k)[M]}(\Pi, r_G^{P^{\circ}} \Pi'_{[L,\pi]})
\end{eqnarray*}
As $r_G^{P^{\circ}}$ commutes with direct sums it is easy to see this implies $\CP_{(M,\pi')}$
is small; projectivity of $\CP_{(M,\pi')}$ follows from exactness of $r_G^{P^{\circ}}$.
\end{proof}
 
\begin{corollary} \label{cor:center}
The Bernstein center $A_{[L,\pi]}$ of $\Rep_{W(k)}(G)_{[L,\pi]}$ is isomorphic
to the center of $\End_{W(k)[G]}(\CP_{[L,\pi]}).$
\end{corollary}
\begin{proof}
Every simple object of $\Rep_{W(k)}(G)_{[L,\pi]}$ is a quotient of $\CP_{(M,\pi')}$
for some pair $(M,\pi')$ with inertial supercuspidal support $(L,\pi)$,
and hence such an object is a quotient of $\CP_{[L,\pi]}$.  It follows that $\CP_{[L,\pi]}$ is faithfully
projective in $\Rep_{W(k)}(G)_{[L,\pi]}$, and the result follows immediately.
\end{proof}

\section{Structure of the Bernstein Center} \label{sec:hecke}

In this section we prove basic facts about
the center $A_{[L,\pi]}$ of $\End_{W(k)[G]}(\CP_{[L,\pi]})$.

\begin{proposition} \label{prop:invert}
There is a natural isomorphism:
$$A_{[L,\pi]} \otimes \overline{\CK} \cong \prod_{(\tM,\tpi)} A_{\tM,\tpi},$$
where $(\tM,\tpi)$ runs over inertial equivalence classes of pairs in which
$\tM$ is a Levi subgroup of $G$ and $\tpi$ is a cuspidal representation of $M$
over $\overline{\CK}$ whose mod $\ell$ inertial supercuspidal support equals $(L,\pi)$.
This isomorphism is uniquely characterised by the property that for any $\Pi$
in $\Rep_{\overline{\CK}}(G)$, and any $x$ in $A_{[L,\pi]}$, the action of $x$
on $\Pi$ coincides with that of its image in $\prod_{(\tM,\tpi)} A_{\tM,\tpi}$.
\end{proposition}
\begin{proof}
The module $\CP$ defined by $\CP = \cInd_{\{e\}}^G W(k)$, where $\{e\}$ is the trivial subgroup of $G$,
is a faithfully projective module in $\Rep_{W(k)}(G)$.  We have 
$\CP \otimes \overline{\CK} = \cInd_{\{e\}}^G \overline{\CK}$; in particular
$\CP \otimes \overline{\CK}$ is faithfully projective in $\Rep_{\overline{\CK}}(G)$.
As $\CP$ is $\ell$-torsion free, we have an injection:
$$\End_{W(k)[G]}(\CP) \rightarrow \End_{W(k)[G]}(\CP) \otimes \overline{\CK} \cong 
\End_{\overline{\CK}[G]}(\CP \otimes \overline{\CK}).$$
In particular we have an isomorphism:
$$Z(\End_{W(k)[G]}(\CP)) \otimes \overline{\CK} \cong Z(\End_{\overline{\CK}[G]}(\CP \otimes \overline{\CK})).$$
Multiplying both sides by the central idempotent $e_{[L,\pi]}$ gives us the desired isomorphism.
This isomorphism is equivariant for the actions of both sides on $\CP \otimes \overline{\CK}$, and hence
for all objects of $\Rep_{\overline{\CK}}(G)$.
\end{proof}

As $\CP_{[L,\pi]}$ is, by definition, a faithful $A_{[L,\pi]}$-module, we have injections:
$$A_{[L,\pi]} \otimes \overline{\CK} \hookrightarrow \End_{\overline{\CK}[G]}(\CP_{[L,\pi]} \otimes \overline{\CK})$$
$$\End_{W(k)[G]}(\CP_{[L,\pi]}) \hookrightarrow \End_{\overline{\CK}[G]}(\CP_{[L,\pi]} \otimes \overline{\CK}),$$
and it is clear that $A_{[L,\pi]}$ is the intersection, inside
$\End_{\overline{\CK}[G]}(\CP_{[L,\pi]} \otimes \overline{\CK})$, of these two subalgebras.  Thus we can identify $A_{[L,\pi]}$
with the set of elements $x$ in $\prod_{(\tM,\tpi)} A_{\tM,\tpi}$ such that the action of $x$
on $\CP_{[L,\pi]} \otimes \overline{\CK}$ preserves $\CP_{[L,\pi]}$.

We now study the interaction between $A_{[L,\pi]}$ and parabolic induction.  We first show:

\begin{lemma} \label{lemma:saturation}
Let $P$ be a parabolic subgroup of $G$, with Levi subgroup $M$,
and let $\Pi$ be a smooth $W(k)[M]$-module that is $\ell$-torsion free.  Then the
natural map:
$$\End_{W(k)[M]}(\Pi) \rightarrow \End_{W(k)[M]}(i_P^G \Pi)$$
is injective.  Moreover, if $f$
is an element of $\End_{W(k)[G]}(i_P^G \Pi)$ such that $\ell^a f$
is in the image of this map for some $a$, then $f$ is also in the image of this map.
\end{lemma}
\begin{proof}
Let $g$ be an element of $\End_{W(k)[M]}(\Pi)$.  Then $g$ fits in a commutative
diagram:
$$
\begin{array}{ccc}
r_G^P i_P^G \Pi & \rightarrow & \Pi\\
\downarrow & & \downarrow\\
r_G^P i_P^G \Pi & \rightarrow & \Pi
\end{array}
$$
in which the horizontal arrows are the natural maps (and are therefore surjective),
and the vertical arrows are $r_G^P i_P^G g$ and $g$, respectively.  The surjectivity of
the horizontal maps shows we can recover $g$ from $r_G^P i_P^G g$, proving the first claim.
As for the second, suppose that $\ell^a f = i_P^G g$ for some $g$.  Then the 
the image of $r_G^P i_P^G g$ is contained in $\ell^a r_G^P i_P^G \Pi$ and the above diagram
shows that the image of $g$ is contained in $\ell^a \Pi$.  As $\Pi$ is $\ell$-torsion free
it follows that $g$ is (uniquely) divisible by $\ell^a$ in $\End_{W(k)[M]}(\Pi)$;
that is, $g = \ell^a g'$.  Then we have $f = i_P^G g'$.
\end{proof}

Now let $L'$ be a Levi subgroup of $G$ containing $L$, such that $L'$ is a product of
factors $L'_i$ isomorphic to $\GL_{n_i}(F)$.  Then if $L_i$ is the Levi subgroup $L \cap L'_i$ of $L'_i$,
we have $L = \prod_i L_i$ and 
$\pi$ factors as a tensor product of supercuspidal representations $\pi_i$ of $L_i$.
In this situation we have:

\begin{theorem} \label{thm:integral Bernstein parabolic induction}
Suppose $V$ lies in $\Rep_{W(k)}(M)_{[L,\pi]}$, and let $P$ be a parabolic subgroup of
$G$ with Levi component $M$.  Then $i_P^G V$ lies in $\Rep_{W(k)}(G)_{[L,\pi]}$.  Moreover,
there is a unique map: $f: A_{[L,\pi]} \rightarrow \bigotimes_i A_{[L_i,\pi_i]}$ such that,
for all $V$ in $\Rep_{W(k)}(M)_{[L,\pi]}$, the diagram:
$$
\begin{array}{ccc}
A_{[L,\pi]} & \rightarrow & \bigotimes_i A_{[L_i,\pi_i]}\\
\downarrow & & \downarrow\\
\End_{W(k)[G]}(i_p^G V) & \leftarrow & \End_{W(k)[M]}(V)
\end{array}
$$
commutes.  Moreover, $f$ is independent of $P$.
\end{theorem}
\begin{proof}
Let $V$ bet the faithfully projective $W(k)[M]$-module $\bigotimes_i \CP_{[L_i,\pi_i]}$.
Then $V$is a direct sum of representations of the form $\bigotimes_i \CP_{M'_i,\pi'_i}$,
where for each $i$, $M'_i$ is a Levi subgroup of $M_i$ and $\pi'_i$ is an irreducible
cuspidal representation of $M'_i$ over $k$ with inertial supercuspidal support $(L_i,\pi_i)$.

Fix a collection $(M'_i,\pi'_i)$, let $M'$ be the product of the $M'_i$ and let $\pi'$
be the tensor product of the $\pi'_i$.  Then $i_P^G \bigotimes_i \CP_{M'_i,\pi'_i}$
is isomorphic to projective representation $\CP_{M',\pi'}$, and $(M',\pi')$ has inertial
supercuspidal support $(L,\pi)$.  In particular $i_P^G \bigotimes_i \CP_{M'_i,\pi'_i}$
lies in $\Rep_{W(k)}(G)_{[L,\pi]}$.  Since $i_P^G V$ is a direct sum of these projectives,
it too lies in $\Rep_{W(k)}(G)_{[L,\pi]}$.

On the other hand, any object of $\Rep_{W(k)}(M)_{[L,\pi]}$ has a resolution by
direct sums of copies of $V$, so its parabolic induction has a resolution by
direct sums of copies of $i_P^G V$ and hence lies in $\Rep_{W(k)}(G)_{[L,\pi]}$.

We now turn to the second claim.  We first construct the map $f$ over $\overline{\CK}$.
In this case we have product decompositions:
$$A_{[L,\pi]} \otimes \overline{\CK} \cong \prod_{(\tM,\tpi)} A_{M,\tpi}$$
$$\bigotimes_i A_{[L_i,\pi_i]} \otimes \overline{\CK} \cong \bigotimes_i \prod_{(\tM_i,\tpi_i)} A_{\tM_i,\tpi_i},$$
where $\tM,\tpi$ runs over pairs with mod $\ell$ inertial supercuspidal support $[L,\pi]$
and $\tM_i,\tpi_i$ runs over pairs with mod $\ell$ inertial supercuspidal support $[L_i,\pi_i]$.

Define $f \otimes \overline{\CK}$ on $A_{\tM,\tpi} \rightarrow \bigotimes_i A_{\tM_i,\tpi_i}$ to be
the map $\Phi$ described immediately before Proposition~\ref{prop:Bernstein induction}
if $(\prod_i \tM_i, \otimes_i \tpi_i)$ is inertially equivalent to $(M,\tpi)$, and the zero
map otherwise.  It is clear from Proposition~\ref{prop:Bernstein induction} that the resulting
$f \otimes \overline{\CK}$ has the desired property with respect to parabolic induction.

Since $W(k)$ contains a square root of $q$, parabolic induction is compatible with twisting
by elements of $\Gal(\overline{\CK}/K)$, and from this we see that $f \otimes \overline{\CK}$
must be $\Gal(\overline{\CK}/K)$-equivariant.  In particular $f \otimes \overline{\CK}$
descends to a map $f \otimes \CK$ from $A_{[L,\pi]} \otimes \CK$ to 
$\bigotimes_i A_{[L_i,\pi_i]} \otimes \CK$.  This latter map also has the desired property with
respect to parabolic induction.

Now let $x$ be an element of $A_{[L,\pi]}$, and let $y$ be the image
of $x$ under $f \otimes \CK$.  Then $y$ lies in $\bigotimes_i A_{[L_i,\pi_i]} \otimes \CK$,
and so there exists a power $\ell^a$ of $\ell$ such that $\ell^a y$ lies in
$\bigotimes_i A_{[L_i,\pi_i]}$.  Now $x$ acts on $i_P^G V$, and the action of
$\ell^a x$ on $i_P^G V$ coincides with that of $\ell^a y$.  It follows by Lemma~\ref{lemma:saturation}
that there is an element $y'$ of $\End_{W(k)[M]}(V)$ such that the action of $x$
on $i_P^G$ coincides with that of $y'$.  Since $\ell^a y$ and $\ell^a y'$ agree on $i_P^G V$,
and $V$ is $\ell$-torsion free, we must have $y' = y$.  Thus $y$ lies in the intersection
of $\bigotimes_i A_{[L_i,\pi_i]} \otimes \CK$ with $\End_{W(k)[M]}(V)$, and we have seen that
this intersection is precisely $\bigotimes_i A_{[L_i,\pi_i]}$.

We have thus constructed the desired $f$, and shown that the given diagram commutes
for our particular choice of $V$.  But every object of $\Rep_{W(k)}(M)_{[L,\pi]}$
has a resolution by direct sums of copies of $V$, so the diagram commutes for an
arbitrary object of this category.

Finally, since $f \otimes \overline{\CK}$ does not depend on $P$, neither does $f$.
\end{proof}

In one particularly important case the map arising from parabolic restriction is
an isomorphism.  Call a pair $(L,\pi)$ {\em simple} if $L$ can be written
as a product of copies of $\GL_m(F)$ for some fixed $m$, and $\pi$
is inertially equivalent to a tensor product $(\pi')^{\frac{n}{m}}$ for some
fixed supercuspidal representation $\pi'$ of $\GL_m(F)$.  Say two simple types
$(L_1,\pi_1)$ and $(L_2,\pi_2)$ {\em overlap} if there exist positive integers $m,j_1,j_2$, and
a supercuspdial representation $\pi'$ of $\GL_m(F)$, such that 
$L_1$ and $L_2$ are isomorphic to $\GL_m(F)^{j_1}$ and $\GL_m(F)^{j_2}$, respectively,
$\pi_1$ and $\pi_2$ are inertially equivalent to $(\pi')^{\otimes j_1}$ and $(\pi')^{\otimes j_2}$,
respectively.  If $(L_1,\pi_2)$ and $(L_2,\pi_2)$ do not overlap, we say they are {\em disjoint}.

Given an arbtrary pair $(L,\pi)$, there is, up to inertial equivalence,
a unique collection $\{(L_i,\pi_i)\}$ of disjoint simple pairs such that $L$ is the product of the $L_i$, and
$\pi$ is the tensor product of the $\pi_i$.
We call this collection the {\em canonical factorization
of $(L,\pi)$ into simple types.} 

Given such a factorization, there is a unique Levi subgroup $M$ of $G$ containing $L$
such that $M$ is a product of general linear groups $M_i$ with $M_i \cap L = L_i$ for all $i$.
We denote by $\Rep_{W(k)}(M)_{\{[L_i,\pi_i]\}}$ the block of $\Rep_{W(k)}(M)$ whose simple
objects $V$ have the form $\otimes_i V_i$ with $V_i$ a simple object of $\Rep_{W(k)}(M_i)_{[L_i,\pi_i]}$
for all $i$.  The center $A_{\{[L_i,\pi_i]\}}$ of $\Rep_{W(k)}(M)_{\{[L_i,\pi_i]\}}$ is isomorphic
to $\bigotimes_i A_{[L_i,\pi_i]}$.
Let $P$ be a parabolic subgroup of $G$ with Levi $M$.  Then parabolic induction $i_P^G$ 
takes $\Rep_{W(k)}(M)_{\{[L_i,\pi_i]\}}$ to $\Rep_{W(k)}(G)_{[L,\pi]}$.

\begin{thm} \label{thm:tensor factorization}
Let $\{(L_i,\pi_i)\}$ be the factorization of $(L,\pi)$ into disjoint simple pairs.
Then the functor 
$$i_P^G: \Rep_{W(k)}(M)_{\{[L_i,\pi_i]\}} \rightarrow \Rep_{W(k)}(G)_{[L,\pi]}$$
is an equivalence of categories.  In particular, the map
map: $f: A_{[L,\pi]} \rightarrow \bigotimes_i A_{[L_i,\pi_i]}$
arising from Theorem~\ref{thm:integral Bernstein parabolic induction} is an isomorphism.
\end{thm}
\begin{proof}

Let $(M',\pi')$ be a pair, with $M'$ a Levi containing $L$ and $\pi'$ a cuspidal representation of $M'$
with inertial supercuspidal support $(L,\pi)$.  Then there is a unique collection $\{(M'_i,\pi'_i)\}$
such that $M'$ is the product of the $M'_i$, $\pi'$ is the tensor product of the $\pi'_i$, and
each $\pi'_i$ has inertial supercuspidal support $(L_i,\pi_i)$.  From this it follows that
$\CP_{M',\pi'}$ is isomorphic to $i_P^G \bigotimes_i \CP_{M'_i,\pi'_i}$.  Taking direct sums,
we find that $\CP_{[L,\pi]}$ is isomorphic to $i_P^G \bigotimes_i \CP_{[L_i,\pi_i]}$.

Let $e_{\{[L_i,\pi_i]\}}$ be the idempotent corresponding to the block
$\Rep_{W(k)}(M)_{\{[L_i,\pi_i]\}}$ of $\Rep_{W(k)}(M)$.  From Bernstein-Zelevinsky's
filtration of the composite functor $r_G^P i_P^G$, we see that the natural map
$$r_G^P i_P^G \bigotimes_i \CP_{[L_i,\pi_i]} \rightarrow \bigotimes_i \CP_{[L_i,\pi_i]}$$
is a surjection, whose kernel is annihilated by $e_{\{[L_i,\pi_i]\}}$.  We thus have
natural isomorphisms:
$$\bigotimes_i \CP_{[L_i,\pi_i]} \cong e_{\{[L_i,\pi_i]\}} r_G^P i_P^G \bigotimes_i \CP_{[L_i,\pi_i]}$$
$$\CP_{[L,\pi]} \cong i_P^G e_{\{[L_i,\pi_i]\}} r_G^P \CP_{[L,\pi]}.$$

From this one deduces that the pair of functors $i_P^G$ and $e_{\{[L_i,\pi_i]\}} r_G^P$ are inverses.
For instance, any object $V$ of $\Rep_{W(k)}(M)_{\{[L_i,\pi_i]\}}$ admits a resolution by direct sums of
copies of $\bigotimes_i \CP_{[L_i,\pi_i]}$.  Applying the functor $e_{\{[L_i,\pi_i]\}} r_G^P i_P^G$
to this resolution yields a complex isomorphic to the original one via the natural isomorphism 
described above.  The remainder of the argument is similar.

The functor $i_P^G$ thus identifies $A_{[L,\pi]}$ with $\bigotimes_i A_{[L_i,\pi_i]}$;
since this isomorphism is induced by parabolic induction it agrees with the map $f$
of Theorem~\ref{thm:integral Bernstein parabolic induction}.
\end{proof}

We now turn to the question of constructing natural elements of $A_{[L,\pi]}.$
We will do this first in the case where the pair $(L,\pi)$ is simple.  In this
case, up to inertial equivalence, $L$ is isomorphic to $\GL_{n_1}(F)^r$ for some
$n_1$ and $r$, and $\pi$ is a tensor power $\pi_1^{\otimes r}$ of
an irreducible, supercuspidal $k$-representation $\pi_1$ of $\GL_{n_1}(F)$.  Then,
as in section~\ref{sec:endomorphisms}, we may choose a maximal distinguished supercuspidal
$k$-type $(K_1,\tau_1)$ contained in $\pi_1$.  Attached to this type is an integer $f$
and, for each $m$ in $\{1, e_{q^f}, \ell e_{q^f}, \ell^2 e_{q^f}, \dots\}$, a maximal distinguished
cuspidal type $(K_m,\tau_m)$; the representation $\tau_m$ factors as usual as $\kappa_m \otimes \sigma_m$.

In this case, by definition, the projective $\CP_{[L,\pi]}$ is given by:
$$\CP_{[L,\pi]} \cong \bigoplus_{\nu} i_{P_{\nu}}^G \bigotimes_j \CP_{K_{\nu_j},\tau_{\nu_j}}$$
where $\nu$ runs over all increasing partitions $\nu_1 \leq \nu_2 \leq \nu_3 \dots$ of $r$ 
such that $\nu_j$ lies in $\{1, e_{q^f}, \ell e_{q^f}, \dots\}$ for all $j$, and $P_{\nu}$
is the standard (upper triangular) parabolic with block sizes corresponding to $\nu$.

Fix a $\nu$ as above, and let $t$ be its length.
For $1 \leq i \leq r$, we can define an endomorphism 
$\Theta_{i,\nu}$ of $i_{P_{\nu}}^G \otimes_j \CP_{L_{\nu_j},\tau_{\nu_j}}$
by setting
$$\Theta_{i,\nu} = \Sigma_{\mu} \bigotimes_j \Theta_{\mu_j,\nu_j},$$
where $\mu$ runs over tuples $(\mu_1, \dots, \mu_s)$ of nonnegative integers such that
$\mu_j \leq \nu_j$ for all $j$, and the sum of the $\mu_j$ is $i$.  (Here
$\Theta_{\mu_j,\nu_j}$ is the endomorphism of $\CP_{K_{\nu_j},\tau_{\nu_j}}$ constructed
in Theorem~\ref{thm:compatibility}; the tensor product of these endomorphisms acts
by functoriality of parabolic induction on 
$i_{P_{\nu}}^G \bigotimes_j \CP_{K_{\nu_j},\tau_{\nu_j}}$.)

Let $\Theta_i$ be the endomorphism of $\CP_{[L,\pi]}$ that acts via $\Theta_{i,\nu}$
on the summand
$i_{P_{\nu}}^G \bigotimes_j \CP_{K_{\nu_j},\tau_{\nu_j}}$ of $\CP_{[L,\pi]}$.

To understand the action of $\Theta_i$ on $\CP_{[L,\pi]}$ we invoke results of
sections~\ref{sec:endomorphisms}.  In particular, we fix a
lift $\tpi_1$ of $\pi_1$ containing the type $\tkappa_1 \otimes \St_{s_1}$,
where $s_1$ is the $\ell$-regular conjugacy class of $\GL_{\frac{n_1}{ef}}(\FF_{q^f})$
corresponding to $\sigma_1$.  This choice gives rise to a compatible system of cuspidals
$(M_s,\pi_s)$ as in section~\ref{sec:endomorphisms}; here $s$ varies over conjugacy
classes in $\GL_{\frac{n_1i}{ef}}(\FF_{q^f})$, with $\ell$-regular part $(s_1)^i$,
for $i \in \{1, e_{q^f}, \ell e_{q^f}, \dots \}$.
As in the discussion before Theorem~\ref{thm:compatibility}, we have an isomorphism:
$$A_{M_s,\pi_s} \cong \overline{\CK}[Z_s]^{W_{M_s}(s)}.$$

For any $m$, the endomorphism ring $E_{K_m,\tau_m} \otimes \overline{\CK}$
of $\CP_{K_m,\tau_m} \otimes \overline{\CK}$
factors as the product:
$$\prod_s \overline{\CK}[Z_s]^{W_{M_s}(s)}.$$
By Theorem~\ref{thm:compatibility}, the endomorphism $\Theta_{i,m}$ of $\CP_{K_m,\tau_m}$
maps to the element $\prod_s \theta_{i,s}$ under this identification, where $\theta_{i,s}$
is the element of $\overline{\CK}[Z_s]$ constructed in section~\ref{sec:endomorphisms}.

Fix a $\nu$ of length $t$.  We then have an isomorphism:
$$\bigotimes_j E_{K_{\nu_j},\tau_{\nu_j}} \otimes \overline{\CK} \cong \prod_{\vec{s}}
\bigotimes_j \overline{\CK}[Z_{s_j}]^{W_{M_{s_j}}(s_j)},$$
where $\vec{s} = s_1, \dots, s_{t}$ is a sequence such that $s_j$ is a
conjugacy class in $\GL_{\frac{n_1\nu_j}{ef}}(\FF_{q^f})$ with $\ell$-regular part $(s_1)^{\nu_j}$.
If $s$ is the semisimple conjugacy class in $\GL_{\frac{n_1r}{ef}}(\FF_{q^f})$ corresponding to
``block diagonal'' matrix with blocks $s_1, \dots, s_t$, then $\overline{\CK}[Z_s]$ is the
tensor product of the $\overline{\CK}[S_j]$.  Moreover, under the composition
$$\bigotimes_j E_{K_{\nu_j},\tau_{\nu_j}} \rightarrow \bigotimes_j \overline{\CK}[Z_{s_j}]^{W_{M_{s_j}}(s_j)}
\hookrightarrow \overline{\CK}[Z_s],$$
the element $\Theta_{i,\nu}$ of $\bigotimes_j E_{K_{\nu_j},\tau_{\nu_j}}$ maps to the element $\theta_{i,s}$.

In particular, action of $\Theta_{i,\nu}$ on the part of
$i_{P_{\nu}}^G \bigotimes_j \CP_{K_{\nu_j},\tau_{nu_j}} \otimes \overline{\CK}$ lying in the block $(M_s,\pi_s)$
is via the element of $A_{M_s,\pi_s}$ corresponding to $\theta_{i,s}$ under the isomorphism:
$$A_{M_s,\pi_s} \cong \overline{\CK}[Z_s]^{W_{M_s}(s)}.$$

We immediately deduce:

\begin{theorem} \label{thm:theta}
The elements $\Theta_1, \dots, \Theta_r$ lie in $A_{[L,\pi]}$.
\end{theorem}
\begin{proof}
The action of $\Theta_i$ on $\CP_{[L,\pi]}$ coincides with the element
$\prod_{M_s,\pi_s} \theta_{i,s}$ of the product $\prod_{M_s,\pi_s} A_{M_s,\pi_s}$; in particular it lies
in the center of the endomorphism ring of $\CP_{[L,\pi]}$.
\end{proof}

We let $C_{[L,\pi]}$ be the $W(k)$-subalgebra of $A_{[L,\pi]}$ generated by
the elements $\Theta_i$, $i \in 1, \dots, r$, together with $\Theta_r^{-1}$.
More generally, if $[L,\pi]$ is not simple but factors into the simple types $[L_i,\pi_i]$,
we let $C_{[L,\pi]}$ be the tensor product of the algebras $C_{[L_i,\pi_i]}$, regarded as a subalgebra
of $A_{[L,\pi]}$ via the isomorphism $A_{[L,\pi]} \cong \bigotimes_i A_{[L_i,\pi_i]}$ constructed above.

When $(L,\pi)$ is simple, it is not hard to show that the elements $\Theta_1, \dots, \Theta_r$
are algebraically independent.  Indeed, let $s$ be the semisimple conjugacy class in $\GL_{\frac{n}{ef}}(\FF_{q^f})$
given by $(s_1)^r$.  Then $Z_s$ is free abelian group on $r$ generators, and
$W_{M_s}(s)$ is the symmetric group that permutes them.  The elements $\theta_{1,s}, \dots, \theta_{r,s}$
of $W(k)[Z_s]$ are the elementary symmetric functions in these generators.  In particular they are algebraically
independent (and generate $W(k)[Z_s]^{W_{M_s}(s)}$ over $W(k)$.)

Since $\Theta_i$ maps to $\theta_{i,s}$ under the composition:
$$C_{[L,\pi]} \hookrightarrow A_{[L,\pi]} \rightarrow A_{M_s,\pi_s} \cong \overline{\CK}[Z_s]^{W_{M_s}(s)},$$ 
we have:
\begin{proposition}
Let $s$ be the semisimple conjugacy class in $\GL_{\frac{n_1r}{ef}}(\FF_{q^f})$ given by $(s_1)^{r}$.
Then the map $C_{[L,\pi]} \rightarrow W(k)[Z_s]^{W_{M_s}(s)}$ is an isomorphism.
\end{proposition}

\begin{proposition}
The module $\CP_{[L,\pi]}$ is an admissible $C_{[L,\pi]}[G]$-module (and hence is also admissible
as an $A_{[L,\pi]}[G]$-module.)
\end{proposition}
\begin{proof}
It suffices to show this in the case where $[L,\pi]$ is simple, and in this case it
suffices to show that each $i_{P_{\nu}}^G \bigotimes_j \CP_{K_{\nu_j},\tau_{\nu_j}}$
is admissible over $C_{[L,\pi]}$.  The module $\bigotimes_j \CP_{K_{\nu_j},\tau_{\nu_j}}$
is admissible over $\bigotimes_j C_{K_{\nu_j},\tau_{\nu_j}}$.

It thus suffices to show that, for each $\nu$, the map $C_{[L,\pi]} \rightarrow \bigotimes_j C_{K_{\nu_j},\tau_{\nu_j}}$
makes the latter into a finitely generated $C_{[L,\pi]}$-module.  We have a commutative diagram:
$$
\begin{array}{ccc}
C_{[L,\pi]} & \rightarrow & \bigotimes_j C_{K_{\nu_j},\tau_{\nu_j}}\\
\downarrow & & \downarrow\\
W(k)[Z_s]^{W_{M_s}(s)} & \rightarrow & \bigotimes_j W(k)[Z_{s_j}]^{W_{M_{s_j}}(s_j)}
\end{array}
$$
where $s$ is the semisimple conjugacy class in $\GL_{\frac{n_1r}{ef}}(\FF_{q^f})$ given by $(s_1)^{r}$
and $s_j$ is the semisimple conjugacy class in $\GL_{\frac{n_1\nu_j}{ef}}(\FF_{q^f})$ given by $(s_1)^{\nu_j}$.
The vertical maps are isomorphisms.  In particular we have inclusions
$$C_{[L,\pi]} \subseteq \bigotimes_j C_{K_{\mu_j},\tau_{\nu_j}} \subseteq W(k)[Z_s].$$
But $W(k)[Z_s]$ is finitely generated as a $C_{[L,\pi]}$-module, so the result follows.
\end{proof}

We thus have:

\begin{theorem} \label{thm:fg}
The ring $A_{[L,\pi]}$ is a finitely generated, reduced,
$\ell$-torsion free $W(k)$-algebra.
\end{theorem}
\begin{proof}
As $\CP_{[L,\pi]}$ is admissible over $C_{[L,\pi]}$ its endomorphism ring is
finitely generated as a $C_{[L,\pi]}$-module.  Thus
$A_{[L,\pi]}$ is a finitely generated $C_{[L,\pi]}$-module.  Since $C_{[L,\pi]}$ is
a polynomial ring over $W(k)$, it is immediate that $A_{[L,\pi]}$ is a finitely
generated $W(k)$-algebra.  The fact that $A_{[L,\pi]}$ is $\ell$-torsion free
follows from $A_{[L,\pi]} \subset \End_{W(k)[G]}(\CP_{[L,\pi]})$ and the fact
that $\CP_{[L,\pi]}$ is $\ell$-torsion free.  Reducedness follows from the
fact that $A_{[L,\pi]}$ embeds in $A_{[L,\pi]} \otimes_{W(k)} \overline{\CK}$, and the latter is
reduced by Proposition~\ref{prop:invert}.
\end{proof}

\begin{theorem} \label{thm:fg adm}
Let $\Pi$ be an object of $\Rep_{W(k)}(G)_{[L,\pi]}$ that is
finitely generated as a $W(k)[G]$-module.  Then $\Pi$ is an admissible
$A_{L,\pi}[G]$-module.
\end{theorem}
\begin{proof}
As $\CP_{[L,\pi]}$ is faithfully projective, there is a surjection of a direct sum
of (possibly infinitely many) copies of $\CP_{[L,\pi]}$ onto $\Pi$.  Any element
$x$ of $\Pi$ is in the image of this surjection, and any element $\tx$ that maps to
$x$ is nonzero only in finitely many copies of $\CP_{[L,\pi]}$.  Thus if $\Pi$ is
generated by finitely many elements, then there is a finite direct sum of copies
of $\CP_{[L,\pi]}$ whose image in $\Pi$ contains all the generators, and this direct
sum surjects onto $\Pi$.
\end{proof}

We next turn to the question of understanding the $k$-points of $\Spec A_{[L,\pi]}$.
If $\Pi$ is an irreducible object in $\Rep_{k}(G)_{[L,\pi]}$, then it admits an action
of $A_{[L,\pi]}$; by Schur's lemma this action is via a map $f_{\Pi}: A_{[L,\pi]} \rightarrow k$.
Conversely, if $f: A_{[L,\pi]} \rightarrow k$ is any map, then $\CP_{[L,\pi]} \otimes_{A_{[L,\pi]},f} k$
is an admissible smooth $k[G]$-module on which $A_{[L,\pi]}$ acts via $f$, so there exists an
irreducible $\Pi$ such that $f = f_{\Pi}$.

There is thus a surjection $\Pi \mapsto f_{\Pi}$
from isomorphism classes of irreducible objects in $\Rep_k(G)_{[L,\pi]}$
to $k$-points of $\Spec A_{[L,\pi]}$.  

\begin{lemma}
Suppose $\Pi$ and $\Pi'$ are irreducible $k$-representations in $\Rep_k(G)_{[L,\pi]}$
with the same supercuspidal support.  Then $f_{\Pi} = f_{\Pi'}$.
\end{lemma}
\begin{proof}
Suppose that $\Pi$ and $\Pi'$ both have supercuspidal support $(L,\pi')$ for some $\pi'$
inertially equivalent to $\pi$.  Let $\tpi'$ be a lift of $\pi'$ to $\overline{\CK}$.  Then
both $\Pi$ and $\Pi'$ are subquotients of $i_P^G \tpi$.  On the other hand the action
of $A_{[L,\pi]}$ on $i_P^G \tpi$ factors through $A_{[L,\pi]} \otimes \overline{\CK}$, and
the latter acts on $i_P^G \tpi$ by scalars.  It follows that any element of $A_{[L,\pi]}$ acts 
on $\Pi$ and $\Pi'$ by the same scalar; that is, $f_{\Pi} = f_{\Pi'}$.
\end{proof}

Conversely, we have:
\begin{prop} \label{prop:c action}
Let $\Pi$ and $\Pi'$ be irreducible $k$-representations in $\Rep_k(G)_{[L,\pi]}$ and
suppose that $f_{\Pi}$ and $f_{\Pi'}$ agree on $C_{[L,\pi]}$.  Then
$\Pi$ and $\Pi'$ have the same supercuspidal support.
\end{prop}
\begin{proof}
It suffices to prove this in the case where $(L,\pi)$ is simple. 
The supercuspidal supports of $\Pi$ and $\Pi'$ have the form $(L,\pi \otimes \chi)$
and $(L,\pi \otimes \chi')$, respectively, where $\chi$ and $\chi'$ are unramified characters of
$L$.  

Choose a lift $\tchi$ of $\chi$ to a character $L/L_0 \rightarrow W(k)^{\times}$,
and consider the pair $(L,\tpi \otimes \tchi)$, where $\tpi$
is the representation $\tpi_1^{\otimes r}$ of $L$.

If $P$ is a parabolic with Levi $L$, then $A_{[L,\pi]}$ acts on $i_P^G \tpi \otimes \tchi$
via the map 
$$A_{[L,\pi]} \rightarrow A_{M_s,\pi_s} \cong (\overline{\CK}[(L/L_0)]^{H_s})^{W_{M_s}(s)}
\cong \overline{\CK}[Z_s]^{W_{M_s}(s)},$$
where $s$ is the semisimple conjugacy class $(s_1)^r$ in $\GL_{\frac{n_1r}{ef}}(\FF_{q^f})$.
The action of an element of $A_{[L,\pi]}$ on $i_P^G \tpi \otimes \tchi$ may be described as follows:
the character $\tchi$ can be viewed as a $W(k)$-point of $\Spec \overline{\CK}[L/L_0]$;
then an element $x$ of $(\overline{\CK}[(L/L_0)]^{H_s})^{W_{M_s}(s)}$ acts on $i_P^G \tpi \otimes \tchi$
by the scalar $x(\tchi)$, where we regard $x$ as a function on $\Spec \overline{\CK}[L/L_0]$.

Reducing mod $\ell$, we find that an element of $A_{[L,\pi]}$ that maps to a given element
$x$ of $(W(k)[(L/L_0)]^{H_s})^{W_{M_s}(s)}$
acts on $i_P^G \pi \otimes \chi$ (and hence on $\Pi$) by the scalar $x(\chi)$.  Similarly, such an element
acts by $x(\chi')$ on $\Pi'$.

Since $f_{\Pi}$ and $f_{\Pi'}$ agree on $C_{[L,\pi]}$, we have $\theta_{i,s}(\chi) = \theta_{i,s}(\chi')$
for all $i$.  Thus, when considered as $k$-points of $\Spec W(k)[L/L_0]$, the characters $\chi$
and $\chi'$ map to the same $k$-point of $(\Spec W(k)[(L/L_0)]^{H_s})^{W_{M_s}(s)}$.  In particular
the pairs $(L,\pi \otimes \chi)$ and $(L,\pi \otimes \chi')$ are conjugate, so $\Pi$ and $\Pi'$ have
the same supercuspidal support as claimed.
\end{proof}

Combining these, we have:

\begin{corollary} \label{cor:points}
If $\Pi$ and $\Pi'$ are irreducible representations of $G$ over $k$ that
lie in $\Rep_{W(k)}(G)_{[L,\pi]}$, 
and $f_{\Pi}, f_{\Pi'}$ are the maps $A_{[L,\pi]} \rightarrow k$ giving the action
of $A_{L,\pi}$ on $\Pi$ and $\Pi'$ respectively, then $f_{\Pi} = f_{\Pi'}$ if, and only if,
$\Pi$ and $\Pi'$ have the same supercuspidal supports.  Moreover, the map
$\Spec A_{[L,\pi]} \otimes k \rightarrow \Spec C_{[L,\pi]} \otimes k$ is a bijection on $k$-points.
\end{corollary}

\begin{corollary} \label{cor:mod ell reduced}
The composition:
$$C_{[L,\pi]} \otimes k \rightarrow A_{[L,\pi]} \otimes k \rightarrow (A_{[L,\pi]} \otimes k)^{\red}$$
identifies $C_{[L,\pi]} \otimes k$ with the reduced quotient of $A_{[L,\pi]} \otimes k$.
\end{corollary}
\begin{proof}
The map $A_{[L,\pi]} \rightarrow \overline{\CK}[Z_s]^{W_{M_s}(s)}$ factors through $W(k)[Z_s]^{W_{M_s}(s)}$,
where $s$ is the semisimple conjugacy class in $\GL_{\frac{n_1r}{ef}}(\FF_{q^f})$ given by $(s_1)^{r}$,
and the composition:
$$C_{[L,\pi]} \otimes k \rightarrow A_{[L,\pi]} \otimes k \rightarrow k[Z_s]^{W_{M_s}(s)}$$
is an isomorphism.  We thus obtain a map $(A_{[L,\pi]} \otimes k)^{\red} \rightarrow C_{[L,\pi]} \otimes k$
that is left inverse to the map $C_{[L,\pi]} \otimes k \rightarrow (A_{[L,\pi]} \otimes k)^{\red}$.  We must
show that it is also a right inverse.  For this it suffices to show that the map
$(A_{[L,\pi]} \otimes k)^{\red} \rightarrow C_{[L,\pi]} \otimes k$ is injective.
But both maps are bijections on $k$-points.  Since
$\Spec C_{[L,\pi]}$ is integral, so is $\Spec (A_{[L,\pi]} \otimes k)^{\red}$;
the map $\Spec C_{[L,\pi]} \otimes k \rightarrow \Spec (A_{[L,\pi]} \otimes k)^{\red}$ is dominant
and so the induced map on rings of functions must be injective.
\end{proof}

\section{Further results} \label{sec:explicit}
In this section we establish more precise results on the structure of $A_{[L,\pi]}$;
we focus on the case in which the pair $(L,\pi)$ is simple.  
Fix an irreducible supercuspidal representation $\pi_1$ of $\GL_n(F)$, and for each positive
integer $m$, let $G_m$ be the group $\GL_{nm}(F)$, let $L_m$ be a Levi of $G_m$ isomorphic to a product
of $m$ copies of $\GL_n(F)$, and let $\pi_m$ be the representation $(\pi_1)^{\otimes m}$ of $L_m$,
so that $(L_m,\pi_m)$ is a simple pair. 

All of our results in this section
rely on the following result, which is proved in~\cite{bernstein3}:

\begin{thm}[\cite{bernstein3}, Proposition 4.9] \label{thm:generic1}
Let $(K,\tau)$ be a maximal distinguished cuspidal $k$-type of $G_m$, and let ${\mathfrak p}$ be
a prime ideal of $E_{K,\tau}$ with residue field $\kappa({\mathfrak p})$.  Then the cosocle of the
tensor product $\CP_{K,\tau} \otimes_{E_{K,\tau}} \kappa({\mathfrak p})$ is absolutely irreducible
and generic.
\end{thm}

Suppose the inertial supercuspidal support of $(K,\tau)$ is $(L_m,\pi_m)$, so that there
is an action of $A_{L_m,\pi_m}$ on $\CP_{K,\tau}$ and a corresponding injection
$A_{L_m,\pi_m} \rightarrow E_{K,\tau}$.  Let ${\mathfrak m}$ be a maximal ideal of
$A_{L_m,\pi_m}$, with residue field $k$.  We then have:

\begin{lemma} \label{lem:filtration}
Every irreducble quotient of $\CP_{K,\tau}/{\mathfrak m} \CP_{K,\tau}$ is generic.
\end{lemma}
\begin{proof}
The quotient $E_{K,\tau}/{\mathfrak m} E_{K,\tau}$ is a finite dimensional $k$-algebra,
and is thus a product of finite dimensional local $k$-algebras $E_i$, with maximal ideals
${\mathfrak m_i}$.  We then have an isomorphism:
$$\CP_{K,\tau}/{\mathfrak m} \CP_{K,\tau} \cong \bigoplus_i \CP_{K,\tau} \otimes_{E_{K,\tau}} E_i$$
and it suffices to shat any irreducible quotient of the right-hand-side is generic.

Fix $i$.  Then $E_i$ admits a filtration by ideals
$$0 = I_0 \subseteq I_1 \subseteq \dots \subseteq I_r = E_i$$
such that the quotients $I_{j+1}/I_j$ are one-dimensional $k$-vector spaces.  We can thus choose,
for each $j$, an element $x_j$ of $I_j$ such that $x_j$ generates $I_j$ modulo $I_{j-1}$.
Then multiplication by $x_j$ induces an surjection of $\CP_{K,\tau} \otimes_{E_{K,\tau}} \kappa({\mathfrak m_i})$
onto $I_j \CP_{K,\tau}/ I_{j-1} \CP_{K,\tau}$.  It follows that $\CP_{K,\tau} \otimes_{E_{K,\tau}} E_i$
admits a filtration such that the cosocle of each successive quotient is irreducible and generic.
Thus every irreducible quotient of $\CP_{K,\tau} \otimes_{E_{K,\tau}} E_i$ is generic as claimed.
\end{proof}

We will
say a simple object $\Pi$ of $\Rep_{W(k)}(G_m)_{[L_m,\pi_m]}$ is generic if either:
\begin{itemize}
\item $\ell \Pi = 0$, and $\Pi$ is an irreducible generic $k$-representation of $G_m$, or
\item $\ell \Pi = \Pi$, and $\Pi \otimes_{\CK} \overline{\CK}$ is a direct sum of irreducible generic
$\overline{\CK}$-representations of $G_m$.
\end{itemize}

We can then rephrase Theorem~\ref{thm:generic1} as follows:

\begin{thm} \label{thm:generic2}
Let $(K,\tau)$ be a maximal disinguished cuspidal $k$-type of $G_m$, with supercuspidal support $(L_m,\pi_m)$.  
Then every simple quotient of $\CP_{K,\tau}$ is generic.  Moreover, every simple generic object of
$\Rep_{W(k)}(G)_{[L_m,\pi_m]}$ is a quotient of $\CP_{K,\tau}$.
\end{thm}
\begin{proof}
Let $\Pi$ be a simple quotient of $\CP_{K,\tau}$.  Then $\End_{W(k)[G]}(\Pi)$ is either isomorphic to $k$,
if $\ell \Pi = 0$, or to an extension $\CK'$ of $\CK$ (note that every endomorphism of $\Pi$ lifts to
an endomorphism of $\CP_{K,\tau}$, so that $\Pi$ has a commutative endomorphism ring.)  

Suppose first that $\End_{W(k)[G]}(\Pi)$ is an extension $\CK'$ of $\CK$.  Then,
since $A_{L_m,\pi_m}$ acts on $\Pi$, and the natural map $A_{L_m,\pi_m} \rightarrow E_{K,\tau}$
is an isomorphism after inverting $\ell$, we have a map $E_{K_m,\tau_m} \rightarrow \End_{W(k)[G]}(\Pi)$.
This map is surjective because every endomorphism of $\Pi$ lifts to an element of $E_{K_m,\tau_m}$, and
hence its kernel is a maximal ideal ${\mathfrak m}$.  Then 
$\Pi$ then a quotient of 
$\CP_{K,\tau} \otimes_{E_{K,\tau}} \kappa({\mathfrak m})$.  Since $\Pi$ is simple,
the cosocle of $\CP_{K,\tau} \otimes_{E_{K,\tau}} \kappa({\mathfrak m)}$ surjects onto
$\Pi$.  It follows that, when regarded as a representation over $\kappa({\mathfrak m})$,
$\Pi$ is absolutely irreducible and generic, so that $\Pi$ is a generic $W(k)[G]$-module as claimed.

On the other hand, if $\ell \Pi = 0$, then the annihilator of $\Pi$ in $A_{L_m,\pi_m}$ is a maximal ideal
${\mathfrak m}$ with residue field $k$, and $\Pi$ is an irreducible quotient of $\CP_{K,\tau}/{\mathfrak m} \CP_{K,\tau}$,
and is therefore generic by Lemma~\ref{lem:filtration}.

Conversely, let $\Pi$ be a simple generic object in $\Rep_{W(k)}(G)_{[L_m,\pi_m]}$, and set $\kappa = k$
if $\ell \Pi = 0$, $\kappa = \overline{\CK}$ otherwise.  Then $\Pi \otimes \kappa$ is 
a direct sum of irreducible representations; let $\Pi'$ be an irreducible summand.  The action of $A_{[L_m,\pi_m]}$
on $\Pi'$ is via a map $A_{[L_m,\pi_m]} \rightarrow \kappa$.  This map depends on the choice of $\Pi'$,
but its kernel ${\mathfrak m}$ does not.

Let ${\tilde {\mathfrak m}}$ be a maximal ideal of 
$E_{K,\tau}$ lying over ${\mathfrak m}$.  Since $\kappa$ is algebraically closed, and $E_{K,\tau}$
is finite over $A_{[L_m,\pi_m]}$, the map $A_{[L_m,\pi_m]} \rightarrow \kappa$ giving the action
of $A_{[L_m,\pi_m]}$ on $\Pi'$ extends to a map $f: E_{K,\tau}/{\tilde {\mathfrak m}} \rightarrow \kappa$.

The cosocle of $\Pi''$ of $\CP_{K,\tau} \otimes_{E_{K,\tau},f} \kappa$ is an irreducible generic representation
on which $A_{[L_m,\pi_m]}$ acts via $f$; in particular the supercuspidal supports of $\Pi''$ and $\Pi'$
coincide.  Since there is a unique generic representation with given supercuspidal support, $\Pi''$ is isomorphic
to $\Pi'$; we thus obtain a surjection $\CP_{K,\tau} \otimes_{E_{K,\tau},f} \kappa \rightarrow \Pi'$.
In particular $\Hom_{W(k)[G]}(\CP_{K,\tau} \otimes_{E_{K,\tau},f} \kappa, \Pi \otimes \kappa)$ is nonzero
and so $\Hom_{W(k)[G]}(\CP_{K,\tau},\Pi)$ is nonzero as well.
\end{proof}

We will apply this result in the setting of section~\ref{sec:endomorphisms}.  
Fix a maximal distinguished supercuspidal
$k$-type $(K_1,\tau_1)$ containing $\pi_1$;
then the construction of section~\ref{sec:endomorphisms}
gives rise to a family of maximal distinguished cuspidal $k$-types $(K_m,\tau_m)$ for $m$
in $\{1,e_{q^f},\ell e_{q^f}, \dots \}$.  Fix an element $m > 1$ in this set, and let $m'$
be the largest element of this set strictly smaller than $m$.  Set $j = \frac{m}{m'}$.
We then have projectives $\CP_m = \CP_{K_m,\tau_m}$ and $\CP_{m'} = \CP_{K_{m'},\tau_{m'}}$; we denote their
endomorphism rings by $E_m$ and $E_{m'}$, respectively.  

Let $M$ be a Levi of $G_m$ containing $L_m$ and isomorphic to a product of $j$ copies of $\GL_{nm'}(F)$,
and let $P$ be a parabolic with Levi $M$.  Theorem~\ref{thm:generic2} gives us the following relationship
between the $W(k)[M]$-module $\CP_{m'}^{\otimes j}$ and the parabolic restriction $r_G^P \CP_m$:

\begin{prop}
The map: 
$$r_G^P \CP_m \otimes \Hom_{W(k)[M]}(r_G^P \CP_m, \CP_{m'}^{\otimes j}) \rightarrow \CP_{m'}^{\otimes j}$$
is surjective.
\end{prop}
\begin{proof}
Note that $r_P^G \CP_m$ is a projective $W(k)[M]$-module.  It thus suffices to show that every simple
quotient of $\CP_{m'}^{\otimes j}$ admits a map from $r_P^G \CP_m$.  By Theorem~\ref{thm:generic2},
any simple quotient $\Pi$ $\CP_{m'}^{\otimes j}$ is generic and lies in $\Rep_{W(k)}(M)_{[L_m,\pi_m]}$.
Thus $i_P^G \Pi$ lies in $\Rep_{W(k)}(G)_{[L_m,\pi_m]}$ and has a simple generic subquotient.  In particular,
$\Hom_{W(k)[G]}(\CP_m, i_P^G \Pi)$ is nonzero, and so $\Hom_{W(k)[M]}(r_P^G \CP_m, \Pi)$ is nonzero as well.
\end{proof}

As a consequence, we deduce:
\begin{theorem}
There exists a unique map: $f_m: E_m \rightarrow E_{m'}^{\otimes j}$ making the diagram:
$$
\begin{array}{ccc}
A_{[L_m,\pi_m]} & \rightarrow & A_{[L_{m'},\pi_{m'}]}^{\otimes j}\\
\downarrow & & \downarrow\\
E_m & \rightarrow & E_{m'}^{\otimes j}
\end{array}
$$
commute.
\end{theorem}
\begin{proof}
Our descriptions of $A_{[L_m,\pi_m]} \otimes \overline{\CK}$ and $E_m \otimes \overline{\CK}$ show that
the natural map $A_{[L_m,\pi_m]} \rightarrow E_m$ giving the action of $A_{[L_m,\pi_m]}$ on $\CP_m$
becomes an isomorphism after tensoring with $\overline{\CK}$, and hence after inverting $\ell$.
Uniqueness of the map is clear thus if such a map exists.  

To construct the map, note that $E_m$
acts on $\CP_m$ and hence on $r_G^P \CP_m$.  Moreover, if $x$ is an element of $E_m$
in the image of $A_{[L_m,\pi_m]}$, then $x$ acts on $r_G^P \CP_m$ via the natural map
$A_{[L_m,\pi_m]} \rightarrow A_{[L_{m'},\pi_{m'}]}^{\otimes j}$, by the adjointness of parabolic induction and
restriction.  It thus suffices to show that the action
of $E_m$ preserves the kernel of the surjection:
$$r_G^P \CP_m \otimes \Hom_{W(k)[M]}(r_G^P \CP_m, \CP_{m'}^{\otimes j}) \rightarrow \CP_{m'}^{\otimes j}.$$

On the other hand, if $x$ lies in $E_m$, then for some $a$, $\ell^a x$ lies in $A_{[L_m,\pi_m]}$,
and hence acts via an element of $A_{[L_{m'},\pi_{m'}]}^{\otimes j}$ on $r_G^P \CP_m$.  In particular
$\ell^a x$ preserves the kernel of the map:
$$\left[r_G^P \CP_m \otimes \Hom_{W(k)[M]}(r_G^P \CP_m, \CP_{m'}^{\otimes j})\right]
\rightarrow \CP_{m'}^{\otimes j}$$
and so $x$ preserves this kernel as well (as the kernel contains all $\ell$-torsion).  
Thus $x$ gives an endomorphism of $\CP_{m'}^{\otimes j}$ and we are done.
\end{proof}

We have the following control over the image of $E_m$ in $E_{m'}^{\otimes j}$:
\begin{theorem}
Let $y$ be an element of $E_{m'}^{\otimes j}$ such that, for some $a$, $\ell^a y$ lies in the image of $E_m$.
Then there exist elements $x_1,x_2$ of $E_m$, and a positive integer $b$,
such that:
\begin{enumerate}
\item $x_2$ is supported on double cosets of the form
$K_m z_{m,m}^i K_m$, and
\item $f_m(x_2) = \ell^b(y - f_m(x_1))$.
\end{enumerate}
\end{theorem}
\begin{proof}
We invoke results from section~\ref{sec:endomorphisms}.  In particular, by Proposition~\ref{prop:P and P'}
we have a support-preserving map
$$E_m/I^{\cusp} \hookrightarrow H(G_m,K_m,\tkappa_m \otimes \CP'_{\sigma_m})$$
that is an isomorphism after inverting $\ell$ and whose cokernel is supported on double cosets
of the form $K_m z_{m,m}^i K_m$.

The $G$-cover construction of section~\ref{sec:endomorphisms}
yields a map:
$$T: E_{m'}^{\otimes j} \rightarrow H(G_m,K_m,\tkappa_m \otimes i_{\overline{P}}^{\overline{G}_m} \CP_{\overline{M}})
$$
and Lemma~\ref{lemma:G cover center} shows that central elements of $H(G_m,K_m,\tkappa_m \otimes
i_{\overline{P}}^{\overline{G}_m} \CP_{\overline{M}})$ descend to elements of $H(G_m,K_m,\tkappa_m \otimes \CP'_{\sigma_m})$.

If $\ell^a y$ is in the image of $E_m$, then for some $c \geq a$, $\ell^c y$ lies in the
image of $A_{[L_m,\pi_m]}$.  In particular $T(\ell^c y)$ coincides with an element of $A_{[L_m,\pi_m]}$
as an endomorphism of $\cInd_{K_m}^{G_m} \tkappa_m \otimes i_{\overline{P}}^{\overline{G}_m} \CP_{\overline{M}}$
and is thus central.  In particular $T y$ is central as well.

We thus obtain an element $y'$ of $H(G_m,K_m,\tkappa_m \otimes \CP'_{\sigma_m})$ corresponding to $T y$.
Let $y'_2$ be the part of $y'$ supported on double cosets of the form $K_m z_{m,m}^i K_m$, and set $y'_1 = y - y'_2$.
Then $y'_1$ is the image of an element $x_1$ of $E_m$.  Moreover, for some $b$, $\ell^b y'_2$ is the image
of an element $x_2$ of $E_m$ supported on double cosets of the form $K_m z_{m,m}^i K_m$.

The action of $T f_m(x_2)$ on $\cInd_{K_m}^{G_m} \tkappa_m \otimes \CP'_{\sigma_m}$ then coincides with
that of $\ell^b y'_2$, and that of $T f_m(x_1)$ coincides with $y'_1$.  We thus have
$T f_m(x_2) = \ell^b(T y - T f_m(x_1))$, and since $T$ is injective the claim follows.
\end{proof}

The existence of maps $f_m$ as above allows us to considerably simplify our description of $A_{[L_m,\pi_m]}$.  From
this point on we let $m$ be arbitrary.  Call a partition $\nu$ of $m$ {\em admissible} if each $\nu_j$ lies
in the set $\{1,e_{q^f},\ell e_{q^f}, \dots \}$.  We order the admissible partitions by refinement, via the convention
$\nu \preceq \nu'$ if $\nu$ refines $\nu'$.  Since each element of the set $\{1,e_{q^f},\ell e_{q^f}, \dots \}$ divides
the next, there is a unique {\em maximal} (i.e. coarsest) partition $\nu$ of $m$ for every $m$.

For any admissible partition $\nu$ of $m$ we have a map 
$$A_{[L_m,\pi_m]} \rightarrow \bigotimes_j A_{L_{\nu_j},\pi_{\nu_j}} \rightarrow \bigotimes_j E_{\nu_j}.$$

\begin{theorem}
Let $\nu$ be the maximal admissible partition of $m$.  Then the map $A_{[L_m,\pi_m]} \rightarrow \bigotimes_j E_{\nu_j}$ is
injective, and has saturated image (i.e., if $\ell^a x$ lies in the image for some $a$, then so does $x$.)
In particular, if $m$ lies in $\{1, e_{q^f}, \ell e_{q^f}, \dots \}$ then the map $A_{[L_m,\pi_m]} \rightarrow E_m$ is
an isomorphism.
\end{theorem}
\begin{proof}
We have seen that
$A_{[L_m,\pi_m]}$ is the intersection, in $\End_{W(k)[G_m]}(\CP_{[L_m,\pi_m]} \otimes \CK)$,
of $A_{[L_m,\pi_m]} \otimes \CK$ and $\End_{W(k)[G_m]}(\CP_{[L_m,\pi_m]})$.

Let $x$ be an element of $\bigotimes_j E_{\nu_j}$ such that $\ell^a x$ is the image of $y \in A_{[L_m,\pi_m]}$.
For any admissible partition $\nu'$ of $m$, $\nu'$ refines $\nu$.  We thus have a map:
$$\bigotimes_j E_{\nu_j} \rightarrow \bigotimes_i E_{\nu'_i}$$
obtained by iteratively applying maps of the form $f_{\nu_j}$ to the tensor factors.  For each $\nu'$,
let $x_{\nu'}$ be the image of $x$ in $\bigotimes_i E_{\nu'_i}$.  Then $x_{\nu'}$ gives an
endomorphism of $i_{P_{\nu'}}^{G_m} \bigotimes_i \CP_{\nu'_i}$, and the action of $\ell^a x_{\nu'}$ on this
space coincides with that of $y$.

Now $\CP_{[L_m,\pi_m]}$ is isomorphic to the direct sum of the spaces $i_{P_{\nu'}}^{G_m} \bigotimes_i \CP_{\nu'_i}$
as $\nu'$ varies over all admissible partitions.  The endomorphism of $\CP_{[L_m,\pi_m]}$ that acts by
$x_{\nu'}$ on the summand corresponding to $\nu'$ for all $\nu'$ is given by $\frac{y}{\ell^a}$ and thus
lies in $A_{[L_m,\pi_m]}$.
\end{proof}

We conclude by giving a characterization of the image of $A_{[L_m,\pi_m]}$ in $\bigotimes_j E_{\nu_j}$,
where $\nu$ remains the maximal admissible partition of $m$.
We have $\tau_1 = \kappa_1 \otimes \sigma_1$, where $\sigma_1$ is a supercuspidal representation
of $\overline{G}_1$ corresponding to an $\ell$-regular semisimple conjugacy class $s'$ in $\overline{G}_1$.
We have an isomorphism:
$$E_{\nu_j} \otimes \overline{\CK} \rightarrow \prod_s \overline{\CK}[Z_s]^{W_{M_s}(s)},$$
where $s$ runs over semsimple conjugacy classes in $G_{\nu_j}$ with $\ell$-regular part $(s')^{\nu_j}$.
For such an $s$, let $\phi_s: E_{\nu_j} \rightarrow \overline{\CK}[Z_s]^{W_{M_s}(s)}$ be the projection
onto the corresponding factor.

Let $\vec{s}$ be a sequence of conjugacy classes, where $\vec{s}_j$ is a semisimple conjugacy class in $G_{\nu_j}$
with $\ell$-regular part $(s')^{\nu_j}$.  The tensor product of the $\phi_{\vec{s}_j}$ gives a map:
$$\phi_{\vec{s}}: \bigotimes_j E_{\nu_j} \rightarrow \bigotimes_j \overline{\CK}[Z_{\vec{s}_j}]^{W_{M_{\vec{s}_j}}(\vec{s}_j)}.$$
We may also regard $\vec{s}$ as a single conjugacy class in $M_{\nu}$, where $M_{\nu}$ is a Levi subgroup
of $G_m$ with ``block decomposition'' given by $\nu$.

Now suppose that $\vec{s}$ and $\vec{s}'$ are two conjugacy classes in $M_{\nu}$ as above, and let
$w \in W(G_m)$ be an element of $W(G_m)$ that conjugates $\vec{s}$ to $\vec{s}'$.  Then
conjugation by $w$ induces an isomorphism:
$$\prod_j Z_{\vec{s}_j} \cong \prod_j Z_{\vec{s}'_j}$$
and a compatible isomorphism:
$$\prod_j W_{M_{\vec{s}_j}}(\vec{s}_j) \cong \prod_j W_{M_{\vec{s}'_j}}(\vec{s}'_j)$$
where we regard all the products as subgroups of $M_{\nu}$, and hence of $G_m$.  We thus obtain isomorphisms:
$$\psi_w: \bigotimes_j \overline{\CK}[Z_{\vec{s}_j}]^{W_{M_{\vec{s}_j}}(\vec{s}_j)} \cong
\bigotimes_j \overline{\CK}[Z_{\vec{s}'_j}]^{W_{M_{\vec{s}'_j}}(\vec{s}'_j)}.$$

\begin{theorem} \label{thm:explicit}
An element $x$ of $\bigotimes_j E_{\nu_j}$ lies in the image of $A_{[L_m,\pi_m]}$ if, and only if,
for all $\vec{s},\vec{s}'$ as above, and all $w \in W(G_m)$ that conjugate $\vec{s}$ to $\vec{s}'$, we have
$\phi_{\vec{s}}(x) = \psi_w \phi_{\vec{s}'}(x)$.
\end{theorem}
\begin{proof}
It suffices to show that the action of $x$ on $i_{P_{\nu}}^{G_m} \bigotimes_j \CP_{\nu_j} \otimes \overline{\CK}$ coincides
with an element of $A_{[L,\pi]} \otimes \overline{\CK}$.  Elements of the latter have the form
$\prod_{M,\tpi} A_{M,\tpi}$, where $(M,\tpi)$ runs over a set of representatives for the inertial equivalence
classes of pairs over $\overline{\CK}$ with mod $\ell$ inertial supercuspidal support $(L_m,\pi_m)$.

For $\vec{s}$, let $M_{\vec{s}}$ be the product of the $M_{\vec{s}_j}$ and let $\pi_{\vec{s}}$ be the tensor
product of the $\pi_{s_j}$.  Then $(M_{\vec{s}},\pi_{\vec{s}})$ is inertially equivalent to
a unique pair $(M,\tpi)$ as above.  This inertial equivalence induces a map:
$$A_{M,\tpi} \rightarrow \bigotimes_j \overline{\CK}[Z_{\vec{s}'_j}],$$
and the image consists of those elements that are invariant under all $w \in W(G_m)$ that
conjugate $\vec{s}'$ to itself.  

Moreover, if $\vec{s}$ and $\vec{s}'$ are such that $(M_{\vec{s}},\pi_{\vec{s}})$ and
$(M_{\vec{s}'},\pi_{\vec{s}'})$ are both inertially equivalent to $(M,\tpi)$, then there exists a $w \in W(G_m)$
that conjugates $\vec{s}$ to $\vec{s}'$, and then the diagram:
$$
\begin{array}{ccc}
A_{M,\tpi} & \rightarrow & \bigotimes_j \overline{\CK}[Z_{\vec{s}_j}]\\
\downarrow & & \downarrow\\
A_{M,\tpi} & \rightarrow & \bigotimes_j \overline{\CK}[Z_{\vec{s'}_j}]
\end{array}
$$
commutes, where the left vertical arrow is the identity and the right is $\psi_w$.

Thus, if $x$ satisfies the hypothesis of the theorem, it gives rise to a well-defined element
of $A_{[L_m,\pi_m]} \otimes \overline{\CK}$, and the action of this element on 
$i_{P_{\nu}}^G \bigotimes_j \CP_{\nu_j} \otimes \overline{\CK}$
coincides with that of $x$.  Thus $x$ lies in the image of $A_{[L_m,\pi_m]}$ as claimed.
\end{proof}

\begin{example} \label{ex:banal} \rm
Suppose that $e_{q^f} > m$ and so the partition
$\nu$ consists entirely of $1$'s.  (This includes, in particular, any banal situation.)
The ring $E_1$ is the ring $W(k)[\Theta_{1,1}^{\pm 1}]$.  An element $x$
of $E_1^{\otimes m}$ satsifes the conditions of Theorem~\ref{thm:explicit} if, and only if $x$ is invariant
under $S_m$.  Thus in this case $A_{[L,\pi]}$ is simply an algebra of symmetric polynomials; in fact,
one has $C_{[L,\pi]} = A_{[L_m,\pi_m]}$ in this case.
\end{example}

\begin{example} \label{ex:quasi-banal} \rm
Suppose that $\ell > m$, and that $q^f \cong 1$ modulo $\ell$.  Then the partition $\nu$ consists entirely of $1$'s,
and $E_1$ is isomorphic to $W(k)[\Theta_{1,1}^{\pm 1},\zeta]/\zeta^{\ell^a} - 1$, where $\ell^a$ is the largest power of
$\ell$ dividing $q - 1$.  Then, as in the last example,
an element of $x$ in $E_1^{\otimes m}$ satisfies the conditions of Theorem~\ref{thm:explicit}
if, and only if $x$ is $S_m$ invariant.  We thus obtain an explicit description of $A_{[L_m,\pi_m]}$ in
this setting as well.  (Note that this case includes- but is slightly larger than- what is often
called the ``quasi-banal'' setting in the literature.)
\end{example}

\begin{example} \label{ex:more complicated} \rm
Suppose that $e_{q^f} \leq m < 2e_{q^f}$, so that the partition $\nu$ is given by $1, \dots, 1, e_{q^f}$.
The ring $E_1$ is given by $W(k)[\Theta_{1,1}^{\pm 1}]$, and, by Proposition~\ref{prop:e_q}, $E_{e_{q^f}}$ is given by:
$$E_{e_{q^f}} = 
E_{\sigma_{e_{q^f}}}[\Theta_{1,e_{q^f}}, \dots, \Theta_{e_{q^f},e_{q^f}},\Theta_{e_{q^f},e_{q^f}}^{-1}]/I^{\cusp} \cdot
\<\Theta_{1,e_{q^f}}, \dots, \Theta_{e_{q^f}-1,e_{q^f}}\>.$$
(Recall that $I^{\cusp}$ is the kernel of the map $E_{\sigma_{e_{q^f}}} \rightarrow W(k)$ giving the action
of $E_{\sigma_{e_{q^f}}}$ on the unique irreducible, noncuspidal summand of $\CP_{\sigma_{e_{q^f}}} \otimes \overline{\CK}$.)

We regard $E_{\sigma_{e_{q^f}}}$ as a subalgebra of $E_1^{\otimes m-e_{q^f}} \otimes E_{e_{q^f}}$ 
by the inclusion $E_{\sigma_{e_{q^f}}} \rightarrow E_{e_{q^f}}$.  It is then not hard to verify, using
Theorem~\ref{thm:explicit}, that the image of $A_{[L,\pi]}$ in $E_1^{\otimes m-e_{q^f}} \otimes E_{e_{q^f}}$
is generated by $E_{\sigma_{e_{q^f}}}$, together with the images of the elements
$\Theta_1, \dots, \Theta_m,\Theta_m^{-1}$ of $A_{[L,\pi]}$.  (First one proves these elements generate after tensoring
with $\overline{\CK}$, then one verifies that if $\ell x$ is in the subalgebra generated by
these elements, then so is $x$.)

This yields a presentation:
$$A_{[L_m,\pi_m]} \cong E_{\sigma_{e_{q^f}}}[\Theta_1, \dots, \Theta_m,\Theta_m^{-1}]/I^{\cusp} \cdot J, $$
where $J$ is the ideal generated by the elements: $\Theta_j$ for $m - e_{q^f} + 1 \leq j \leq e_{q^f} - 1$
and $\Theta_j - \Theta_{e_{q^f}} \Theta_{j - e_{q^f}}$ for $e_{q^f} < j \leq m$.  As remarked earlier,
Paige's description of the structure of $E_{\sigma_{e_{q^f}}}$~\cite{paige} means that this
presentation is completely explicit.
\end{example}

In general, when $\nu$ is a more complicated partition, the criteria of Theorem~\ref{thm:explicit}
do not realize $A_{[L_m,\pi_m]}$ as a ring of invariants in $\bigotimes_j E_{\nu_j}$.  Obtaining
a clean description of $A_{[L_m,\pi_m]}$ in this setting is substantially more complicated- in particular
it is unlikely that one can find a clean presentation like the ones above.  We refer the reader to~\cite{curtis},
where we will address this question from a rather different perspective.

\end{document}